\documentclass[reqno]{amsart} 
\usepackage[utf8]{inputenc} 
\usepackage{amsthm,amsmath,amssymb,amsfonts} 
\usepackage{tikz} 
\usepackage{graphicx} 
\usepackage{psfrag} 
\usepackage{cancel} 
\usepackage{hyperref} 
\usepackage{comment} 
\usepackage{enumerate} 
\newtheorem{definition}{Definition}[section] 
\newtheorem{theorem}{Theorem}[section] 
\newtheorem{lemma}{Lemma}[section] 
\newtheorem{proposition}{Proposition}[section] 
\newtheorem*{yano}{Yano's Theorem} 
\newtheorem*{Annulus}{{\sc Cat}-Annulus Conjecture} 
\newtheorem*{corollary*}{Corollary} 
\newtheorem{corollary}{Corollary}[section] 
\newtheorem{remark}{Remark}[section] 
 
\newtheorem{claim}{Claim}%[section] 
\numberwithin{equation}{section}

\theoremstyle{plain}
\newtheorem{maintheorem}{Theorem}

\title[]{Genericity of infinite entropy for maps \\ with low regularity}
\thanks{
This work has been supported by 
``Projeto Tem\'atico Din\^amica em Baixas Dimens\~oes'' FAPESP Grant 2011/16265-2, 
FAPESP Grant 2015/17909-7, 
Projeto PVE CNPq 401020/2014-2,
and EU Marie-Curie IRSES Brazilian-European partnership in Dynamical Systems (FP7-PEOPLE-2012-IRSES 318999 BREUDS)
}

\author{Edson de Faria} 
\address{Edson de Faria, Instituto de Matem\'{a}tica e Estat\'{i}stica, USP, S\~{a}o Paulo, SP, Brazil}
\email[]{edson@ime.usp.br}

\author{Peter Hazard} 
\address{Peter Hazard, Instituto de Matem\'{a}tica e Estat{\'i}stica, USP, S\~{a}o Paulo, SP, Brazil}
\email[]{pete@ime.usp.br}

\author{Charles Tresser} 
\address{Charles Tresser, 301 W. 118th Street, \#8B, New York, NY 10026, USA
}
\email[]{tresser.charles@gmail.com}

\date{\today}

\begin{document}

%\paragraph{Comments.}
%The following are some changes that have been made or need making.
%\begin{enumerate}
%\item 
%Check little-o in H\"older Closing Lemma
%[DONE] 
%\item 
%H\"older gluing lemma (version with infinite gluings can be extracted from Claim 1)
%[DONE--DOUBLE CHECK] 
%\item 
%Prove H\"older blow-up (Proposition 5.2)
%\item 
%Rewrite Holder Rescaling principle with Lipschitz, instead of affine, maps
%\end{enumerate}

%\newpage

\begin{abstract}
For bi-Lipschitz homeomorphisms of a compact manifold it is known that topological entropy is always finite.
For compact manifolds of dimension two or greater, we show that in the closure of the space of bi-Lipschitz homeomorphisms, with respect to either 
the H\"older or the Sobolev topologies, topological entropy is generically infinite.
We also prove versions of the $C^1$-Closing Lemma in either of these spaces.
Finally, we give examples of homeomorphisms with infinite topological entropy which are H\"older and/or Sobolev of every exponent. 
\end{abstract}

\maketitle

\section{Introduction}
\subsection{Background.} 
In 1974 Palis, Pugh, Shub and Sullivan~\cite{PPSS74} 
published a list of dynamical properties satisfied by a generic
homeomorphism acting on an arbitrary compact manifold. 
Six years later, Yano~\cite{Yano80} submitted 
an extra striking property in relation to {\it topological entropy\/}.
Recall that 
topological entropy is a non-negative extended real number 
defined for any continuous self-map of a compact space, and that 
this number is an invariant of topological conjugacy. 
It was first introduced by 
Adler, Konheim, and McAndrew~\cite{AKM} as an 
analogue, 
in the topological category, 
of the Kolmogorov-Sinai entropy 
for measure-preserving transformations.
Topological entropy, whose precise definition will be given below, is a useful way of 
quantifying topological aspects of chaos.
Yano proved the following result.

\begin{yano}
For 
homeomorphisms of compact manifolds of dimension greater than one 
topological entropy is generically infinite.
\end{yano}

Here, the space of homeomorphisms is endowed with the uniform topology.
Yano also states the same result in the case of endomorphisms on compact manifolds of dimension one or greater.
However, in this article we will focus on the homeomorphism case.
Note that, in Yano's result, the fact that the space being acted 
upon is a manifold matters:
there are compact metric spaces for which a generic homeomorphism has zero topological entropy~\cite{AGW}.

For smooth maps on compact manifolds it had already been demonstrated, 
by Ito~\cite{Ito70} for homeomorphisms and soon afterwards by Bowen~\cite{Bowen71} for general endomorphisms,
that the topological entropy is always finite.
Thus, in~\cite{Yano80}, Yano also obtained the following result as a consequence.

\vspace{5pt}

\begin{corollary*}
A generic homeomorphism 
 of a compact manifold of dimension two or greater is not topologically conjugate 
 to any diffeomorphism. 
\end{corollary*}

\begin{comment} 
We remark that in~\cite{PPSS74}, the authors of that paper where more interested 
in describing generic properties common to both continuous and smooth maps, 
rather than in finding means of distinguishing them.
This leads to the natural question whether one may hope to find a complete list 
of the properties that are generic (or typical) in categories of continuous maps and not in smooth ones, or the other way around.      
\end{comment}

%%%%%%%%%%%%%%%%%%%%%%%%%%%%%%%%%%%%%%%%%%%%
Let us now recall the definition of topological entropy~\cite{AKM}.
As we will be working only in compact metric spaces, 
we give the reformulation in this setting due to Bowen\footnote{
Michel H\'enon pointed out to one of the authors that 
this definition has a significant advantage over the original definition 
when trying to compute entropy for specific systems:  
only the forward iterates of the map must be considered, 
rather than the backward iterates!
}.
Let $f$ be a continuous self-mapping of a compact metric space $(X,d)$.
A subset $E$ of $X$ is $(n,\epsilon)$-separated for $f$ if for all distinct points $x,y\in E$ there exists
a non-negative integer $k<n$ such that $d(f^k(x),f^k(y))>\epsilon$.
Let $S_f(n,\epsilon)$ denote the maximal cardinality of an $(n,\epsilon)$-separated set.
Then the {\it topological entropy\/} of $f$ is given by
\begin{equation*}
h_{\mathrm{top}}(f)=\lim_{\epsilon\to 0}\limsup_{n\to\infty}\frac{1}{n}\log S_f(n,\epsilon) \ .
\end{equation*}
That is, topological entropy is the growth rate of the maximal size of $(n,\epsilon)$-separated sets at arbitrarily small scales $\epsilon$.
Note that in some applications, one prefers to compute the metric entropy 
$h_{\mu} (f)$ of $f$ with respect to some $f$-invariant measure $\mu$.
However, by virtue of the variational principle~\cite[Theorem 8.6]{WaltersBook}, topological entropy is always the supremum 
(and in some important cases the maximum) of the metric entropies $h_{\mu} (f)$, where $\mu$ varies over all $f$-invariant Borel probability measures.
%%%%%%%%%%%%%%%%%%%%%%%%%%%%%%%%%%%%%%%%%%%%%

As was already stated, for smooth maps on manifolds the topological entropy is finite.
In fact, if $X$ is a compact metric space of Hausdorff 
dimension $\dim_H(X)$ and $f$ is a self-map of $X$ 
with Lipschitz constant $L$, then
\begin{equation*}
h_{\mathrm{top}}(f)\leq \dim_H(X)\cdot \log^+(L) \ .
\end{equation*}
See, for instance,~\cite[Theorem 3.2.9]{KandH}.
We note that, in the case of smooth maps acting on smooth manifolds, 
this bound was already implicit in the work of Ito~\cite{Ito70} and Bowen~\cite{Bowen71}.

The above discussion thus shows that the existence of bounds on $h_{\mathrm{top}}(f)$ 
changes dramatically when the regularity 
goes from just continuity to Lipschitz continuity. 
However, the notion of ``going from'' continuity to 
Lipschitz continuity must be treated with care. 
In this paper we make an initial foray into the 
problem of determining what occurs between these two cases 
by considering mappings in H\"older and 
Sobolev classes. 
These are perhaps two of the most classical ways of 
interpolating between $C^0$- and Lipschitz-regularity.
Homeomorphisms with H\"older or Sobolev regularity have been 
of interest recently in the study of certain PDE's, 
such as the Ball-Evans Problem in nonlinear elasticity 
(cf.~\cite{IKO} and references therein).
However, the study of the dynamics of maps in either 
of these classes has essentially remained untouched.
With this work we hope to remedy this situation while also laying 
the groundwork for further dynamical investigations. 

One possible reason for why the dynamics of maps in these 
spaces has not been considered before is that $\alpha$-H\"older and $W^{1,p}$-Sobolev 
classes are not closed under composition.
This being said, spaces of such maps are closed under pre- and post-composition 
by Lipschitz maps, and the union of such spaces, under $\alpha$ and $p$ 
respectively, {\it is} closed under composition.
This allows us to still make local perturbations of these maps.
Also note that $C^1$, or even Lipschitz, 
is not in general dense in any H\"older or Sobolev class.
Thus results concerning such classes cannot be derived 
from direct approximation arguments.

Below we will show that, in the closure of the space of 
bi-Lipschitz maps in either of these topologies, 
for suitable parameters of regularity,
infinite entropy is a generic property. 
It also follows from our results that there is no ``barrier'' 
separating infinite entropy maps from the 
space of Lipschitz maps (we even give explicit 
examples of homeomorphisms 
with infinite entropy which are H\"older or 
Sobolev of every exponent).

\subsection{Summary of our results.}\label{subsect:summary_of_results}
To topologise the space of H\"older or Sobolev homeomorphisms on a smooth manifold one requires 
(in principle) additional structure:
a distance function in the first case and a Riemannian structure in the second.
We take a different approach by defining topologies 
on function spaces which are analogous to the Whitney topology.
More specifically, for $0\leq \alpha<1$,
let $\mathcal{H}^\alpha(M)$ denote the space of homeomorphisms on $M$ which are bi-$\alpha$-H\"older continuous in all local charts.
We also denote by $\mathcal{H}^1(M)$ the space of homeomorphisms which are bi-Lipschitz in all local charts. 
In Section~\ref{sec:Holder_prelim} we define a topology on $\mathcal{H}^\alpha(M)$ which we call the {\it $\alpha$-H\"older-Whitney topology}.
For $0\leq \alpha<\beta\leq 1$, we denote by $\mathcal{H}_\alpha^\beta(M)$ the closure
of $\mathcal{H}^\beta(M)$ with respect to the $\alpha$-H\"older-Whitney topology.
Recall that a property is {\it generic} in a Baire space if the set of 
points satisfying this property contains a residual subset 
({\it i.e.\/}, a countable intersection of open and dense subsets).
We show the following.
\begin{maintheorem}[Generic Infinite Entropy for H\"older Classes]\label{thm:generic_little-Holder_homeomorphism}
Let $M$ be a smooth compact manifold of dimension $d$ greater than or equal to two.
For $0\leq \alpha< 1$,
the following holds.
In $\mathcal{H}^{1}_{\alpha}(M)$,
infinite topological entropy is a generic property.
\end{maintheorem} 
Similarly, for $1\leq p,p^*<\infty$,
let $\mathcal{S}^{p,p^*}(M)$ denote the space of 
homeomorphisms on $M$ which in all local charts 
are of Sobolev class $W^{1,p}$ and whose inverse 
is of Sobolev class $W^{1,p^*}$.
In Section~\ref{sec:Sobolev_prelim} we define a topology on $\mathcal{S}^{p,p^*}(M)$ which we call the {\it $(p,p^*)$-Sobolev-Whitney topology}. 
\begin{maintheorem}[Generic Infinite Entropy for Sobolev Classes]\label{thm:generic_Sobolev_homeomorphism}
Let $M$ be a smooth compact manifold of dimension $d$.
\begin{itemize}
\item[(a)]
If $d=2$ and $1\leq p,p^*<\infty$
then, in $\mathcal{S}^{p,p^*}(M)$,  
infinite topological entropy is a generic property.
\item[(b)]
If $d>2$ and $d-1<p,p^*<\infty$,
then, in $\mathcal{S}^{p,p^*}(M)$,  
infinite topological entropy is a generic property.
\end{itemize}
\end{maintheorem}
Additionally, we give an alternative proof of (a) in the case when $p^*=1$.
This proof uses a variant of the Rad\'o-Kneser-Choquet 
theorem for $p$-harmonic mappings~\cite{AS,IKO}.
We do not know whether this approach extends to higher dimensions, 
though we suspect not, as there exists a counterexample to the 
classical Rad\'o-Kneser-Choquet theorem in dimension three (see, {\it e.g.\/},~\cite[Section 3.7]{DurenBook}).

Let us also note that we prove two versions of the Closing Lemma along the way.
Namely, for both the spaces of bi-H\"older homeomorphisms and bi-Sobolev homeomorphisms stated above, 
we show that the analogue of Pugh's $C^1$-Closing Lemma holds. 
It would be interesting to determine whether there is 
another, more direct, approach using Pugh's $C^1$-result and an approximation argument, 
demonstrating that a homeomorphism of bi-H\"older or bi-Sobolev type 
is approximable by $C^1$-diffeomorphisms.

\subsection{Structure of the paper.}
In Part I we investigate some properties of bi-H\"older homeomorphisms.
After the preliminary Section~\ref{sec:Holder_prelim}, where a suitable 
H\"older-Whitney topology is given on the space 
of bi-H\"older homeomorphisms between manifolds,
the Closing Lemma for this class of maps is proved in Section~\ref{sec:holderclosing}.
Following this the genericity of infinite topological 
entropy is shown in Section~\ref{sect:Holder-infinite_entropy}.

Part II investigates bi-Sobolev homeomorphisms.
The structure of Part II mirrors that of Part I, 
with the exception that we also give another  
proof of the genericity of infinite entropy in 
the special case of compact surfaces.
Specifically, in Section~\ref{sec:Sobolev_prelim} we introduce the space of bi-Sobolev homeomorphisms together with the Sobolev-Whitney topology.
We prove a Closing Lemma for maps in this class in Section~\ref{sec:Sobolev_Closing}. 
Two different proofs of the genericity of infinite topological entropy, 
one specific to dimension two and another for dimensions two and greater, are given in Section~\ref{sec:Sobolev-infinite_entropy}.

In Appendix A we give explicit examples of homeomorphisms in dimension two 
with infinite topological entropy which lie in {\it all\/} H\"older or Sobolev classes.
These examples can be thought of as certain perturbations of the identity transformation.
Finally, the perturbation tools used throughout the paper are collected in Appendices~\ref{sect:basic_moves} and~\ref{sect:cylinder_perturbations}.

\subsection{Notation and terminology.}    
Throughout this article, we use the following notation.
We denote the Euclidean norm in $\mathbb{R}^d$ by $|\cdot|_{\mathbb{R}^d}$.
We denote the Euclidean distance by $d(\cdot,\cdot)$.
Denote the open $r$-ball about the point $x$ in $\mathbb{R}^d$ by $B^d(x,r)$.
When the dimension is clear we will write this as $B(x,r)$.
In the special case of the unit ball in $\mathbb{R}^d$ centred at the origin we denote this by $B^d$.

Given a manifold $M$ endowed with distance function $d_M(\cdot,\cdot)$
denote the open $r$-ball about $\xi$ in $M$, with respect to $d_M$, by $B_M(\xi,r)$. 
Given points $a,b\in\mathbb{R}^d$ and $r>0$ define
\begin{equation*}
E(a,b;r)=\left\{x\in \mathbb{R}^d : d\left(x,tp+(1-t)q\right)<r, \ \mbox{some} \ t\in[0,1]\right\} 
\end{equation*}
We call such a set an {\it elongated neighbourhood\/}.
Given subsets $\Omega_0$ and $\Omega_1$ in some metric space we denote the Hausdorff distance 
between $\Omega_0$ and $\Omega_1$ by $\mathrm{dist}_H(\Omega_0,\Omega_1)$, {\it i.e.\/},
\begin{equation*}
\mathrm{dist}_H(\Omega_0,\Omega_1)
=
\max\left\{\sup_{x_0\in\Omega_0}\inf_{x_1\in\Omega_1}d(x_0,x_1),\sup_{x_1\in\Omega_1}\inf_{x_0\in\Omega_0}d(x_0,x_1)\right\}
\end{equation*}
and the diameter of $\Omega_0$ by $\mathrm{diam}(\Omega_0)$.

\section{Part I -- H\"older Mappings}
\subsection{Preliminaries.}\label{sec:Holder_prelim}
We recall some basic definitions and facts concerning H\"older maps.
Much of what we state here is classical and proofs are left to the reader.
\subsubsection*{H\"older mappings between metric spaces.}
Let $\Omega$ and $\Omega^*$ be metric spaces.
For each $\alpha\in (0,1)$, 
let 
$C^\alpha(\Omega,\Omega^*)$ 
denote the space of all maps $f$ from 
$\Omega$ to $\Omega^*$ satisfying 
the following {\it $\alpha$-H\"older condition} 
\begin{equation*}\label{cond:Holder_condition}
[f]_{\alpha,\Omega}
\;\stackrel{\tiny{\mathrm{def}}}{=}\;
\sup_{x,y\in \Omega; x\neq y} 
\frac{d_{\Omega^*}(f(x),f(y))}{d_{\Omega}(x,y)^\alpha}
<\infty \ .
\end{equation*}
When the domain of $f$ is clear we will write $[f]_{\alpha}$ instead of $[f]_{\alpha,\Omega}$.
In the case when $\Omega^*=\mathbb{R}^d$, the set
$C^\alpha(\Omega,\mathbb{R}^d)$ has a linear structure and
$[\,\cdot\,]_{\alpha,\Omega}$ 
defines a semi-norm\footnote{This also induces a pseudo-distance which we will call the {\it $C^\alpha$-pseudo-distance}.}, which we call the {\it $C^\alpha$-semi-norm}.
Consequently
\begin{equation*}
\|f\|_{C^\alpha(\Omega,\mathbb{R}^d)}
\;\stackrel{\tiny{\mathrm{def}}}{=}\;
\|f\|_{C^0(\Omega,\mathbb{R}^d)}+[f]_{\alpha,\Omega}
\end{equation*}
defines a complete norm on 
$C^\alpha(\Omega,\mathbb{R}^d)$. 
(Note that, in this case we will 
often consider the expression 
$[f-g]_{\alpha,\Omega}$ 
which obviously has no meaning 
unless $\Omega^*$ is contained in some linear space.)

Let $\mathcal{H}^\alpha(\Omega,\Omega^*)$ 
denote the space of invertible maps 
$f$ from $\Omega$ to $\Omega^*$
for which 
$f\in C^\alpha(\Omega,\Omega^*)$ 
and
$f^{-1}\in C^\alpha(\Omega^*,\Omega)$.
The {\it bi-$\alpha$-H\"older constant} of $f$ in $\mathcal{H}^\alpha(\Omega,\Omega^*)$ is the positive real number
$\max([f]_{\alpha,\Omega},[f^{-1}]_{\alpha,\Omega^*})$.

\subsubsection*{H\"older mappings between manifolds.}
On spaces more general than Euclidean 
domains, there are several ways to 
define H\"older continuity.
A direct way is to endow the 
space with a distance function. 
However, this leads to difficulties 
in defining a topology on the space 
of H\"older maps.
(Either we could introduce a distance 
function $d$ on the range and consider 
$[d(f,g)]_{\alpha,\Omega}$ 
or, if 
$\delta_{f,\alpha}(x,y)$ 
denotes the $\alpha$-H\"older difference 
quotient with respect to $f$, 
then we could consider 
$\sup_{x\neq y} |\delta_{f,\alpha}(x,y)-\delta_{g,\alpha}(x,y)|$.
Only when the range is contained in a normed 
linear space and the natural distance function 
is used do these definitions coincide, 
with both expressions being equal to 
$[f-g]_{\alpha,\Omega}$.)

Instead, as we only consider the case when 
the underlying spaces are manifolds, 
we proceed with the following construction, 
which is analogous to the construction 
of the $C^r$-Whitney topology~\cite{HirschBook}.

Take smooth compact manifolds $M$ and $N$.
We say that $f\in C^0(M,N)$ is $\alpha$-H\"older 
continuous if, for any pair of charts 
$(U,\varphi)$ on $M$ and $(V,\psi)$ on $N$, 
the map 
$\psi\circ f\circ\varphi^{-1}$ 
is $\alpha$-H\"older continuous on the 
Euclidean domain $\varphi(U\cap f^{-1}(V))$.
(Note: in a given pair of charts, 
since any smooth metric is Lipschitz 
equivalent to the Euclidean metric, 
this definition will coincide with 
the definition above.)  
Let $C^\alpha(M,N)$ denote the set of 
$\alpha$-H\"older continuous maps from $M$ to $N$.
Denote by $\mathcal{H}^\alpha(M,N)$ the subspace of 
homeomorphisms $f$ such that $f\in C^\alpha(M,N)$ 
and $f^{-1}\in C^\alpha(N,M)$.
When $M$ and $N$ coincide we denote this subspace by $\mathcal{H}^\alpha(M)$.

\subsubsection*{Spaces of bi-H\"older mappings.}
We define a topology 
on $\mathcal{H}^\alpha(M,N)$ as follows.
Given 
$f\in\mathcal{H}^\alpha(M,N)$,
take
$\epsilon>0$,
charts 
$(U,\varphi)$ on $M$ and 
$(V,\psi)$ on $N$, 
such that 
$f(U)\cap V\neq \emptyset$, 
and compact sets 
$K\subset U\cap f^{-1}(V)$, 
$L\subset f(U)\cap V$, 
which are the closure of open sets. 
%(Note: the Sobolev case in \S~\ref{sec:Sobolev_prelim} below will also require piecewise-smooth boundary)
Denote by $\mathcal{N}_{C^\alpha}(f;(U,\varphi),(V,\psi),K,L,\epsilon)$ the set of maps 
$g\in\mathcal{H}^\alpha(M,N)$ such that 
$g(K)\subseteq V$,
$g^{-1}(L)\subseteq U$,
\begin{equation*}
\|\psi\circ f\circ \varphi^{-1}-\psi\circ g\circ\varphi^{-1}\|_{C^\alpha(\varphi(K),\mathbb{R}^d)}
<\epsilon \ ,
\end{equation*}
and 
\begin{equation*}
\|\varphi\circ f^{-1}\circ \psi^{-1}-\varphi\circ g^{-1}\circ\psi^{-1}\|_{C^\alpha(\psi(L),\mathbb{R}^d)}
<\epsilon \ .
\end{equation*}
The collection of all sets defined in this way form a subbasis for a topology on $\mathcal{H}^\alpha(M,N)$.
We call it the {\it (weak) $\alpha$-H\"older-Whitney topology\/}.
As shorthand, we will also refer to it as the {\it (weak) $C^\alpha$-Whitney topology}.
Observe that the definition is analogous to the (weak) $C^r$-Whitney topology,
the only difference being the choice of norm we use in each chart.
As in the $C^r$-case (see, {\it e.g.\/},~\cite[Chapter 2]{HirschBook}) this topology is Hausdorff, and one can show the following.
\begin{proposition}
For each $\alpha\in (0,1)$, and each pair of smooth 
compact manifolds $M$ and $N$ (possibly with boundary), 
the space
$\mathcal{H}^\alpha(M,N)$, 
endowed with the (weak) $C^\alpha$-Whitney topology, 
satisfies the Baire property.
\end{proposition} 
For Lipschitz maps we can define all the objects above as in the H\"older case.
However, for clarity we will use a different notation.
Namely, denote by $C^{\mathrm{Lip}}(\Omega,\Omega^*)$ the space of all Lipschitz continuous maps from $\Omega$ to $\Omega^*$ and denote the Lipschitz constant by
$[f]_{\mathrm{Lip},\Omega}$.
Abusing notation slightly, 
we denote by $\mathcal{H}^{1}(\Omega,\Omega^*)$  
the subspace of bi-Lipschitz maps from $\Omega$ to $\Omega^*$. 
The {\it bi-Lipschitz constant} of the bi-Lipschitz map
$f$ in $\mathcal{H}^{1}(\Omega,\Omega^*)$ is the positive real number 
$\max([f]_{\mathrm{Lip},\Omega},[f^{-1}]_{\mathrm{Lip},\Omega^*})$.

For manifolds $M$ and $N$ we may also define the {\it (weak) Lipschitz-Whitney topology} on
$\mathcal{H}^{1}(M,N)$, the space of bi-Lipschitz homeomorphisms 
from $M$ to $N$, as in the H\"older case.
For $0\leq \alpha<\beta\leq 1$,
let 
$\mathcal{H}^\beta_\alpha(M,N)$ 
denote the closures of 
$\mathcal{H}^\beta(M,N)$ 
with respect to the $C^\alpha$-Whitney topology.
\begin{remark}
As previously mentioned, the 
$C^\alpha$-Whitney topology does not 
require the existence of a distance 
function on the manifold $M$.
However, we fix now, once and for all, a distance function $d_M$ on $M$.
This is merely to simplify notation in the construction of open sets, etc.
In particular, our results do not depend on this metric.
\end{remark}
\subsubsection*{Basic properties of H\"older mappings.}
In the remainder of this subsection we collect the following straightforward, though useful, results.
\begin{lemma}[H\"older Arzela-Ascoli Principle]\label{lem:holderAAprinciple_simple}
For $\alpha\in (0,1)$, and $\beta\in (\alpha,1)$ or $\beta=\mathrm{Lip}$,
the space 
$C^\beta(\Omega,\mathbb{R}^d)$ 
embeds compactly into 
$C^\alpha(\Omega,\mathbb{R}^d)$.
\end{lemma}
\begin{proposition}[H\"older Rescaling Principle]\label{prop:holderrescaling1}
Let $0\leq \alpha<\beta\leq 1$.
Let $\Omega$, $\Omega_0$ and $\Omega_1$ be bounded subsets of $\mathbb{R}^d$. 
Let $f\colon \Omega\to\Omega$ be $\beta$-H\"older continuous.
Let $\phi_0\colon \Omega\to\Omega_0$ and $\phi_1\colon \Omega\to\Omega_1$ be bi-Lipschitz continuous bijections.
Let $g=\phi_1\circ f\circ \phi_0^{-1}\colon\Omega_0\to\Omega_1$.
Then
\begin{equation*}
[g]_{\alpha}
\leq [\phi_1]_{\mathrm{Lip}}[f]_{\beta}[\phi_0^{-1}]_{\mathrm{Lip}}^\beta \mathrm{diam}(\Omega_0)^{\beta-\alpha} \ .
\end{equation*}
\end{proposition}
Observe that the following Gluing Principles 
allow us to show that H\"older maps constructed 
by gluing with charts are H\"older in the more 
usual sense, when the manifold is endowed with a smooth metric.
\begin{proposition}[First H\"older Gluing Principle]\label{prop:Holder_gluing_i}
For $\alpha\in (0,1)$ the following holds.
Let $\Omega\subset \mathbb{R}^d$ be a connected bounded open domain.
Let $\Omega_1,\Omega_2\subset \Omega$ be disjoint subdomains such that 
$\overline{\Omega}_1\cup \overline{\Omega}_2=\overline{\Omega}$.
Let 
$f_1\in C^\alpha\left(\overline{\Omega}_1,\mathbb{R}^d\right)$ and 
$f_2\in C^\alpha\left(\overline{\Omega}_2,\mathbb{R}^d\right)$ 
have the property that they extend to a continuous function $f$ on $\mathrm{\Omega}$.
Then $f$ is $\alpha$-H\"older continuous.
In fact,
\begin{equation*}
[f]_{\alpha}
\leq 
C\max\left\{
[f_1]_{\alpha}, [f_2]_{\alpha}
\right\} \ ,
\end{equation*}
where $C>0$ is a constant depending upon $\alpha$, $\Omega_1$ and $\Omega_2$ only.
\end{proposition}
We will say that a collection of pairwise disjoint bounded open subsets 
$\Omega_1,\Omega_2,\ldots$ of a metric space $(\Omega,d)$
are {\it $\kappa$-well-positioned} if
\begin{equation}\label{eq:well_positioned}
\frac{\max_{m}\mathrm{diam}(\Omega_m)}{\min_{i<j}\mathrm{dist}_H(\Omega_i,\Omega_j)}
\leq 
\kappa \ .
\end{equation} 
\begin{proposition}[Second H\"older Gluing Principle]\label{prop:Holder_gluing_ii}
For $\alpha\in(0,1)$ the following holds.
Let 
$(\Omega,d)$ and $(\Omega^*,d^{*})$ be 
connected metric spaces.
Take pairwise disjoint bounded open sets 
$\Omega_1,\Omega_2,\ldots,\Omega_n\subset \Omega$ 
which are $\kappa$-well positioned for some positive real number $\kappa$,
and define 
$\Omega_0=\Omega\setminus \bigcup_{1\leq m\leq n}\Omega_m$.
%i.e.
%\begin{equation}
%\kappa=
%\frac{\max_{1\leq m\leq n}\mathrm{diam}(\Omega_m)}{\min_{1\leq i<j\leq n}\mathrm{dist}_H(\Omega_i,\Omega_j)}
%\end{equation}
%is finite.
Let 
$f\in C^0(\Omega,\Omega^*)$ 
have 
restrictions
$f|_{\Omega_0}$ 
and 
$f|_{\overline{\Omega}_m}$, $m=1,2,\ldots,n$, 
which are $\alpha$-H\"older continuous.
Then $f$ is $\alpha$-H\"older continuous on $\Omega$ and
\begin{equation*}\label{eq:Holder_gluing_ii}
[f]_{\alpha,\Omega}
\leq K \max_{0\leq m\leq n}[f]_{\alpha,\overline{\Omega}_m} \ ,
\end{equation*}
where $K$ is a constant depending only upon $\alpha$ and $\kappa$.
\end{proposition}
A consequence of the H\"older Rescaling Principle 
(Proposition~\ref{prop:holderrescaling1}) 
is the following.
\begin{lemma}\label{lem:Lipschitz_coord_change}
Let $0\leq \alpha<\beta\leq 1$.
Let $(\Omega,d)$ and $(\Omega^*,d^{*})$ be compact metric spaces.
For any 
$f\in \mathcal{H}^{\beta}_{\alpha}(\Omega,\Omega^*)$, 
$\phi\in\mathcal{H}^{1}(\Omega)$, 
and 
$\psi\in\mathcal{H}^{1}(\Omega^*)$,
the map
$\psi\circ f\circ \phi$ 
lies in 
$\mathcal{H}^{\beta}_{\alpha}(\Omega,\Omega^*)$.
\end{lemma}
The Second H\"older Gluing Principle combines with Lemma~\ref{lem:Lipschitz_coord_change} above to give the following.
\begin{corollary}\label{cor:Lip-alpha-perturbation}
Let $0\leq \alpha<\beta\leq 1$.
Let $(\Omega,d)$ and $(\Omega^*,d^{*})$ be compact metric spaces and 
let $\Omega_1, \Omega_2, \ldots$ be pairwise disjoint open subsets of $\Omega$.
Take $f\in\mathcal{H}^{\beta}_{\alpha}(\Omega,\Omega^*)$ and,
for $k=1,2,\ldots$, take homeomorphisms 
$\phi_k\in\mathcal{H}^{1}(\Omega)$, supported in $\Omega_k$. 
Define
\begin{equation*}
g=\left\{\begin{array}{ll}
f\circ\phi_k & \mbox{in} \ \Omega_k\\
f & \mbox{in} \ \Omega\setminus\bigcup_k \Omega_k
\end{array}\right. \ .
\end{equation*}
Then $g$ lies in $\mathcal{H}^{\beta}_{\alpha}(\Omega,\Omega^*)$.
\end{corollary}
\begin{remark}
In Corollary~\ref{cor:Lip-alpha-perturbation} if, 
instead of pre-composition by bi-Lipschitz mappings 
with compact pairwise disjoint supports, once considers 
is replaced by post-composition, then the equivalent statement is also valid.
\end{remark}
%
%
%
%

%
%
%
%%%%%%%%%%%%%%%%%%%%%%%%%%%%%%%%%%%%%%%%%%%%%%%%%%%%%%%%%%%%%%%%
\subsection{The H\"older Closing Lemma}\label{sec:holderclosing}
In this section we will consider spaces of 
homeomorphisms on compact manifolds of 
dimension greater than one.
We prove an analogue of Pugh's $C^1$-Closing 
Lemma~\cite{Pugh67a,Pugh67b} in a subspace of 
bi-H\"older maps.

Recall that,
given a continuous self-map $f$ of a topological space $X$, 
a point $x$ in $X$ is {\it non-wandering} if for all 
neighbourhoods $U$ of $x$ there exists some positive integer 
$n$ for which
$f^n(U)\cap U\neq\emptyset$.

%%%%%%%%%%%%%%%
\begin{comment}
\begin{theorem}[Pugh's $C^1$-Closing Lemma]
Let $M$ be a smooth compact manifold.
Let $f\in \mathrm{Diff}^1(M)$ and $x$ a point 
in the non-wandering set of $f$.
Then in any neighbourhood $\mathcal{N}$ of $f$ 
in $\mathrm{Diff}^1(M)$ there exists $g$ such 
that $x$ is a periodic point of $g$.
\end{theorem}
\end{comment}
%%%%%%%%%%%%%

\vspace{5pt}

\begin{theorem}[H\"older Closing Lemma]~\label{lem:Holder_closing}
Let $M$ be a smooth compact manifold.
For $0\leq \alpha<\beta\leq 1$, the following holds:
Take $f\in \mathcal{H}^{\beta}_{\alpha}(M)$ and 
let $y$ be a non-wandering point of $f$.
For each neighbourhood $W$ of $y$ in $M$ and each neighbourhood $\mathcal{N}$ of $f$ in $\mathcal{H}^{\beta}_{\alpha}(M)$ 
there exists $g$ in $\mathcal{N}$ and a point $x$ in $W$ such that $x$ is a periodic point of the map $g$.
\end{theorem}
\begin{remark}\label{rmk:sub-basic_nhds-Holder_closing}
It suffices to show that such a map $g$ exists in any finite intersection of sub-basic sets
of the form 
$\mathcal{N}_{C^\alpha}(f;(U,\varphi),(V,\psi),K,L,\epsilon)$,
since such neighbourhoods form a local basis about $f$, {\it i.e.}, 
any neighbourhood of $f$ will contain such a finite intersection.
%where 
%(i) 
%$f$ is the given map, 
%(ii) 
%$K$ and $L$ are the closures of unions of balls
%, and
%(iii) $\epsilon$ is fixed (i.e., independent of the sub-basic set)
\end{remark}
To prove Theorem~\ref{lem:Holder_closing} %H\"older Closing Lemma 
we need the following preparatory lemma.
\begin{lemma}\label{lem:non-wandering}
For each non-wandering point $y$ and each sufficiently small, positive real number $\eta$ the following holds: 
there exists 
a point $x$ in $B(y,\eta)$ and 
a positive integer $k$ 
such that 
$f^{k}(x)$ also lies in $B(y,\eta)$ and, for all $j=1,2,\ldots,k-1$, 
\begin{equation*}
f^j(x)\notin B\left(x,\tfrac{3}{4}\rho\right)\cup B\left(f^{k}(x),\tfrac{3}{4}\rho\right) \ ,
\end{equation*}
where $\rho=d_M(x,f^{k}(x))$.
\end{lemma}
%
%\begin{remark}[REMOVE]
%We say that an orbit segment $\{p_0,\ldots,p_n\}$ satisfies the {\it minimal distance property} if $|p_0-p_n|=\inf_{k\neq l}|p_k-p_l|$.
%The fact that any finite orbit segment starting at $B(y,\eta)$ 
%has a sub-segment satisfying the minimal distance property is trivial
%(from the finiteness of the orbit segment).
%The non-trivial part is that this can be done with endpoints in the same 
%neighbourhood $B(y,\eta)$.
%\end{remark}
%
\begin{remark}
As was pointed out to us by Charles Pugh, this is, in essence, the {\it Fundamental Lemma} given in his paper~\cite{Pugh67a}.
However, we include this version here for completeness.
\end{remark}
\begin{proof}[Proof of Lemma~\ref{lem:non-wandering}]
As $y$ is a non-wandering point, there exists 
a point $z$ in $B(y,\frac{\eta}{10})$ 
such that 
$f^{m_1}(z)$ also lies in $B(y,\frac{\eta}{10})$ 
for some positive integer $m_1$.
Let $z_m=f^m(z)$ for each integer $m$ and denote the orbit segment $\{z_0,z_1,\ldots,z_{m_1}\}$ by $O$.
Also let $m_0=0$.
Let $\rho_1=d_M(z_{m_1},z_{m_0})$.
If 
\begin{equation*}
\left(B\left(z_{m_0},\tfrac{3}{4}\rho_1\right)\cup B\left(z_{m_1},\tfrac{3}{4}\rho_1\right)\right)\cap O
=\{z_{m_0},z_{m_1}\}
\end{equation*}
holds, then we are done.
Otherwise there exists a point $z_{m_2}$ in the orbit segment $O$, with $m_2$ different from $m_0$ and $m_1$, such that 
$d_M(z_{m_2},z_{m_1})<\frac{3}{4}d_M(z_{m_1},z_{m_0})$, say.
Let $\rho_2=d_M(z_{m_2},z_{m_1})$.
If
\begin{equation*}
\left(B\left(z_{m_1},\tfrac{3}{4}\rho_2\right)\cup B\left(z_{m_2},\tfrac{3}{4}\rho_2\right)\right)\cap O
=\{z_{m_1},z_{m_2}\}
\end{equation*}
holds, then we are done.
Otherwise there exists a point $z_{m_3}$ in the orbit segment $O$, with $m_3$ different from $m_1$ and $m_2$, such that
$d_M(z_{m_3},z_{m_2})<\frac{3}{4}d_M(z_{m_2},z_{m_1})$, say, etc.

Continuing in this way, we move from 
one pair of points, $z_{m_n}$ and $z_{m_{n-1}}$, 
to the next, $z_{m_{n+1}}$ and $z_{m_{n}}$. 
Since there are only finitely many points in the orbit segment $O$, 
and since the distance between pairs decreases at least geometrically 
(which implies that $z_{m_a}\neq z_{m_b}$ for $a\neq b$), 
it follows that this process must terminate. 
Hence there are
points $z_{m_{N}}$ and $z_{m_{N-1}}$ in the orbit segment $O$ such that
\begin{equation*}
\left(B\left(z_{m_{N}},\tfrac{3}{4}\rho_{N}\right)\cup B\left(z_{m_{N-1}},\tfrac{3}{4}\rho_{N}\right)\right)\cap O
=\{z_{m_{N}},z_{m_{N-1}}\} \ ,
\end{equation*}
where $\rho_{N}=d_M(z_{m_{N}},z_{m_{N-1}})$.
Moreover, as distances between subsequent pairs of points decreases at least geometrically, 
the distance between the initial point $z_0=z_{m_0}$ and the terminal point $z_{m_{N}}$ satisfies the following upper bound
\begin{align*}
d_M(z_0,z_{m_{N}})
\leq \sum_{n=0}^{N} d_M(z_{m_{n+1}},z_{m_n})
\leq \sum_{n=0}^{N} \left(\frac{3}{4}\right)^n d_M(z_{m_0},z_{m_1})
%\leq \frac{\eta}{5}\frac{1}{1-\frac{3}{4}}
=\frac{4}{5}\eta \ .
\end{align*}
Consequently the points $z_{m_{N}}$ and $z_{m_{N-1}}$ also lie in $D(y,\eta)$.
Consider the case when
$m_{N}<m_{N-1}$.
Then setting $x=z_{m_{N}}$ and $k=m_{N-1}-m_{N}$, 
so that $f^{k}(x)=z_{m_{N-1}}$, 
we find that the point $x$ and integer $k$ satisfy the conclusion of the lemma.
The case when $m_{N-1}<m_{N}$ is similar. 
Hence the lemma is shown.
\end{proof}
We now proceed with the proof of the H\"older Closing Lemma (Theorem~\ref{lem:Holder_closing}).
\begin{proof}[Proof of the H\"older Closing Lemma]
We will prove the theorem in the case $\beta=1$, {\it i.e.}, 
for maps in the $C^\alpha$-closure of the space of bi-Lipschitz homeomorphisms.
The general case follows analogously as the only property being used here is that the map $f$ satisfies the {\it little H\"older condition}, {\it i.e.},
for each $x\in M$, $[f]_{\alpha,B(x,r)}=o(r)$.

\vspace{5 pt}

\noindent
{\sl Setup:\/}
Following Remark~\ref{rmk:sub-basic_nhds-Holder_closing},
it suffices to construct the perturbation $g$ so that it 
lies in the intersection $\mathcal{N}$ of a finite collection of 
sub-basic sets of the form
\begin{equation}\label{def:sub-basic}
\mathcal{N}_{C^\alpha}(f;(U_n,\varphi_n),(V_n,\psi_n),K_n,L_n,\epsilon_n)
\end{equation}
as defined in Section~\ref{sec:Holder_prelim}.
By adding to the collection of sub-basic sets, 
if necessary, we may assume that $y$ is contained in $U_0\cap f^{-1}(V_0)$.
Take compact neighbourhoods 
$\mathsf{U}_n$ of $K_n$ in $U_n\cap f^{-1}(V_n)$
and
$\mathsf{V}_n$ of $L_n$ in $f(U_n)\cap V_n$.
We will also take a compact neighbourhood $\mathsf{W}_0$ in $U_0$ which contains $y$ in its interior.

To simplify notation, for each index $n$ define the map
\begin{equation}\label{def:fn}
f_n=\psi_n\circ f\circ\varphi_n^{-1}\colon \varphi_n(U_n\cap f^{-1}(V_n))\longrightarrow \psi_n(f(U_n)\cap V_n) \ .
\end{equation}
When considering a perturbation $g$ of $f$ we will also use the notation
\begin{equation}\label{def:gn}
g_n=\psi_n\circ g\circ\varphi_n^{-1}\colon \varphi_n(U_n\cap g^{-1}(V_n))\longrightarrow \psi_n(g(U_n)\cap V_n) \ .
\end{equation}
Fix a positive real number $\epsilon$. 
This will denote the order of the size of the perturbation.
Take a positive real number $\delta$.
This will denote the size of the support of the local perturbation.
Take $\delta$ sufficiently small so that
\begin{enumerate}
\item[(a)]
$B_M(y,\delta)$ is contained in $W\cap \mathsf{W}_0$
\item[(b)]
$\mathsf{U}_n$ contains a $2\delta$-neighbourhood of the compact set $K_n$,
and
$\mathsf{V}_n$ contains a $2\delta$-neighbourhood of compact set $L_n$.
\item[(c)]
$\max\left\{
[f_n]_{\alpha,\varphi_n(B(y,\delta))},
[f_n^{-1}]_{\alpha,f_n\circ\varphi_n(B(y,\delta))}
\right\}\leq \epsilon \ $.
\end{enumerate}
The neighbourhood 
$W\cap\mathsf{W}_0$ 
will contain the support of our perturbation.
However, we also need to control the size of the perturbation in charts other than $(U_0,\varphi_0)$. 
(This explains why we consider $W\cap\mathsf{W}_0$ and not just the open set $W$.)
Therefore we will also assume that, for any index $n$
\begin{equation}\label{ineq:transition-Lip-Holder-closing}
\max\left\{
[\varphi_n\circ \varphi_0^{-1}]_{\mathrm{Lip},\varphi_0(\mathsf{W}_0\cap \mathsf{U}_n)}, \ 
[\varphi_0\circ \varphi_n^{-1}]_{\mathrm{Lip},\varphi_n(\mathsf{W}_0\cap \mathsf{U}_n)}
\right\}
<c_1 \ .
\end{equation}
\vspace{5 pt}

\noindent
{\sl Construction of the perturbation:\/}
Given a point $x^0$ in $M$, for each integer $k$ let $x^k=f^k(x^0)$.
For each index $n$, let $x^k_n=\varphi_n(x^k)$, whenever it is defined.
As $y$ is a non-wandering point of $f$, a consequence of Lemma~\ref{lem:non-wandering}, is the following.
\begin{claim}
There exists a positive real number $c$ with the following property:
for each sufficiently small positive real number $\delta$
there exists 
a point $x^0$ in $M$ and 
a positive integer $k_0$
such that
\begin{enumerate}
\item[(1)] 
$x^0, x^{k_0}\in B(y,\delta)$, 
\item[(2)]
setting 
$r_0=|x^0_0-x^{k_0}_0|$, 
for each integer $k$, where $0<k<k_0$, 
either $x^k_0$ is not defined or $x^k_0$ is defined and
$x^k_0\notin 
E(x^0_0,x^{k_0}_0;c r_0)
%\subsetB\bigl(x^0_0,\tfrac{3}{4}r_0\bigr)\cup B\bigl(x^{k_0}_0,\tfrac{3}{4}r_0\bigr)
$
\item[(3)] 
$E(x^0_0,x^{k_0}_0;c r_0)
%\subset B\bigl(x^{0}_{0},\tfrac{3}{4}r_0\bigr)\cup B\bigl(x^{k_0}_{0},\tfrac{3}{4}r_0\bigr)
\subset \varphi_0(B(y,\delta))$, 
\end{enumerate}
\end{claim}
\vspace{5 pt}

\noindent
(Claim 1(3) follows by applying the Lemma~\ref{lem:non-wandering} to a slightly smaller disk
%, i.e. contained in $\varphi_0(U_0\cap f^{-k}(U_0))$, and so that 
%the argument for $f$ on $M$, together with the fact the charts are uniformly Lipschitz, can be transported to $f_0$ on $\mathbb{R}^d$ 
%if necessary
.)
Define
\begin{equation*}
E=E\bigl(x^0_0,x^{k_0}_0,cr_0\bigr)\ , \qquad
E'=E\bigl(x^0_0,x^{k_0}_0;\tfrac{c}{2}r_0\bigr)
\end{equation*}
and let
$E_M=\varphi_0^{-1}(E)$
and
$E'_M=\varphi_0^{-1}(E')$.
Applying Lemma~\ref{lem:E-perturbation} to the neighbourhoods $E'$ and $E$, 
there exists a diffeomorphism $\phi$ supported on $E$ such that 
\begin{equation}\label{eq:phi_0-closing-up}
\phi(x^{k_0}_0)=x^0_0 \ .
\end{equation}
Moreover, there exists a positive real number $c_1$, independent of $\epsilon$, such that 
\begin{equation}\label{ineq:phi_Lip-Holder_closing}
[\phi]_{\mathrm{Lip}}\leq c_1 \ . 
\end{equation} 
Define the self-map $g$ on $M$ by 
\begin{equation}
g=\left\{\begin{array}{ll}
f\circ\varphi_0^{-1}\circ \phi\circ\varphi_0 & \ \mbox{in} \ E_M\\
f & \ \mbox{elsewhere}
\end{array}\right. \ .
\end{equation}
Since $\phi_0$ is supported in $E$, it is clear that $g$ is a homeomorphism.
In fact, 
Corollary~\ref{cor:Lip-alpha-perturbation} implies that the map
$g$ lies in $\mathcal{H}^\beta_\alpha(M)$.
By equality~\eqref{eq:phi_0-closing-up}, and since $x^k\notin E_M$ for $0<k<k_0$, 
we also know that
\begin{equation*}
g^{k_0}(x^{k_0})
=f^{k_0}\circ\varphi_0^{-1}\circ\phi\circ\varphi_0(x^{k_0})
%=f^{k_0}\circ\varphi_0^{-1}\circ\phi(x^{k_0}_0)
%=f^{k_0}\circ\varphi_0^{-1}(x^{0}_{0})
%=f^{k_0}(x^{0})
=x^{k_0} \ .
\end{equation*}
Thus $g$ possesses a periodic point in the neighbourhood $W$.
Below it will be important to observe that, for each index $n$, 
we also have the expression
\begin{equation}\label{eq:gn_ii}
g_n=\left\{\begin{array}{ll}
f_n\circ\phi_n & \ \mbox{in} \ \varphi_n(E_M)\\
f_n & \ \mbox{elsewhere}
\end{array}\right. \ .
\end{equation}
where
\begin{equation}\label{eq:phin}
\phi_n=\varphi_n\circ\varphi_0^{-1}\circ\phi\circ\varphi_0\circ\varphi_n^{-1} \ .
\end{equation}

\vspace{5 pt}

\noindent
{\sl Size of the perturbation:\/}
It remains to estimate the $C^\alpha$-pseudo-distance between $f$ and $g$ corresponding to each of the sub-basic sets.
Fix an index $n$. 
First, we must estimate
$[f_n-g_n]_{\alpha,\varphi_n(K_n\cap E_M)}$.
If $K_n$ and $E_M$ are disjoint, there is nothing to show.
Otherwise, by (b) above,
$E_M$ is contained in $\mathsf{U}_n$.
By the triangle inequality
\begin{equation}\label{ineq:triangle-Holder_closing}
[f_n-g_n]_{\alpha,\varphi_n(K_n\cap E_M)}
\leq 
[f_n-g_n]_{\alpha,\varphi_n(E_M)}
\leq
[f_n]_{\alpha,\varphi_n(E_M)}+[g_n]_{\alpha,\varphi_n(E_M)} \ .
\end{equation}
By the expression~\eqref{eq:gn_ii} 
and the observation that $\phi_n\circ\varphi_n(E_M)=\varphi_n(E_M)$, 
the H\"older Rescaling Principle (Proposition~\ref{prop:holderrescaling1}) 
gives
\begin{equation}\label{ineq:gn_Lip-Holder_closing}
[g_n]_{\alpha,\varphi_n(E_M)}
%\leq [f_n]_{\alpha,\phi_n\varphi_n(E_M)}[\phi_n]_{\alpha,\varphi_n(E_M)}
\leq [f_n]_{\alpha,\varphi_n(E_M)}[\phi_n]_{\mathrm{Lip},\varphi_n(E_M)} \ .
\end{equation}
Next, applying the H\"older Rescaling Principle (Proposition~\ref{prop:holderrescaling1}) 
to the expression~\eqref{eq:phin}, after observing that $\phi(E)=E$, gives
\begin{equation*}
[\phi_n]_{\mathrm{Lip},\varphi_n(E_M)}
\leq 
[\varphi_n\circ\varphi_0^{-1}]_{\mathrm{Lip},\varphi_0(E_M)}
[\phi]_{\mathrm{Lip},E}
[\varphi_0\circ\varphi_n^{-1}]_{\mathrm{Lip},\varphi_n(E_M)} \ .
\end{equation*}
By~\eqref{ineq:phi_Lip-Holder_closing} and~\eqref{ineq:transition-Lip-Holder-closing},
which we apply as $E_M$ is contained in $\mathsf{W}_0\cap \mathsf{U}_n$,
this implies that $[\phi_n]_{\mathrm{Lip},\varphi_n(E_M)}$ is bounded from above independently of $\epsilon$.
By (c), together with~\eqref{ineq:triangle-Holder_closing} and~\eqref{ineq:gn_Lip-Holder_closing},
this implies that there exists a positive real number $c_3$, independent of $\epsilon$, such that
\begin{equation*}
[f_n-g_n]_{\mathrm{Lip},\varphi_n(K_n\cap E_M)}\leq c_3\epsilon \ .
\end{equation*}
Applying the First H\"older Gluing Principle (Proposition~\ref{prop:Holder_gluing_i}), 
there exists a positive real number $c_4$, also independent of $\epsilon$, for which
\begin{equation*}
[f_n-g_n]_{\mathrm{Lip},\varphi_n(K_n)}\leq c_4\epsilon \ .
\end{equation*}
%(This requires the observations that 
%(1) $f_n=g_n$ on $\varphi_n(K_n\setminus E_M)$, and 
%(2) shrinking one domain $\Omega_1$ in the H\"older Gluing Lemma doesn't change the constant $C$ when the complementary function $f_2$ is trivial.)
Since the $C^0$-distance between $f_n$ and $g_n$ can be made arbitrarily small by shrinking the support of $\phi$, 
the size of the perturbation $\epsilon$ may be chosen so that
\begin{equation*}
\|f_n-g_n\|_{C^\alpha(\varphi_n(K_n),\mathbb{R}^d)}< \epsilon_n \ .
\end{equation*}
The same argument applied to the inverse mappings shows that, after shrinking $\epsilon$ is necessary, we also have 
\begin{equation*}
\|f_n-g_n\|_{C^\alpha(\psi_n(L_n),\mathbb{R}^d)}<\epsilon_n \ .
\end{equation*}
Thus, taking the minimum of all suitable $\epsilon$ over all indices $n$, of which there are finitely many,
the resulting map $g$ will lie in the common intersection of all the sub-basic sets above.
Thus the theorem is shown.
\end{proof}
\begin{remark}
The reader may wonder why the Closing Lemma is much simpler in the H\"older category than the $C^1$ category.
While in both cases perturbations may be made by pre- or post-composing by diffeomorphisms supported on a small neighbourhood,
by shrinking the neighbourhood and conjugating the perturbation by a dilation, this leaves the $C^1$-size of the perturbation unchanged,
whereas, by the H\"older Rescaling Principle, the $\alpha$-H\"older size of the perturbation can be made arbitrarily small.
\end{remark}
The proof of the H\"older Closing Lemma (Theorem~\ref{lem:Holder_closing}) above also yields the following corollary.
\begin{corollary}\label{cor:Holder_closing}
Let $M$ be a smooth compact manifold.
For $0\leq \alpha<\beta\leq 1$, the following holds:
Take $f\in \mathcal{H}^{\beta}_{\alpha}(M)$ and
let $y$ be a recurrent point of $f$.
For each neighbourhood $\mathcal{N}$ of $f$ in $\mathcal{H}^{\beta}_{\alpha}(M)$ 
there exists $g$ in $\mathcal{N}$ and a positive integer $k$ such that $f^k(x)$ is a periodic point of the map $g$.
\end{corollary}
\vspace{5pt}
%
%
%
%
%
%%%%%%%%%%%%%%%%%%%%%%%%%%%%%%%%%%%%%%%%%%%%%%%%%%%%%%%
\subsection{Genericity of infinite topological entropy for H\"older mappings.}~\label{sect:Holder-infinite_entropy}
In this section we prove that infinite 
topological entropy is a generic property in the 
H\"older context.
Theorem~\ref{thm:generic_little-Holder_homeomorphism} will follow from 
Theorem~\ref{thm:open+dense_entropy>logN-little_Holder} below.
Before we can state this we need to introduce the following terminology.
Call $B^{d-1}\times B^1$ the {\it standard solid cylinder\/} in $\mathbb{R}^d$.
We call 
images under affine transformations of the standard solid cylinder {\it rigid solid cylinders\/} and 
homeomorphic images of the standard solid cylinder {\it topological solid cylinders\/}.
Given a rigid solid cylinder $C$ denote the axial length 
and the coaxial radius of $C$ respectively by $\mathrm{len}(C)$ and $\mathrm{rad}(C)$.
Given distinct points $a$ and $b$ in $\mathbb{R}^d$ and $r>0$,
denote by $C(a,b;r)$ the rigid solid cylinder in $\mathbb{R}^d$ 
whose axis is the line segment $[a,b]$ and whose co-axial radius is $r$.

Let $C$ be a topological solid cylinder in $\mathbb{R}^d$ 
with disjoint marked boundary balls $C^+$ and $C^-$. 
We say that an embedding $\phi$, from some domain containing $C$ into $\mathbb{R}^d$,
{\it maps $C$ across the standard solid cylinder\/} 
$B^{d-1}\times B^{1}$ if the following properties are satisfied
(see Figure~\ref{fig:cylinders2}) 
\begin{enumerate}
%\item 
%$\phi(C)\subset D^{d-1}\times \mathbb{R}$
%\item 
%$\phi(D^\mp)\subset D^{d-1}\times (-\infty,-1]$
%\item 
%$\phi(D^\pm)\subset D^{d-1}\times [+1,+\infty)$
\item
$\phi(C)$ intersects $B^{d-1}\times B^1$,
\item
$\phi(C)$ does not intersect $\partial \left(B^{d-1}\times B^1\right)\setminus \left(B^{d-1}\times \partial B^1\right)$,
\item
$\overline{\phi(C^-)}$ and $\overline{\phi(C^+)}$ do not intersect $\overline{B^{d-1}\times B^1}$,
\item
the connected component of $\phi(C)\setminus \left(B^{d-1}\times B^1\right)$ 
whose boundary contains $\phi(C^\pm)$ has closure intersecting 
$B^{d-1}\times \{\pm 1\}$ but not $B^{d-1}\times \{\mp 1\}$.
\end{enumerate}
%%%%%%%%%%%%%%%%%%%%%%%%%%%%%%%%%%%%%%%%%%%%%%%%% 
%\begin{comment} 
%%%%%%%%%%%%%%%%%%%%%%%%%%%%%%%%%%%%%%%%%%%%%%%% 
\begin{figure}[h] 
\begin{center} 
\psfrag{Dd-1xD1}[][]{$B^{d-1}\times B^1$} 
\psfrag{C}[][]{$C$} 
\psfrag{phi}[][]{$\phi$} 
\includegraphics[width=4.1in]{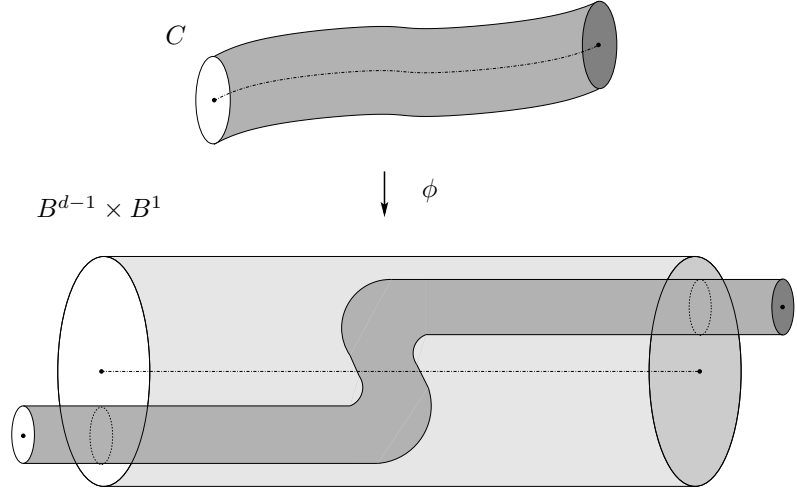}
\end{center} 
\caption[slope]{\label{fig:cylinders2} 
The cylinder $C$ maps across the standard solid 
cylinder $B^{d-1}\times B^1$ via the embedding $\phi$.} 
\end{figure} 
%%%%%%%%%%%%%%%%%%%%%%%%%%%%%%%%%%%%%%%%%%%%%%%% 
%\end{comment} 
%%%%%%%%%%%%%%%%%%%%%%%%%%%%%%%%%%%%%%%%%%%%%%%%

Given a topological solid cylinder $C'$ in $\mathbb{R}^d$, 
which is the image of the standard solid cylinder under the homeomorphism $\psi$,
we say that {\it $\phi$ maps $C$ across $C'$} if $\psi^{-1}\circ \phi$ maps $C$ across the standard solid cylinder.

Given a positive integer $N$, we say that an 
embedding $f$ maps a (topological) solid cylinder $C_0$ across 
a (topological) solid cylinder $C_1$ {\it like an $N$-branched horseshoe\/} 
if there exist pairwise disjoint subcylinders 
$C_{0,1},C_{0,2},\ldots,C_{0,N}$ of $C_0$ such that 
$f$ maps $C_{0,k}$ across $C_{1}$ for each $k=1,2,\ldots,N$. 
See Figure~\ref{fig:cylinders3}. 
% 
%%%%%%%%%%%%%%%%%%%%%%%%%%%%%%%%%%%%%%%%%%%%%%%%
%\begin{comment}
%%%%%%%%%%%%%%%%%%%%%%%%%%%%%%%%%%%%%%%%%%%%%%%%
\begin{figure}[h]
\begin{center}
\psfrag{C0}[][]{$C_{0}$} 
\psfrag{C1}[][]{$C_{1}$}
\psfrag{phi}[][]{$f$}
\psfrag{phiC0}[][]{$f(C_0)$}
\psfrag{C01}[][][1]{$C_{0,1}$}
\psfrag{C02}[][][1]{$C_{0,2}$}
\psfrag{C03}[][][1]{$C_{0,3}$}
\includegraphics[width=4.1in]{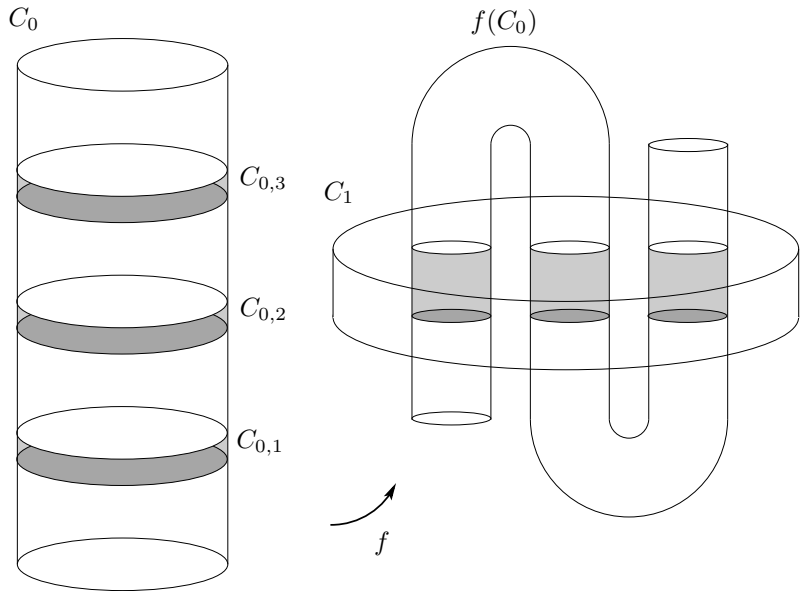}
\end{center}
\caption[slope]{\label{fig:cylinders3}
The embedding $f$ maps the cylinder $C_0$ across the cylinder $C_1$ like a $3$-branched horseshoe.
}
\end{figure}
%%%%%%%%%%%%%%%%%%%%%%%%%%%%%%%%%%%%%%%%%%%%%%%%
%\end{comment}
%%%%%%%%%%%%%%%%%%%%%%%%%%%%%%%%%%%%%%%%%%%%%%%%
% 
For results 
concerning solid cylinders used in this section 
see Appendix~\ref{sect:cylinder_perturbations}. 
Theorem~\ref{thm:generic_little-Holder_homeomorphism}, 
stated in \S~\ref{subsect:summary_of_results}, 
follows directly from the following result. 
\begin{theorem}\label{thm:open+dense_entropy>logN-little_Holder} 
Let $0\leq \alpha< 1$. 
Assume that $f\in\mathcal{H}^{1}_{\alpha}(M)$. 
For 
each neighbourhood $\mathcal{N}$ of $f$ in 
$\mathcal{H}^{1}_{\alpha}(M)$ 
and 
each positive integer $N$ 
there exists 
$g\in \mathcal{H}^{1}_{\alpha}(M)$ 
such that 
\begin{itemize} 
\item[(i)] 
$g\in\mathcal{N}$ 
\item[(ii)] 
there exists 
a positive integer $k_0$, 
a topological solid cylinder $S$ in $M$ and 
solid sub-cylinders $S_1,S_2,\ldots,S_{N^{k_0}}$ 
such that 
$g^{k_0}$ maps $S_j$ across $S$ 
for $j=1,2,\ldots,N^{k_0}$ 
\end{itemize} 
The second property implies that 
$h_{\mathrm{top}}(g)\geq \log N$ 
and that this property is satisfied in an 
open neighbourhood of $g$. 
\end{theorem} 
\begin{proof}
Before starting the proof let us describe the idea.
Take a forward-recurrent orbit for $f$.
Take a segment of this orbit, of length $k_0$ say, 
whose start- and end-points are sufficiently close.
In pairwise disjoint neighbourhoods of each of the points in this orbit segment
take a solid cylinder.
Perturb $f$ in each of these neighbourhoods so that the solid 
cylinder maps over the next solid cylinder like an 
$N$-branched horseshoe.
Finally, `close-up' the orbit of the horseshoe by mapping 
the solid cylinder at the end of the the orbit segment across 
the solid cylinder prescribed at the start of the orbit segment.
Observe that if $h_0$ maps the solid cylinder $C_0$ 
across $C_1$ and if $h_1$ maps $C_1$ across $C_2$ 
then $h_1\circ h_0$ maps $C_0$ across $C_1$. 
Thus property (ii) will be satisfied.
The discussion below will therefore focus on showing that (i) is satisfied.

\vspace{5pt}

\noindent
{\sl Setup:\/}
Since $f\in \mathcal{H}^{1}_{\alpha}(M)$ 
we may assume, by making an arbitrarily small 
perturbation if necessary, that $f$ is 
bi-Lipschitz. 
As in the proof of the H\"older Closing Lemma 
(Theorem~\ref{lem:Holder_closing}) 
it suffices to 
consider the case when the neighbourhood $\mathcal{N}$ 
is a finite intersection of sub-basic sets of the form
\begin{equation*}
\mathcal{N}_{C^\alpha}(f; (U_n,\varphi_n),(V_n,\psi_n),K_n,L_n,\epsilon_n)
\end{equation*}
as defined in Section~\ref{sec:Holder_prelim}.
By adding to the collection of sub-basic sets if necessary, 
we may assume that the collections of neighbourhoods $\{U_n\}$ 
and $\{V_n\}$ both form open covers of $M$.
For each index $n$, fix a compact neighbourhood  
$\mathsf{U}_n$ of $K_n$ in $U_n\cap f^{-1}(V_n)$,
and a compact neighbourhood 
$\mathsf{V}_n$ of $L_n$ in $f(U_n)\cap V_n$.
Let
\begin{equation}\label{def:f_in_charts}
f_n=\psi_{n}\circ f\circ\varphi_{n}^{-1}
\colon 
\varphi_{n}(U_n\cap f^{-1}(V_n))
\longrightarrow
\psi_{n}(f(U_n)\cap V_n) \ .
\end{equation}
Below we will construct a perturbation $g$ of $f$.
When considering the $C^\alpha$-pseudo-distance 
between $f$ and $g$ corresponding to each sub-basic 
set we will use the notation
\begin{equation}\label{def:g_in_charts}
g_n=\psi_{n}\circ g\circ\varphi_{n}^{-1}
\colon 
\varphi_{n}(U_n\cap g^{-1}(V_n))
\longrightarrow
\psi_{n}(g(U_n)\cap V_n) \ .
\end{equation}
The map $g$ will be constructed from a finite sequence of local perturbations of $f$. 
These perturbations will be in charts. 
Rather than introduce another set of charts we will use the collection $\{(U_n,\varphi_n)\}$.
This is merely to simplify notation, and note that any other set of charts covering $M$ could be used just as well.
However, we will need to estimate these perturbations in each pair of charts corresponding to a sub-basic set.
%Each perturbation will be supported inside a neighbourhood small enough so that, for each $n$, 
%it either lies inside $\mathsf{U}_n$, or is disjoint from $K_n$, and similarly the image either lies inside $\mathsf{V}_n$ or is disjoint from $L_n$.
%In the case the supports are disjoint from $K_n$ or $L_n$, there is nothing to do.
To facilitate this, 
for each index $n$, 
take a compact neighbourhood $\mathsf{W}_n$ contained in $U_n$, 
with the property that $\{\mathsf{W}_n\}$ covers $M$. 
%We will assume that $\delta$ is also less than the Lebesgue number of this cover. 
The perturbations will be supported in these sets.
Since there are only finitely many charts under consideration, 
all of which are smooth, 
and all the sets $\mathsf{W}_n$, $\mathsf{V}_n$ and $\mathsf{U}_n$ are compact, 
there exists a positive real number $c_1$ with the property
that for any chart $(U_m,\varphi_m)$
and for each index $n$,
\begin{equation}\label{ineq:transition-Lip1}
\max\left\{
\left[\varphi_m\circ\psi_n^{-1}\right]_{\mathrm{Lip},\psi_n(\mathsf{W}_m\cap f \mathsf{U}_n)}, \
\left[\psi_n\circ\varphi_m^{-1}\right]_{\mathrm{Lip},\varphi_m(\mathsf{W}_m\cap f \mathsf{U}_n)}
\right\}
\leq c_1
\end{equation}
and 
\begin{equation}\label{ineq:transition-Lip2}
\max\left\{
\left[\varphi_m\circ\varphi_n^{-1}\right]_{\mathrm{Lip},\varphi_n(\mathsf{W}_m\cap\mathsf{U}_n)}, \ 
\left[\varphi_n\circ\varphi_m^{-1}\right]_{\mathrm{Lip},\varphi_m(\mathsf{W}_m\cap\mathsf{U}_n)}
\right\}
\leq c_1 \ ,
\end{equation}
and similarly for the sets $\mathsf{V}_n$ (with $f$ replaced appropriately by $f^{-1}$).

Fix a positive real number $\epsilon$. 
This will be the order of the size of the perturbation.
Let $\delta$ be a positive real number.
This will denote the size of the support of the perturbation.
Take $\delta$ sufficiently small so that
\begin{enumerate}
\item[(a)]
$\delta$ is less than the Lebesgue number of the 
common refinement of the finite open covers 
$\{U_n\}$ and $\{V_n\}$.
Thus any ball of radius $\delta$ or less lies in 
some set of the form $U_m\cap V_n$.
\item[(b)]
$\mathsf{U}_n$ contains a $2\delta$-neighbourhood of $K_n$, 
and 
$\mathsf{V}_n$ contains a $2\delta$-neighbourhood of $L_n$.
Importantly, this implies that, 
given $x$ in $M$, 
either 
$B_M(x,\delta)\cap K_n=\emptyset$ %i.e. $d_M(x,K_n)>\delta$
or 
$B_M(x,\delta)\subset \mathsf{U}_n$, %i.e. $d_M(x,K_n)\leq\delta$
and similarly for $L_n$ and $\mathsf{V}_n$.
\item[(c)]
for each index $n$, whenever $x$ lies in a $\delta$-neighbourhood of $K_n$,
\begin{equation*}
\max\left\{
\left[f_n\right]_{\alpha,\overline{\varphi_n(B_M(x,\delta))}}, \ 
\left[f_n^{-1}\right]_{\alpha,\overline{f_n\circ\varphi_n(B_M(x,\delta))}}
\right\}<\epsilon
\end{equation*}
and similarly when $x$ lies in $\delta$-neighbourhood of $L_n$.
(This is possible since, as $f$ and $f^{-1}$ are bi-Lipschitz, they are both little $\alpha$-H\"older continuous.
As all charts are smooth, it follows that for each index $n$,
the maps $f_n$ and $f_n^{-1}$ are also both little $\alpha$-H\"older continuous.)
\end{enumerate}

\vspace{5pt}

\noindent
{\sl Support of the perturbation:\/}
We will also need the following notation.
Given a point $x^0$ in $M$ and an integer $k$, let 
$x^k=f^k(x^0)$. 
For any index $n$ let $x^k_n=\varphi_n(x^k)$, whenever this is defined.
A slight variation of Claim 1 in the proof of the H\"older Closing Lemma gives us the following.
\begin{claim}
There exist positive real numbers $c$ and $\kappa$ with the following property:
For each sufficiently small positive real number $\delta$, there exists 
a point $x^0$ in $M$, 
a positive integer $k_0$, and 
a chart $(U_0,\varphi_0)$
such that
\begin{enumerate}
\item[(1)]
$x^{0}$ and $x^{k_0}$ lie in $U_0$
\item[(2)]
if $r_0=|x^0_0-x^{k_0}_0|$ 
then 
\begin{equation*}
E(x^{0}_{0},x^{k_0}_{0};c\cdot r_0)\subset \varphi_0(B_M(x^0,\delta))\subset\varphi_0(\mathsf{W}_0)
\end{equation*}
\item[(3)]
$x^k_0\notin E(x^{0}_{0},x^{k_0}_{0};c\cdot r_0)$ for $0<k<k_0$
\end{enumerate}
Also, there exists a positive real number $r_1<cr_0$ such that for each integer $k$ satisfying $0<k<k_0$, 
there is a chart, denoted by $(U_k,\varphi_k)$, for which
\begin{enumerate}
\item[(4)]
$B(x^k_k,r_1)\subset \varphi_k(B_M(x^k,\delta))\subset \varphi_k(\mathsf{W}_k)$
\item[(5)]
the following sets are pairwise disjoint
\begin{equation*}
\varphi_0^{-1}(E(x^0_0,x^{k_0}_{0},c\cdot r_0)), \varphi_1^{-1}(B(x^1,r_1)),\ldots, \varphi_{k_0-1}^{-1}(B(x^{k_0-1},r_1))
\end{equation*}
and, in fact, in any chart $(U_n,\varphi_n)$ these sets are $\kappa$-well positioned, 
{\it i.e.\/}, satisfy inequality~\eqref{eq:well_positioned}.
\end{enumerate}
\end{claim}

\vspace{5pt}

\noindent
To simplify notation, set
\begin{equation*}
E=E(x^{0}_{0},x^{k_0}_{0};c\cdot r_0)
\qquad
\mbox{and}
\qquad
E_M=\varphi_{0}^{-1}(E)
\end{equation*}
and, 
for $0\leq k\leq k_0$,
\begin{equation*}
B^{k}=B(x^{k}_{k},r_1)
\qquad
\mbox{and}
\qquad
B^{k}_{M}=\varphi_{k}^{-1}(B^{k}) \ .
\end{equation*}
By shrinking $r_0$ if necessary we may thus assume that each of the sets 
$E_M$, $f(E_M)$, $B^k_M$ and $f^{-1}(B^k_M)$ are contained in a ball of radius $\delta$
or less.
%%%%%%%%%%%%%%%OLD FIGURE%%%%%%%%%%%%%%%%%%%%%%%
\begin{comment}
%%%%%%%%%%%%%%%%%%%%%%%%%%%%%%%%%%%%%%%%%%%%%%%%
\begin{figure}[t]
\begin{center}
\psfrag{S_k1}[][]{$S_{k_1}$} 
\psfrag{S_k2}[][]{$S_{k_2}$}
\psfrag{f}[][]{$f$}
\psfrag{phi}[][]{$\psi_k^{-1}\circ \phi_{1,k}\circ f_k\circ \psi_k$}
\psfrag{x_k}[][][1]{$x_k$}
\psfrag{x_k1}[][][1]{$x_{k_1}$}
\psfrag{x_k2}[][][1]{$x_{k_2}$}
\includegraphics[width=5.0in]{cylinders1.eps}
\end{center}
\caption[slope]{\label{fig:cylinders1}
The preimages $S_k$ of the rigid solid cylinders $C(a_k,b_k;\rho)$ under the charts $\psi_k$, their images under $f$, and the subsequent perturbation mapping them across $S_{k+1}$.}
\end{figure}
%%%%%%%%%%%%%%%%%%%%%%%%%%%%%%%%%%%%%%%%%%%%%%%%
\end{comment}
%%%%%%%%%%%%%%%%%%%%%%%%%%%%%%%%%%%%%%%%%%%%%%%%
%
%%%%%%%%%%%%%%NEW FIGURE%%%%%%%%%%%%%%%%%%%%%%%%
\begin{figure}[t]
\begin{center}
\psfrag{M}[][]{$M$}
\psfrag{psi_k}[][]{$\varphi_{k}$}
\psfrag{psi_k+1}[][]{$\varphi_{k+1}$}
\psfrag{f}[][]{$f$}
\psfrag{f_k}[][]{$\varphi_{k+1}\circ f\circ\varphi_{k}^{-1}$}
\psfrag{phi_k}[][]{$\phi^{k}$}
\psfrag{x_k}[][]{$x^{k}$}
\psfrag{x_k+1}[][]{$x^{k+1}$}
\psfrag{psi_k(x_k)}[][]{$x^{k}_{k}$}
\psfrag{psi_k+1(x_k+1)}[][]{$x^{k+1}_{k+1}$}
\psfrag{r_3}[][]{$r_3$}
\psfrag{D_k}[][]{$B^{k}$}
\psfrag{D_k+1}[][]{$B^{k+1}$}
\psfrag{D(x_k,rho2)}[][]{$B_M(x^{k},\delta)$}
\psfrag{D(x_k+1,rho2)}[][]{$B_M(x^{k+1},\delta)$}
\psfrag{C_k}[][]{$\!\!C(a_{k},b_{k};\varrho)$}
\psfrag{C_k+1}[][]{$C(a_{k+1},b_{k+1};\varrho)$}
\psfrag{f_k(C_k)}[][]{$\varphi_{k+1}\circ f\circ\varphi_{k}^{-1}(C(a_{k},b_{k};\varrho))$}
\includegraphics[width=4.2in]{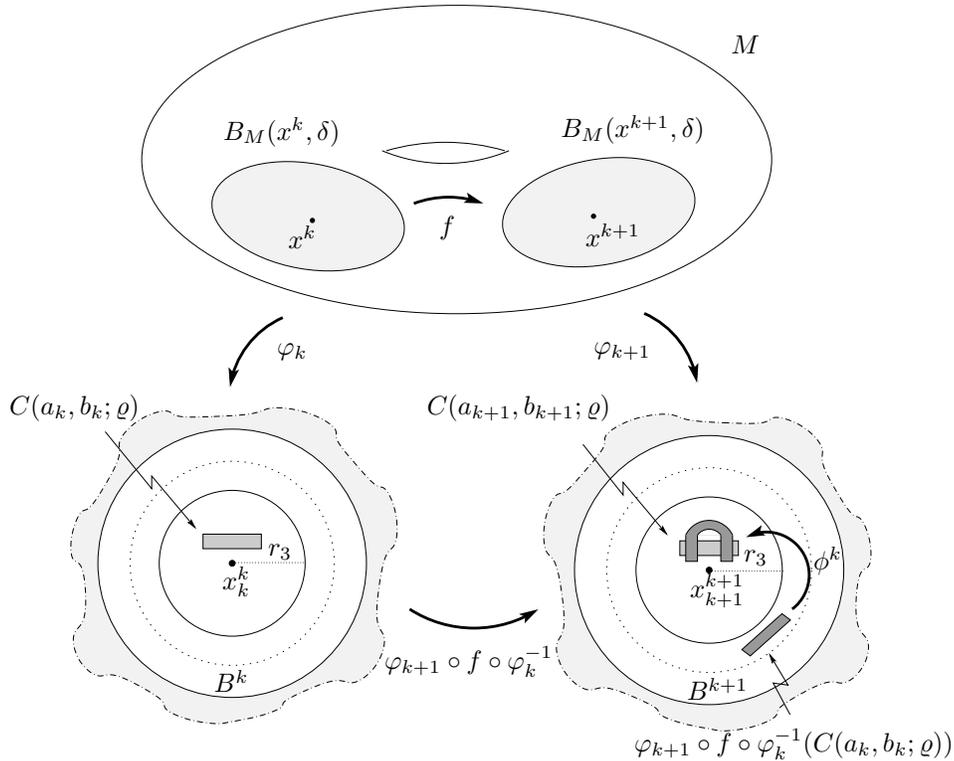}
\end{center}
\caption[slope]{\label{fig:neighbourhoods} 
The map $g_k$ is constructed as the composition of 
$\phi^k\circ\varphi_{k+1}\circ f\circ \varphi_{k}^{-1}$ and maps the cylinder $C(a_k,b_k;\varrho)$ 
across $C(a_{k+1},b_{k+1};\varrho)$ like an $N$-branched horseshoe.}
\end{figure}
%%%%%%%%%%%%%%%%%%%%%%%%%%%%%%%%%%%%%%%%%%%%%%%%

\vspace{5pt}

\noindent
{\sl Construction of the perturbation:\/}
For the choice of $r_1$ given above we now take positive real numbers $r_2$ and $r_3$ as in the statement 
of Corollary~\ref{cor:transport_cylinder+horseshoe_lip}.
Take a positive real number $\varrho$ and, for each $0\leq k\leq k_0$, 
points $a_k$ and $b_k$ in $\mathbb{R}^d$, so that the
rigid solid cylinder $C(a_k,b_k;\varrho)$ in $\mathbb{R}^d$ satisfies the following properties 
\begin{enumerate}
\item[(d)]
$C(a_k,b_k;\varrho)$ 
is contained in
$B(x^k_{k},r_3)$, 
\item[(e)]
the $C(a_k,b_k;\varrho)$ are isometric to one another,
\item[(f)]
for each $k$, 
$C(a_k,b_k;\varrho)$ and 
$C(a_{k+1},b_{k+1};\varrho)$ 
satisfy the properties (i)--(iii) of Corollary~\ref{cor:transport_cylinder+horseshoe_lip}, 
where addition in the lower index is taken modulo $k_0$.
\end{enumerate}
(Since the rigid solid cylinders are isometric, property (ii) of 
Corollary~\ref{cor:transport_cylinder+horseshoe_lip} is automatically satisfied.)
Since $f$ is bi-Lipschitz we may now apply Corollary~\ref{cor:transport_cylinder+horseshoe_lip}.
Thus, for $0\leq k<k_0$,
there exists a $C^1$-smooth diffeomorphism $\phi^k$ supported in $B^{k+1}$ such that
\begin{equation*}
\phi^k\circ \varphi_{k+1}\circ f\circ\varphi_{k}^{-1}
\colon 
\varphi_{k}(U_{k}\cap f^{-1}(U_{k+1}))
\longrightarrow 
\varphi_{k+1}(f(U_{k})\cap U_{k+1})
\end{equation*}
maps the solid cylinder 
$C(a_k,b_k;\varrho)$ 
across the solid cylinder 
$C(a_{k+1},b_{k+1};\varrho)$ 
as an $N$-branched horseshoe.
(Here, addition in the lower index is taken modulo $k_0$.)
Moreover, since the ratio of length to radius of each solid cylinder was chosen independently of $\delta$, it follows that
there exists a positive real number $c_2$, independent of $\delta$, such that, for $0\leq k< k_0$,
\begin{equation}\label{eq:phi_k-Lip}
\max \left\{
\left[\phi^k\right]_{\mathrm{Lip}, B^{k+1}}, \ 
\left[(\phi^k)^{-1}\right]_{\mathrm{Lip}, B^{k+1}}
\right\}
\leq 
c_2 N \ .
\end{equation} 
Applying Corollary~\ref{cor:cylinder_isometry_2}, 
there also exists a $C^1$-smooth diffeomorphism $\phi$, supported in $E$, such that $\phi$ maps the 
solid cylinder 
$C(a_{k_0},b_{k_0};\varrho)$ 
across the solid cylinder 
$C(a_{0},b_{0};\varrho)$.
Moreover, since the radius $c\cdot r_0$ and the length $r_0$ of $E$ are comparable, with comparability constant independent of $\epsilon$,
there exists a positive real number $c_3$, also independent of $\epsilon$, such that
\begin{equation}\label{eq:phi-Lip}
\max \left\{
\left[\phi\right]_{\mathrm{Lip}, E} \ , \ 
\left[\phi^{-1}\right]_{\mathrm{Lip}, E}
\right\}
\leq 
c_3 \ .
\end{equation} 
We are now in a position to define the perturbation. 
Namely, define $g\colon M\to M$ by
\begin{equation}\label{def:g-Holder-entropy}
g=
\left\{\begin{array}{ll}
\varphi_1^{-1}\circ\phi^0\circ \varphi_{1}\circ f\circ\varphi_{0}^{-1}\circ \phi\circ\varphi_0 & \mbox{in} \ E_M\\
\varphi_{k+1}^{-1}\circ\phi^k\circ \varphi_{k+1}\circ f & \mbox{in} \ f^{-1}(B^{k+1}_M), \ 0< k<k_0\\
f & \mbox{elsewhere}
\end{array}\right. \ .
\end{equation}
Observe that $g$ is a homeomorphism.
%Since $\phi^k$ has support in $B^k$, the map $\varphi_{k+1}^{-1}\circ \phi^k\circ \varphi_{k+1}$ has support in $B^k_M$.
%Similarly, as $\phi$ has support in $E$, the map $\varphi_0^{-1}\circ \phi\circ\varphi_0$ is supported in $E_M$.
By Corollary~\ref{cor:Lip-alpha-perturbation}, the map $g$ lies in $\mathcal{H}^{1}_{\alpha}(M)$.
%%%%%%%%%%%%%%%
\begin{comment}
%%%%%%%%%%%%%%%
Take $S=\varphi_0^{-1}(C(a_0,b_0;\varrho))$. 
As remarked in the first paragraph of the proof, since 
$g$ maps $\varphi_k^{-1}(C(a_k,b_k;\varrho))$ across $\varphi_{k+1}^{-1}(C(a_{k+1},b_{k+1};\varrho))$, it follows that 
$S$ has subcylinders $S_1, S_2,\ldots, S_{N^{k_0}}$ such that $g^{k_0}$ maps each subcylinder across $S$.
Thus property (ii) holds.
%%%%%%%%%%%%%%%
\end{comment}
%%%%%%%%%%%%%%%
Thus, when $g$ is given in charts as per~\eqref{def:g_in_charts}
we therefore have, for each index $n$, the expression
\begin{equation}\label{def:gn-Holder-entropy}
g_n=
\left\{\begin{array}{ll}
\phi^0_{n}
\circ f_n
\circ \phi_{n} 
& \mbox{in} \ \varphi_n(E_M)\\
\phi^k_{n}
\circ f_n 
& \mbox{in} \ \varphi_n(f^{-1}B^{k+1}_M), \ 0< k<k_0\\
f_n & \mbox{elsewhere}
\end{array}\right. \ ,
\end{equation}
where
\begin{equation}\label{def:phikn+phin}
\phi^k_{n}=\psi_n\circ\varphi_{k+1}^{-1}\circ\phi^k\circ \varphi_{k+1}\circ\psi_n^{-1}
\quad
\mbox{and}
\quad
\phi_{n}=\varphi_n\circ\varphi_0^{-1}\circ\phi\circ\varphi_0\circ \varphi_n^{-1} \ .
\end{equation}

\vspace{5pt}

\noindent
{\sl Size of the perturbation:\/}
Let us estimate the size of the $C^\alpha$-pseudo-distances 
between $f$ and $g$ corresponding to each of the sub-basic sets.
Fix an index $n$.
We will bound from above
\begin{equation*}
[f_n-g_n]_{\alpha,\varphi_n(K_n\cap E_M)} 
\quad
\mbox{and}
\quad
[f_n-g_n]_{\alpha,\varphi_n(K_n\cap f^{-1}B^k_M)} \ .
\end{equation*}
Consider the first quantity.
Since $E_M$ lies in a ball of radius $\delta$ or less, (b) above implies 
that either $K_n\cap E_M$ is empty or $E_M$ is contained in $\mathsf{U}_n$.
In the first case there is nothing to show.
Otherwise, by the triangle inequality 
\begin{align}\label{ineq:fn-gn-triangle}
[f_n-g_n]_{\alpha,\varphi_{n}(K_n\cap E_M)}
\leq 
[f_n-g_n]_{\alpha,\varphi_{n}(E_M)}
\leq
[f_n]_{\alpha,\varphi_{n}(E_M)}+[g_n]_{\alpha,\varphi_{n}(E_M)} \ .
\end{align}
As the mapping $\phi$ is supported in the neighbourhood $E$, 
we have the equality 
$\phi_n\circ\varphi_n(E_M)=\varphi_n(E_M)$.
From the definition~\eqref{def:f_in_charts} it follows that
$f_n\circ\varphi_n=\psi_n\circ f$,
so the H\"older Rescaling Principle (Proposition~\ref{prop:holderrescaling1}) 
applied to the expression~\eqref{def:gn-Holder-entropy}
gives
\begin{align}
\left[g_n\right]_{\alpha,\varphi_{n}(E_M)}
%&\leq 
%\left[\phi^{0}_{n}\right]_{\mathrm{Lip},f_n\circ\phi_{n}\circ\varphi_{n}(E_M)}
%\left[f_n\right]_{\alpha,\phi_{n}\circ\varphi_{n}(E_M)}
%\left[\phi_{n}\right]_{\mathrm{Lip},\varphi_{n}(E_M)}^\alpha\\
%&=
%\left[\phi^{0}_{n}\right]_{\mathrm{Lip},f_n\circ\varphi_{n}(E_M)}
%\left[f_n\right]_{\alpha,\varphi_{n}(E_M)}
%\left[\phi_{n}\right]_{\mathrm{Lip},\varphi_{n}(E_M)}^\alpha\\
&\leq
\left[\phi^{0}_{n}\right]_{\mathrm{Lip},\psi_n\circ f(E_M)}
\left[f_n\right]_{\alpha,\varphi_{n}(E_M)}
\left[\phi_{n}\right]_{\mathrm{Lip},\varphi_{n}(E_M)}^\alpha \ . \label{ineq:gn-EM}
\end{align}
Consider the first and last factor on the right-hand side.
Applying the H\"older Rescaling Principle (Proposition~\ref{prop:holderrescaling1}) 
again to the expression on the right-hand side in~\eqref{def:phikn+phin}
and recalling that $E=\varphi_0(E_M)$, 
we find that
\begin{align*}
\left[\phi_{n}\right]_{\mathrm{Lip},\varphi_n(E_M)}
&\leq
\left[\varphi_n\circ\varphi_0^{-1}\right]_{\mathrm{Lip},\varphi_0(E_M)}
\left[\phi\right]_{\mathrm{Lip},E}
\left[\varphi_0\circ \varphi_n^{-1}\right]_{\mathrm{Lip},\varphi_n(E_M)} \ .
\end{align*}
To this expression we apply 
inequality~\eqref{eq:phi-Lip}, 
and,
since $E_M\subset\mathsf{U}_n\cap\mathsf{W}_0$, 
we may also apply 
inequality~\eqref{ineq:transition-Lip2}. 
Consequently $\left[\phi_{n}\right]_{\mathrm{Lip},\varphi_n(E_M)}$ is bounded from above by a constant independent of $\epsilon$.
Next, using the expression on the left-hand side in~\eqref{def:phikn+phin}, 
the H\"older Rescaling Principle (Proposition~\ref{prop:holderrescaling1}) 
implies that
\begin{align*}
\left[\phi^0_{n}\right]_{\mathrm{Lip},\psi_n(B^1_M)}
&\leq 
\left[\psi_n\circ\varphi_{1}^{-1}\right]_{\mathrm{Lip},\varphi_1(B^1_M)}
\left[\phi^{0}\right]_{\mathrm{Lip},B^1}
\left[\varphi_{1}\circ\psi_n^{-1}\right]_{\mathrm{Lip},\psi_n(B^1_M)} \ .
\end{align*}
To this expression we apply
inequality~\eqref{eq:phi_k-Lip}, and since $B^1_M\subset f(\mathsf{U}_n)\cap\mathsf{W}_1$ 
%which follows as $f^{-1}(B^1_M)\subset E_M\subset \mathsf{U}_n$, 
we may apply inequality~\eqref{ineq:transition-Lip1}.
Thus by the First H\"older Gluing Principle (Proposition~\ref{prop:Holder_gluing_i})
%since the map $\phi^0$ is supported on $B^1$, the map $\phi^0_n$ is supported on $\psi_n(B^1_M)$.
%\begin{equation}
%\left[\phi^0_n\right]_{\mathrm{Lip},\psi_n\circ f(E_M)}
%\leq 
%\max\left\{1,\left[\phi^0_n\right]_{\mathrm{Lip},\psi_n(B^1_M)}\right\}
%\end{equation}
we find that $\left[\phi^0_{n}\right]_{\mathrm{Lip},\psi_n\circ f(E_M)}$ is bounded 
from above, also by a constant independent of $\epsilon$.
By the condition (c) above, this therefore implies, together with~\eqref{ineq:fn-gn-triangle} and~\eqref{ineq:gn-EM}, 
that there exists a positive
real number $c_5$, independent of $\epsilon$, such that
\begin{equation*}
[f_n-g_n]_{\alpha,\varphi_n(K_n\cap E_M)}\leq c_5 \epsilon \ .
\end{equation*}
The same argument shows that, after increasing $c_5$ by a factor independent of $\epsilon$ if necessary, that for $0\leq k<k_0$,
\begin{equation*}
[f_n-g_n]_{\alpha,\varphi_n(K_n\cap f^{-1}B^k_M)}\leq c_5 \epsilon \ .
\end{equation*}
Since $f_n=g_n$ on $\varphi_{n}(K_n)\setminus \varphi_n(E_M\cup\bigcup_k f^{-1}(B^k_M))$ 
the Second H\"older Gluing Principle (Proposition~\ref{prop:Holder_gluing_ii}) 
together with Claim 2(5) implies that there exists a positive real number $c_6$, independent of $\epsilon$, such that
\begin{align*}%\label{ineq:fn_vs_gn-Holder_entropy-gluing}
[f_n-g_n]_{\alpha,\varphi_{n}(K_n)}
&\leq 
%c_5\max_k\left\{
%[f_n-g_n]_{C^\alpha,\varphi_{n}(K_n\cap E_M)},
%[f_n-g_n]_{C^\alpha,\varphi_{n}(K_n\cap B^k_M)}
%\right\}
c_6c_5\epsilon \ .
\end{align*}
Also notice that the $C^0$-distance between $f$ and $g$ can be made arbitrarily small provided that 
$\epsilon$, and thus $\delta$, is sufficiently small.
Combining this observation with the preceding inequality 
%~\eqref{ineq:fn_vs_gn-Holder_entropy-gluing}
we therefore find that, for $\epsilon$ sufficiently small,
\begin{equation*}
\|f_n-g_n\|_{C^\alpha(\varphi_n(K_n),\mathbb{R}^d)}<\epsilon_n \ .
\end{equation*}
The same argument applied to the inverse mappings shows that, 
after shrinking $\epsilon$ if necessary, 
we also have the inequality
\begin{equation*}
\|f_n^{-1}-g_n^{-1}\|_{C^\alpha(\psi_n(L_n),\mathbb{R}^d)}<\epsilon_n \ .
\end{equation*}
Thus, taking the minimum of all such $\delta$ and $\epsilon$ over each index $n$, 
the map $g$ must lie in the common intersection of all the sub-basic sets given above.
Hence the theorem is shown.
\end{proof}
Analogously to the homeomorphism case~\cite{Yano80} we also get the following.
\begin{corollary}\label{cor:generic_not_conjugate-Holder}
Let $M$ be a compact manifold of dimension at least two.
Let $0\leq \alpha<1$.
A generic homeomorphism in $\mathcal{H}^{1}_{\alpha}(M)$ 
is not conjugate to any diffeomorphism (or any bi-Lipschitz homeomorphism).
\end{corollary}
Recall that horseshoes possess (unique) measures of maximal entropy.
Observe that in the proof of Theorem~\ref{thm:open+dense_entropy>logN-little_Holder}, 
the worst that can happen is that the recurrent point being used already lies
in a horseshoe.
However, as the perturbation being used is arbitrarily small, if the horseshoe of the original map has $N$ branches, we may assume that the perturbed map
has a horseshoe with at least $N$ branches.
Thus, considering all possible sums of these measures over all possible horseshoes, 
we get the following corollary.
\begin{corollary}\label{cor:generic_measures_max_entropy-Holder}
Let $M$ be a compact manifold of dimension at least two.
Let $0\leq \alpha<1$.
A generic homeomorphism in $\mathcal{H}^{1}_{\alpha}(M)$ has uncountably many measures of maximal entropy.
\end{corollary}
Next recall the following.
Let $X$ be a compact topological space.
%Denote the set of Borel sets by $\mathcal{B}$.
Let $\mathcal{M}(X)$ denote the set of Borel probability measures on $X$.
Given $f\in C^0(X,X)$, let $\mathcal{M}(X,f)$ denote the set of $f$-invariant Borel probability measures.
For any $\phi\in C^0(X,\mathbb{R})$, let $P(f,\phi)$ denote the {\it pressure} of $f$ wth respect to $\phi$.
Then the Variational Principle~\cite[Section 9]{WaltersBook} states that
\begin{equation*}
P(f,\phi)=\sup_{\mu\in \mathcal{M}(X,f)}\left( h_\mu(f)+\int \phi\,d\mu\right) \ .
\end{equation*}
Recall that $\mu\in\mathcal{M}(X,f)$ is an {\it equilibrium state} for $(f,\phi)$ if
\begin{equation*}
P(f,\phi)=h_\mu(f)+\int \phi\,d\mu \ .
\end{equation*}
\begin{lemma}
If $h_\mathrm{top}(f)=+\infty$ then the set of equilibrium states of $(f,\phi)$ is independent of $\phi\in C^0(X,\mathbb{R})$.
\end{lemma}
\begin{proof}
By~\cite[Section 9.2]{WaltersBook}, 
$h_\mathrm{top}(f)=+\infty$ implies that $P(f,\phi)=+\infty$, for all $\phi\in C^0(X,\mathbb{R})$.
But $P(f,\phi)=+\infty$ implies that any equilibrium state $\mu$ must satisfy either 
$h_\mu(f)=+\infty$, 
or $\int \phi\,d\mu=+\infty$.
However, $\mu$ is a probability measure, so and continuous function $\phi$ satisfies
$|\int \phi\,d\mu|\leq |\phi|_X |\int\,d\mu|=|\phi|_X<\infty$.
Consequently, any equilibrium state $\mu$ must satisfy $h_\mu(f)=+\infty$. 
Conversely, any $\mu\in\mathcal{M}(X,f)$ with $h_\mu(f)=+\infty$
is an equilibrium state for any $\phi\in C^0(X,\mathbb{R})$.
The result follows.
\end{proof}
Combining this with Theorem~\ref{thm:generic_little-Holder_homeomorphism} 
we therefore get the following corollary.
\begin{corollary}\label{cor:generic_eqm_states-Holder}
Let $M$ be a compact manifold of dimension at least two.
Let $0\leq \alpha <1$.
For a generic homeomorphism $f$ in $\mathcal{H}^1_\alpha(M)$, 
the set of equilibrium states of $(f,\phi)$ is independent of $\phi\in C^0(M,\mathbb{R})$.
In fact, generically the set of equilibrium states, for any $\phi\in C^0(M,\mathbb{R})$, 
coincides with the set of measures of maximal entropy.  
\end{corollary}

\section{Part II -- Sobolev Mappings}
\subsection{Preliminaries.}\label{sec:Sobolev_prelim} 
Let us recall some basic definitions and facts about Sobolev functions and maps. 
For details on the material here 
we strongly recommend \cite{GR90}, \cite{MRSY} and~\cite{Zi}. Here and throughout, all open domains in Euclidean spaces will 
be assumed to have piecewise-smooth boundaries. 

\subsubsection*{Sobolev functions} 
Let $\Omega\subseteq \mathbb{R}^d$ be open, and let $k\in \mathbb{N}$ and $1\leq p<\infty$. 
Recall that a measurable function $u\colon\Omega\to \mathbb{R}$ is in the Sobolev class $W^{k,p}(\Omega)$ if $u$ has 
distributional partial derivatives of all orders up to $k$ and, for each multi-index 
$\alpha=(\alpha_1,\alpha_2,\ldots,\alpha_d)\in \mathbb{N}^d$ with $|\alpha|=\sum_{i=1}^d \alpha_i\leq k$, the 
corresponding distributional partial derivative 
$D^{\alpha}u$ belongs to $L^p(\Omega)$. 
The space $W^{k,p}(\Omega)$ is a Banach space under the norm
\begin{equation*}
\|u\|_{k,p}\;=\; \sum_{|\alpha|\leq k} \|D^{\alpha}u\|_p \ ,
\end{equation*}
where $\|\cdot\|_p$ denotes the standard $L^p$-norm in $\Omega$. 
It is known that every Sobolev function $u$ is 
absolutely continuous on lines (ACL), 
{\it i.e.\/}, 
its restriction to Lebesgue almost every straightline 
(parallel to some coordinate axis) is absolutely continuous~\cite[Section 1.1.3]{Maz'yaBook}. 
It is also known that $u$ is differentiable Lebesgue almost everywhere in $\Omega$ provided that $p>d$.
(This was proved for $d=2$ by Cesari~\cite{Cesari1941} and for arbitrary $d$ by Calder{\'o}n~\cite{Calderon1951}.) 

\subsubsection*{Sobolev maps}
Let us consider a measurable map $f\colon\Omega\to \mathbb{R}^d$. 
We say that $f$ is a {\it Sobolev map\/} in 
the class $W^{k,p}$ if, writing $f=(f_1,f_2,\ldots,f_d)$, each 
component $f_i\in W^{k,p}(\Omega)$. 
Note that such a map has a {\it formal\/} Jacobian matrix 
$Df(x)=(\partial_{x_j}f_i(x))_{1\leq i,j\leq d}$ defined 
at Lebesgue almost every point $x\in \Omega$. 

The space of Sobolev maps in the class $W^{k,p}$, which we denote by $W^{k,p}(\Omega,\mathbb{R}^d)$, can be made 
into a Banach space in several equivalent ways. One natural way is to define, for $f\in W^{k,p}(\Omega,\mathbb{R}^d)$, 
its Sobolev norm by 
$\|f\|_{W^{k,p}(\Omega,\mathbb{R}^d)} = \sum_{i=1}^{d} \|f_i\|_{k,p}$, where $f_i$, $i=1,\ldots,d$, are the components 
of $f$. With this norm $W^{k,p}(\Omega,\mathbb{R}^d)$ is a Banach space.

\subsubsection*{Continuous Sobolev maps}
We are not interested in {\it all\/} Sobolev maps, 
only in those that are {\it continuous\/} up to the boundary. 
Let us write 
\begin{equation*}
\mathbb{W}^{k,p}\left(\Omega, \mathbb{R}^d\right) 
= 
W^{k,p}\left(\Omega,\mathbb{R}^d\right) \cap C^0\left({\overline{\Omega}}, \mathbb{R}^d\right) \ .
\end{equation*}
We need a topology on this space. 
Rather than giving a general definition covering all cases, we 
restrict ourselves to the cases when $k=1$ and $p\geq 1$ is arbitrary; these are the only cases that will be 
relevant in the present paper. We define a norm in 
$\mathbb{W}^{1,p}\left(\Omega,\mathbb{R}^d\right)$ 
as follows. 
First, given a $d\times d$ matrix $A=(a_{ij})$, we define its {\it norm\/} to be $|A|=\sum_{i,j=1}^{d} |a_{ij}|$. 
Given 
$f\in \mathbb{W}^{1,p}\left(\Omega,\mathbb{R}^d\right)$, 
let 
\begin{equation*}
 \|f\|_{\mathbb{W}^{1,p}(\Omega,\mathbb{R}^d)}
\;=\; 
\|f\|_{C^0(\Omega)} 
+\left(\int_{\Omega} |Df(x)|^p\,d\mu(x)\right)^{\frac{1}{p}} \ .
\end{equation*}
This defines a norm, and with this norm $\mathbb{W}^{1,p}\left(\Omega,\mathbb{R}^d\right)$ is a Banach space.

\subsubsection*{Sobolev homeomorphisms}
Let $f\colon\Omega\to f(\Omega)\subseteq \mathbb{R}^d$ 
be an orientation-preserving homeomorphism 
(continuous up to the boundary of $\Omega$), and 
suppose that $f\in \mathbb{W}^{1,p}\left(\Omega,\mathbb{R}^d\right)$. 
Observe that differentiability properties of $f$, or more generally for open mappings, are better than for general mappings.
More precisely, $f$ is differentiable almost everywhere in $\Omega$ provided that $p>d-1$.
(This was proved for $d=2$ by Gehring and Lehto~\cite{GehringLehto}, and for arbitrary $d$ by V{\"a}is{\"a}l{\"a}~\cite{Vaisala1965}.
In fact, in dimension two Gehring and Lehto showed the result also holds when $p=1$.) 
We denote by $J_f(x)=\mathrm{det}Df(x)$ the {\it Jacobian determinant\/} of $f$ at $x$ in $\Omega$. 
%%%%%%%%%%%%%%%%%%%%%%%%%%%%%%%%%%%%%%%%%%%%%%
\begin{comment}
When can it be shown that $J_f(x)=\mathrm{det}Df(x)$ is the Jacobian of $f$ 
in the sense of measure theory?
\end{comment}
%%%%%%%%%%%%%%%%%%%%%%%%%%%%%%%%%%%%%%%%%%%%%%

The {\it chain rule\/} for Sobolev maps 
does not always hold. For instance, for 
$d>2$ there exist homeomorphisms 
$f\colon (0,1)^d\to (0,1)^d$ 
such that $f$ and $f^{-1}$ are of class $\mathbb{W}^{1,d-1}$ 
but both $f$ and $f^{-1}$ have zero Jacobian 
matrix at Lebesgue almost every point (see \cite{DHS}). 
%Hence $D(f\circ f^{-1})=\mathrm{id}\neq 0=Df(f^{-1})D(f^{-1})$.
However, in the present paper, we will 
need the chain rule for the composition 
of two Sobolev homeomorphisms only in 
the case when one of them is a 
{\it diffeomorphism\/}. 
In this case, the following result is available.
\begin{lemma}[Chain rule]\label{chainrule}
Let $U, V\subseteq \mathbb{R}^d$ be open domains. 
Let $f\in \mathbb{W}^{1,p}\left(U, \mathbb{R}^d\right)$ be a homeomorphism onto its image, 
and let $\phi\colon V\to \mathbb{R}^d$ be a $C^1$-diffeomorphism onto its image. 
Assume that either (a) $d>2$ and $p>d-1$; or (b) $d=2$ and $p\geq 1$. 
 Then the following statements are true.
 \begin{enumerate}
  \item[(i)] If $\phi(V)\subseteq U$, then $f\circ \phi\in \mathbb{W}^{1,p}\left(V,\mathbb{R}^d\right)$, the composition $f \circ \phi$ is differentiable 
  almost everywhere and 
  \begin{equation*}\label{chain1}
   D(f\circ\phi)(x) \;=\; Df(\phi(x))D\phi(x) \ \ \ \textrm{for Lebesgue a.e.}\ \ \ x\in V\ .
  \end{equation*}
 \item[(ii)] If $f(U)\subseteq V$, then $\phi\,\circ f\in \mathbb{W}^{1,p}\left(U,\mathbb{R}^d\right)$, the composition $\phi\,\circ f$ is differentiable 
  almost everywhere and 
  \begin{equation}\label{chain2}
   D(\phi\circ f)(x) \;=\; D\phi(f(x))D\phi(x) \ \ \ \textrm{for Lebesgue a.e.}\ \ \ x\in U\ .
  \end{equation}
 \end{enumerate}

\end{lemma}
Given a $C^1$-diffeomorphism $\phi$ 
it is known that both pre- and post-composition operators 
$f\mapsto f\circ \phi$ and 
$f\mapsto \phi\circ f$ 
map $\mathbb{W}^{1,p}$ to $\mathbb{W}^{1,p}$ (in the appropriate domains). 
See, for instance, \cite{ArendtKreuter}.
For the proof of part (i), 
one combines the chain rule for Sobolev {\it functions\/} 
as stated, say, in~\cite[Theorem 2.2.2, p. 52]{Zi} 
with the fact that $D(f\circ \phi)$ exists Lebesgue almost everywhere. 
For the proof of part (ii), 
note that the set of points $x$ where {\it both\/} 
sides of~\eqref{chain2} are defined has full measure; 
then one may write the first-order Taylor expressions 
for both $f$ and $\phi\circ f$ at $x$ in the direction $v$, 
substitute them into the expression $\phi\circ f(x+tv)$, 
and compare the resulting expressions 
(after reminding oneself that $\phi$ is differentiable {\it everywhere}). 

\subsubsection*{Spaces of bi-Sobolev homeomorphisms}
Let $\Omega,\Omega^*\subset\mathbb{R}^d$ be bounded open 
sets with piecewise-smooth boundary and let $1\leq p,p^*< \infty$.
Denote by $\mathcal{S}^{p,p^*}\left(\Omega,\Omega^*\right)$ 
the space of orientation-preserving homeomorphisms 
$f\colon\Omega\to\Omega^*$
such that 
$f\in \mathbb{W}^{1,p}\left(\Omega,\mathbb{R}^d\right)$, 
$f^{-1}\in \mathbb{W}^{1,p^*}\left(\Omega^*,\mathbb{R}^d\right)$, 
and $f$ and $f^{-1}$ extend continuously to the boundaries 
of $\Omega$ and $\Omega^*$ respectively.
Provided that $\Omega$ and $\Omega^*$ are chosen so that 
$\mathcal{S}^{p,p^*}\left(\Omega,\Omega^*\right)$ is non-empty,
$\mathcal{S}^{p,p^*}\left(\Omega,\Omega^*\right)$ is a complete 
metric space when endowed with the distance function
\begin{equation*}
\rho(f,g)=
\bigl\|f-g\bigr\|_{\mathbb{W}^{1,p}(\Omega,\mathbb{R}^d)}
+\bigl\|f^{-1}-g^{-1}\bigr\|_{\mathbb{W}^{1,p^*}(\Omega^*,\mathbb{R}^d)}
\end{equation*}
for all $f,g\in\mathcal{S}^{p,p^*}(\Omega,\Omega^*)$. 
In particular, $\mathcal{S}^{p,p^*}(\Omega,\Omega^*)$ is a Baire space.

%%%%%%%%WHITNEY TOPOLOGY%%%%%%%%%%%%
As in the H\"older case, there are several 
ways to define Sobolev classes on spaces 
more general than Euclidean domains.
This can be done when the space is a manifold 
with a smooth Riemannian metric 
(or more generally a smooth connection), and
it can also be done for embedded manifolds 
in $\mathbb{R}^n$ (see, {\it e.g.\/},~\cite{SchoenUhlenbeck}).
Instead, we will construct and use a 
topology analogous to that constructed 
in the H\"older case in \S~\ref{sec:Holder_prelim}.

First consider bi-Sobolev homeomorphisms 
between manifolds $M$ and $N$.
Let $1\leq p, p^*<\infty$.
Denote by $\mathbb{W}^{1,p}(M,N)$ 
the space of maps $f$ from $M$ to $N$ such that, 
for any pair of charts $(U,\varphi)$ on $M$ and $(V,\psi)$ on $N$, 
the map $\psi\circ f\circ\varphi$ is $\mathbb{W}^{1,p}$ on $\psi(U\cap f^{-1}(V))$.
Let $\mathcal{S}^{p,p^*}(M,N)$ denote the space of homeomorphisms $f$ from $M$ to $N$ such that
$f\in \mathbb{W}^{1,p}(M,N)$ and $f^{-1}\in \mathbb{W}^{1,p^*}(N,M)$.
In the case when $M$ and $N$ coincide we 
denote this space by $\mathcal{S}^{p,p^*}(M)$.

We define the {\it (weak) $(p,p^*)$-Sobolev-Whitney topology\/} on $\mathcal{S}^{p,p^*}(M,N)$
analogously to the (weak) $C^\alpha$-Whitney topology constructed in \S~\ref{sec:Holder_prelim}.
Namely we define sets 
$\mathcal{N}_{\mathbb{W}^{1,p},\mathbb{W}^{1,p^*}}(f;(U,\varphi),(V,\psi),K,L,\epsilon)$ 
in exactly the same way as we defined the sub-basic sets
$\mathcal{N}_{C^\alpha}(f;(U,\varphi),(V,\psi),K,L,\epsilon)$, 
except that we replace the $C^\alpha$-distance between maps and their inverses
with the $\mathbb{W}^{1,p}$-distance between maps and $\mathbb{W}^{1,p^*}$-distance between their inverses.
Observe that the compact sets $K$ and $L$ used in this construction are required to be the closure of open sets
whose boundaries are piecewise-smooth.
The collection of sets of the form 
\begin{equation*}
\mathcal{N}_{\mathbb{W}^{1,p},\mathbb{W}^{1,p^*}}(f;(U,\varphi),(V,\psi),K,L,\epsilon)
\end{equation*}
then forms a subbasis for a topology which we call the 
{\it (weak) $(p,p^*)$-Sobolev-Whitney topology}.

As in the H\"older case, the $(p,p^*)$-Sobolev-Whitney topology is Hausdorff and satisfies the following.
\begin{proposition}
For $1\leq p,p^*<\infty$,
and each pair of smooth compact manifolds $M$ and $N$ (possibly with boundary),
the space $\mathcal{S}^{p,p^*}(M,N)$, 
endowed with the weak $(p,p^*)$-Sobolev-Whitney topology, 
satisfies the Baire property. 
\end{proposition}

\subsubsection*{Further comments}
%%%%%%%%%%%%%%%%%%%%%%%%%%%%%%%%%%%%%%%%%%%%%%
The following facts, concerning certain special classes of Sobolev homeomorphisms, 
are worth mentioning even though they will not be used in the present paper.
The homeomorphism $f$ is said to be a map of {\it finite distortion\/} if 
the quotient $K_f(x)=|Df(x)|^d/J_f(x)$ is finite almost everywhere in $\Omega$. When $K=\|K_f\|_{\infty}<\infty$, 
we say that $f$ is a {\it $K$-quasiconformal homeomorphism\/}. Thus a $K$-quasiconformal homeomorphism 
satisfies the inequality $|Df(x)|^d \leq KJ_f(x)$ almost everywhere. 
An inequality in the opposite direction is possible for general homeomorphisms in $\mathbb{W}^{1,p}(\Omega,\mathbb{R}^d)$ 
for sufficiently large $p$, as shown by the following easy lemma. 

\begin{lemma}\label{detJac}
If $f\in \mathbb{W}^{1,d}(\Omega,\mathbb{R}^d)$ is a homeomorphism, then the determinant Jacobian $J_f$ belongs to $L^1(\Omega)$. 
%In particular, $f$ satisfies Lusin's $N$-property. 
\end{lemma}

\begin{proof}
Writing $f_{ij}=\partial_{x_j}f_i$ for the components of the matrix $Df(x)$, we have by definition of the determinant
\begin{equation*}
J_f(x)
\;=\; 
\mathrm{det}Df(x)
\;=\; 
\sum_{\sigma\in S_d} (-1)^{\mathrm{sign}(\sigma)}f_{1\sigma(1)}(x)f_{2\sigma(1)}(x)\cdots f_{d\sigma(d)}(x) \ .
\end{equation*}
Taking absolute values on both sides and taking into account that $|f_{ij}(x)|\leq |Df(x)|$, we deduce that 
\begin{equation}\label{ineq:W1,d-jacobian}
|J_f(x)|
\;\leq\; 
d! \, |Df(x)|^d \ .
\end{equation}
Integrating both sides, we deduce that $\|J_f\|_{L^1(\Omega)}\;\leq\; d! \, \|Df\|_{L^d(\Omega)}^d < \infty$, and hence the result. 
\end{proof}

We say that a homeomorphism $f$ satisfies {\it Lusin's $N$-property\/}  
if $f$ maps Lebesgue null-sets onto Lebesgue null-sets. 
Suppose we know that $Df(x)$ exists Lebesgue almost everywhere, that $J_f$ is integrable, and that 
\begin{equation}\label{jacobian} 
\mu(f(E))\leq \int_{E} J_f\,d\mu \ 
\end{equation}
for each measurable set $E$ in the domain of $f$. Then, clearly, $f$ has Lusin's $N$-property.  
It has been proved by Reshetnyak in \cite{Re} that every homeomorphism 
$f\in \mathbb{W}^{1,d}(\Omega,\mathbb{R}^d)$ has Lusin's $N$-property. 
Despite appearances, this non-trivial result 
does not follow directly from Lemma \ref{detJac}, since we don't know a priori that~\eqref{jacobian} holds. 
%%%%%%%%%%%%%%%%%%%%%%%%%%%%%%%%%%%%
%
%
%
%
%
%
%
%
%
%%%%%%%%%%%%%%%%%%%%%%%%%%%%%%%%%%%%%%%%%%%%%%%%%%%%%%%%%%%%%%%
\subsection{The Sobolev Closing Lemma}\label{sec:Sobolev_Closing}
As in \S~\ref{sec:holderclosing}, we consider spaces of 
homeomorphisms on smooth compact manifolds of dimension greater than one.
Here we prove a version of Pugh's $C^1$-Closing Lemma for bi-Sobolev mappings.
\begin{theorem}[Sobolev Closing Lemma]\label{thm:sobolevclosing}
Let $M$ be a smooth compact manifold of dimension $d$.
For $d=2$ and $1\leq p,p^*<\infty$; or $d>2$ and $d-1< p,p^*<\infty$, the following holds:
Take $f\in \mathcal{S}^{p,p^*}(M)$ and
let $y$ be a non-wandering point of $f$.
For 
each neighbourhood $W$ of $y$ in $M$ and 
each neighbourhood $\mathcal{N}$ of $f$ in $\mathcal{S}^{p,p^*}(M)$ 
there exists 
$g$ in $\mathcal{N}$ and a point $x$ in $W$ such that $x$ is a periodic point of the map $g$.
\end{theorem}
%
%%%NEW PROOF%%%
\begin{proof}
Our approach will be the same as in the proof of the H\"older Closing Lemma
(Theorem~\ref{lem:Holder_closing}).
The first significant difference is that, 
given a finite collection of sub-basic sets for the $(p,p^*)$-Sobolev-Whitney topology,
\begin{equation*}
\mathcal{N}_{\mathbb{W}^{1,p},\mathbb{W}^{1,p^*}}(f;(U_n,\varphi_n),(V_n,\psi_n),K_n,L_n,\varepsilon(n))
\end{equation*}
rather than constructing a single perturbation and then showing it lies in each sub-basic set,
we will construct a sequence of perturbations converging to our original map in the $C^0$-topology, 
such that the sequence eventually lands inside each of our (finitely many) sub-basic sets.

\vspace{5pt}

\noindent
{\sl Setup:\/}
We adopt the notation of the proof of Theorem~\ref{lem:Holder_closing}.
In particular, 
$f_n$, $\mathsf{U}_n$, $\mathsf{V}_n$, $\mathsf{W}_0$, etc. 
are as before. 
Given a sequence of perturbations $g_m$ of $f$, we denote 
by $g_{m,n}$ the map $g_m$ in the pair of charts $(U_n,\varphi_n)$ and $(V_n,\psi_n)$.

As mentioned above, we take a decreasing sequence $\epsilon_{m}$ of positive real numbers tending to zero, 
denoting the order of the size of the $m$th perturbation (instead of just $\epsilon$ as in the H\"older case).
Similarly, we take a decreasing sequence $\delta_{m}$ of positive real numbers converging to zero,
denoting the size of the support of the local perturbation (instead of just $\delta$). 
We assume that all $\delta_{m}$ are chosen so that
\begin{enumerate}
\item[(a)]
$B_M(y,\delta_{m})$ is contained in $W\cap\mathsf{W}_0$
\item[(b)]
for each $n$, 
$\mathsf{U}_n$ contains a $2\delta_{m}$-neighbourhood of the compact set $K_n$, and
$\mathsf{V}_n$ contains a $2\delta_{m}$-neighbourhood of the compact set $L_n$
\item[(c)]
For any $n$,
given an arbitrary ball $B$ in $\mathsf{U}_n$ of radius $\delta_m$ or less,
the image $f_n\circ\varphi_n(B)=\psi_n\circ f(B)$ has diameter $\epsilon_m$ or less.
\item[(d)]
For any $n$,
given an arbitrary ball $B$ in $\mathsf{V}_n$ of radius $\delta_m$ or less,
the image $f_n^{-1}\circ\psi_n(B)=\varphi_n\circ f^{-1}(B)$ has diameter $\epsilon_m$ or less.
\end{enumerate}
Observe that (c) and (d) are possible by compactness of $\mathsf{U}_n$ and $\mathsf{V}_n$ respectively.

\vspace{5pt}

\noindent
{\sl Construction of the perturbation:\/}
For each $m$, the construction of the perturbation $g_m$ is identical to the H\"older case.
Namely, by Claim 1 in the proof of the H\"older Closing Lemma (Theorem~\ref{lem:Holder_closing}),
for each $m$ there is a point $x_m^0$ and integer $k_m$ such that $x_m^0$ and $x_m^{k_m}=f^{k_m}(x_m^0)$ lie in $B_M(y,\delta_m)$ and 
no other points in the orbit segment $x_m^0, x_m^1,\ldots,x_m^{k_m}$ lie in $B_M(y,\delta_m)$.
We set $E_m$, $E'_m$, $E_{m,M}$, and $E_{m,M}'$ as before and define $\phi_m$ via Lemma~\ref{lem:E-perturbation}.
Then we define
\begin{equation*}
g_m=\left\{
\begin{array}{ll}
f\circ \varphi_0^{-1}\circ\phi_m\circ \varphi_0 & \mbox{in} \ E_{m,M}\\
f & \mbox{elsewhere}
\end{array}\right. \ .
\end{equation*}
As in the H\"older case, it will be important to note that, for each $n$,
\begin{equation*}
g_{m,n}=\left\{\begin{array}{ll}
f_n\circ\phi_{m,n} & \mbox{in} \ \varphi_n(E_{m,M})\\
f_n & \mbox{elsewhere}
\end{array}\right. \ .
\end{equation*}
where $\phi_{m,n}$ denotes $\phi_m$ expressed in the chart $(U_n,\varphi_n)$, {\it i.e.\/},
\begin{equation*}
\phi_{m,n}=\varphi_n\circ\varphi_0^{-1}\circ\phi_m\circ\varphi_0\circ\varphi_n^{-1} \ .
\end{equation*}
It is clear that $g_m$ is a homeomorphism.
Since composition of a Sobolev map with a smooth map is again Sobolev, it also follows that $g_m$ lies in $\mathcal{S}^{p,p^*}(M)$.
By the same argument as in the H\"older case, 
$g_m^{k_m}(x_m^{k_m})=x_m^{k_m}$.
Thus it just remains to show that $g_m$ will lie in $\mathcal{N}$ for $m$ sufficiently large.

\vspace{5pt}

\noindent
{\sl Size of the perturbation:\/}
For each $n$, it suffices to estimate the
\begin{enumerate}
\item[(i)]
$\mathbb{W}^{1,p}$-pseudo-distance between $f_n$ and $g_{m,n}$ on $\Omega_n=\varphi_n(K_n)$
\item[(ii)]
$\mathbb{W}^{1,p^*}$-pseudo-distance between $f_n^{-1}$ and $g_{m,n}^{-1}$ on $\Omega_n^*=\psi_n(L_n)$
\end{enumerate}
Below we will construct a subsequence $m_1,m_2,\ldots$ of the natural numbers such that
\begin{equation*}
\lim_{j\to\infty}\| f_n-g_{m_j,n} \|_{\mathbb{W}^{1,p^*}(\Omega_n,\mathbb{R}^d)}
=0
=\lim_{j\to\infty}\| f_n^{-1}-g_{m_j,n}^{-1} \|_{\mathbb{W}^{1,p^*}(\Omega_n^*,\mathbb{R}^d)} \ .
\end{equation*}
Applying this inductively for each $n$, taking a subsequence at each step, will then give the result.
 
Consider (i).
Observe that $f$ and $g_m$ only differ on $E_{m,M}$, and hence that the sets $f(E_{m,M})$ and $g_{m}(E_{m,M})$ agree.
Since $E_{m,M}$ is contained in $B_M(y,\delta_m)$, {\it i.e.\/}, a ball of radius $\delta_m$, by (b) above, 
either 
$E_{m,M}$ is contained in %a $2\delta_m$-neighbourhood of $K_n$, which is contained in 
$\mathsf{U}_n$, 
or 
$E_{m,M}$ is disjoint from $K_n$.
In the second case there is nothing to prove.
Thus we focus on the first case.
Recall that
\begin{equation*}
\|f_n-g_{m,n}\|_{\mathbb{W}^{1,p}(\Omega_n,\mathbb{R}^d)}
=
\|f_n-g_{m,n}\|_{C^0(\Omega_n)}
+
\left(\int_{\Omega_n}|Df_n-Dg_{m,n}|^p \,d\mu\right)^{\frac{1}{p}} \ .
\end{equation*}
As $f$ and $g_m$ differ only on $E_{m,N}$ it follows that $f_n$ and $g_{m,n}$ differ only on $\varphi_n(E_{m,M})$.
Since $E_{m,M}$ is contained in a ball of radius $\delta_m$, it follows from (c) above that
$f_m(\varphi_n(E_{m,M}))$, and hence $g_{m,n}(\varphi_n(E_{m,M}))$, are contained in some ball of radius $\epsilon_m$.
Consequently
\begin{equation*}
\|f_n-g_{m,n}\|_{C^0(\Omega_n)}\leq \epsilon_m \ .
\end{equation*}
It therefore suffices to show that there is some subsequence $J$ of the natural numbers such that
\begin{equation}\label{eq:W1p-semi-norm}
\liminf_{J\ni m\to\infty}\int_{\Omega_n}|Df_n-Dg_{m,n}|^p \,d\mu=0 \ .
\end{equation}
In fact, it will be slightly easier to show this on the slightly larger set $\varphi_n(\mathsf{U}_n)$.
Namely we will show that for some subsequence $J$ of natural numbers
\begin{equation}\label{eq:W1p-semi-norm_ii}
\liminf_{J\ni m\to\infty}\int_{\varphi_n(\mathsf{U}_n)}|Df_n-Dg_{m,n}|^p \,d\mu=0 \ .
\end{equation}
Obviously, \eqref{eq:W1p-semi-norm} follows directly from~\eqref{eq:W1p-semi-norm_ii}.
The chain rule (Lemma~\ref{chainrule} (i)) gives
\begin{equation*}
Dg_{m,n}(x)=Df_n(\phi_{m,n}(x))D\phi_{m,n}(x)
\end{equation*}
for Lebesgue almost every $x\in\varphi_n(\mathsf{U}_n)$.
Let $u_{ij}^{(n)}$, $v_{ij}^{(m,n)}$ and $w_{ij}^{(m,n)}$, $1\leq i,j\leq d$, denote the entries 
of the matrices $Df_n$, $D\phi_{m,n}$ and $Dg_{m,n}$ so that, for Lebesgue almost every $x\in\varphi_n(\mathsf{U}_n)$
\begin{equation*}
w_{ij}^{(m,n)}(x)=\sum_{1\leq k\leq d} u_{ik}^{(n)}(\phi_{m,n}(x))v_{kj}^{(m,n)}(x) \ .
\end{equation*}
Thus, to show that the integral in~\eqref{eq:W1p-semi-norm_ii} can be made arbitrarily small we must show that for $1\leq i,j\leq d$,
\begin{equation}\label{eq:W1p-semi-norm_iii}
\lim_{J\ni m\to \infty} \left\|w_{ij}^{(m,n)}-u_{ij}^{(n)}\right\|_{L^p}=0 \ .
\end{equation}
Since the diameter of the support of $\phi_{m,n}$ tends to zero
and $\|D\phi_{m,n}\|_{C^0}$ is uniformly bounded over $m$ and $n$,
we know that 
$\phi_{m,n}\to\mathrm{id}$ in $\varphi_n(\mathsf{U}_n)$ in measure, and 
$D\phi_{m,n}\to \mathrm{id}_{\mathbb{R}^d}$ in measure, both as $m\to\infty$.
Thus $v_{kj}^{(m,n)}\to \delta_{kj}$ in measure as $m\to\infty$ 
(where $\delta_{kj}$ denotes the Kronecker delta).
Hence $w_{ij}^{(m,n)}\to u_{ij}^{(n)}$ in measure as $m\to\infty$.
By a well-known result in measure theory (see for instance~\cite{Fo})
this implies that there exists a subsequence $J_1$ of the natural numbers such that for $1\leq i,j\leq d$,
we have $w_{ij}^{(m,n)}(x)\to u_{ij}^{(n)}(x)$ as $J_1\ni m\to\infty$, for Lebesgue almost every $x\in\varphi_n(\mathsf{U}_n)$.
Now we use the following fact from measure theory (see~\cite[p. 76]{Ru}), valid for arbitrary measure spaces with a positive measure:

\medskip
{\sl Fact:\/} 
{\it{Suppose $\sigma\in L^r$, $\sigma_m\in L^r$ where $1<r<\infty$. 
If 
$\sigma_m(x)\to \sigma(x)$ almost everywhere 
and 
$\|\sigma_m\|_{L^r}\to \|\sigma\|_{L^r}$ as $m\to\infty$, 
then 
$\lim_{m\to\infty}\|\sigma_m-\sigma\|_{L^r}=0$\/}}. 
\medskip

For each $i$ and $j$, $1\leq i,j\leq d$, we apply the above fact in the case 
$r=p$,
$\sigma=u_{ij}^{(n)}$ and
$\sigma_m=w_{ij}^{(m,n)}$, for $m\in J_1$.
(As $f_n$ and $g_{m,n}$ lie in $\mathbb{W}^{1,p}$ 
it follows that $w_{ij}^{(m,n)}$ and $u_{ij}^{(n)}$ lie in $L^p$ for all $i$ and $j$, $1\leq i,j\leq d$.)
We already know that $w_{ij}^{(m,n)}(x)\to u_{ij}^{(n)}(x)$ for Lebesgue almost every $x$ along $J_1$.
Hence we only need to check that
$\|w_{ij}^{(m,n)}\|_{L^p}\to \|u_{ij}^{(n)}\|_{L^p}$ as $J_1\ni m\to \infty$.
It suffices to show that there exists a subsequence $J_2\subseteq J_1$ for which
\begin{equation}\label{eq:Lp-convergence}
\lim_{J_2\ni m\to\infty}\int_{\varphi_n(\mathsf{U}_n)} \left|w_{ij}^{(m,n)}\right|^p\,d\mu(x)
=
\int_{\varphi_n(\mathsf{U}_n)}\left|u_{ij}^{(n)}(x)\right|^p \,d\mu(x) \ .
\end{equation}
Observe that $E_{m,M}$ as is contained in $\mathsf{U}_n$, 
$\phi_{m,n}(\varphi_n(\mathsf{U}_n))=\varphi_n(\mathsf{U}_n)$.
Thus, applying the change of variables 
$y=\phi_{m,n}(x)$ 
we can write
\begin{align}
&
\int_{\varphi_n(\mathsf{U}_n)}\left|w_{ij}^{(m,n)}(x)\right|^p \,d\mu(x)
=
\int_{\varphi_n(\mathsf{U}_n)}
\Bigl|\sum_{1\leq k\leq d} u_{ik}^{(n)}(\phi_{m,n}(x))v_{kj}^{(m,n)}(x) \, \Bigr|^p \,d\mu(x)\notag\\
&=
\int_{\varphi_n(\mathsf{U}_n)}
\Bigl|\sum_{1\leq k\leq d} u_{ik}^{(n)}(y)v_{kj}^{(m,n)}(\phi_{m,n}^{-1}(y)) \, \Bigr|^p J_{\phi_{m,n}^{-1}}(y) \,d\mu(y) \ ,\label{eq:Lp-change_of_variables}
\end{align}
where 
$J_{\phi_{m,n}^{-1}}(y)=\det D\phi_{m,n}^{-1}(y)$ denotes the Jacobian of $\phi_{m,n}^{-1}$ at $y$.
Note, using the change of variables formula here is legitimate as $\phi_{m,n}$ is a diffeomorphism.
As before, we have $\phi_{m,n}^{-1}\to \mathrm{id}$ in measure and $J_{\phi_{m,n}^{-1}}\to 1$ in measure as $m$ tends to infinity,
so, 
passing to a subsequence $J_2\subseteq J_1$ if necessary, 
we can once again assume convergence at Lebesgue almost every point of $\varphi_n(\mathsf{U}_n)$.
Thus we now know that
$v_{kj}^{(m,n)}(\phi_{m,n}^{-1}(y))\to \delta_{kj}$ 
and that 
$J_{\phi_{m,n}^{-1}}(y)\to 1$
for Lebesgue almost every $y$ in $\varphi_n(\mathsf{U}_n)$, 
as $J_2\ni m\to\infty$. 
Hence the integrand in~\eqref{eq:Lp-change_of_variables} converges to $|u_{ij}^{(n)}(y)|^p$ for Lebesgue almost every $y$ in $\varphi_n(\mathsf{U}_n)$.
Since 
$\phi_{m,n}$ leaves $\varphi_n(\mathsf{U}_n)$ invariant, 
and since 
$\|D\phi_{m,n}\|_{C^0}\leq K$ for some constant $K$ independent of $m$, we have that
\begin{equation}\label{ineq:Dphi_mn-entries-Linfty}
\bigl\|v_{kj}^{(m,n)}\circ\phi_{m,n}^{-1}\bigr\|_\infty
\leq
\bigl\|v_{kj}^{(m,n)}\bigr\|_\infty
\leq
\left\|D\phi_{m,n}\right\|_{C^0}
\leq 
K \ .
\end{equation}
For example, by inequality~\eqref{ineq:W1,d-jacobian} above, we also have
\begin{equation}\label{ineq:Dphi_mn-jacobian-bdd}
J_{\phi_{m,n}^{-1}}(y)
=\det D\phi_{m,n}^{-1}(y)
\leq d!\, \|D\phi_{m,n}^{-1}(y)\|_{C^0}^d
\leq d!\, K^d \ .
\end{equation}
Combining inequalities~\eqref{ineq:Dphi_mn-entries-Linfty} and~\eqref{ineq:Dphi_mn-jacobian-bdd}, 
we deduce from Lebesgue's Dominated Convergence Theorem that
\begin{align*}
&
\lim_{J_2\ni m\to\infty}
\int_{\varphi_n(\mathsf{U}_n)} 
\Bigl|\sum_{1\leq k\leq d} u_{ik}^{(n)}(y) v_{kj}^{(m,n)}\left(\phi_{m,n}^{-1}(y)\right) \, \Bigr|^p J_{\phi_{m,n}^{-1}}(y)\,d\mu(y) \notag\\
&=
\int_{\varphi_n(\mathsf{U}_n)} \left|u_{ij}(y)\right|^p\,d\mu(y) \ .
\end{align*}
This proves inequality~\eqref{eq:Lp-convergence},
which in turn show -- given the fact stated above -- that~\eqref{eq:W1p-semi-norm_iii} holds for $J=J_2$.
Hence~\eqref{eq:W1p-semi-norm} is satisfied and this concludes the proof of part (i).

Now consider (ii).
Recall that
\begin{equation*}
\left\|f_n^{-1}-g_{m,n}^{-1}\right\|_{\mathbb{W}^{1,p^*}(\Omega_n^*,\mathbb{R}^d)}
=
\left\|f_n^{-1}-g_{m,n}^{-1}\right\|_{C^0(\Omega_n^*)}
+
\left(\int_{\Omega_n^*} \left|Df_n^{-1}-Dg_{m,n}^{-1}\right|^{p^*}\,d\mu\right)^{\frac{1}{p^*}} \ .
\end{equation*}
By the same argument as that given in part (i), the hypotheses (b) and (d) imply that
\begin{equation*}
\left\|f_n^{-1}-g_{m,n}^{-1}\right\|_{C^0(\Omega_n^*)}\leq \epsilon_m \ .
\end{equation*}
Hence it suffices to show that, for some subsequence $J'\subseteq J$,
\begin{equation}\label{eq:W1p-semi-norm_inv}
\liminf_{J'\ni m\to\infty}\int_{\Omega_n^*} \left|Df_n^{-1}-Dg_{m,n}^{-1}\right|^{p^*}\,d\mu \;=\;0 \ .
\end{equation}
Since $\Omega_n^*$ is contained in $\psi_n(\mathsf{V}_n)$, \eqref{eq:W1p-semi-norm_inv} this will follow if we can show that
\begin{equation}\label{eq:W1p-semi-norm_inv_ii}
\liminf_{J'\ni m\to\infty}\int_{\psi_n(\mathsf{V}_n)} |Df_n^{-1}-Dg_{m,n}^{-1}|^{p^*}\,d\mu \;=\;0 \ .
\end{equation}
Let $\sigma_{m,n}=\left|Df_n^{-1}-Dg_{m,n}^{-1}\right|^{p^*}$.
Once more by the chain rule (Lemma~\ref{chainrule}(ii)), 
for Lebesgue almost every $x\in\psi_n(\mathsf{V}_n)$ we have
\begin{equation*}
Dg_{m,n}^{-1}(x)\;=\;D\phi_{m,n}^{-1}\left(f_{n}^{-1}(x)\right)Df_n^{-1}(x) \ .
\end{equation*}
Moreover, $\sigma_{m,n}\in L^1\left(\psi_n(\mathsf{V}_n)\right)$ since
\begin{align*}
\left|Df_{n}^{-1}-D\phi_{m,n}^{-1}\circ f_{n}^{-1} Df_n^{-1}\right|^{p^*}
&\leq \left|\mathrm{id}_{\mathbb{R}^d}-D\phi_{m,n}^{-1}\circ f_n^{-1}\right|^{p^*} \left|Df_{n}^{-1}\right|^{p^*}\\
&\leq (1+K)^{p^*} \left|Df_{n}^{-1}\right|^{p^*} \ ,
\end{align*}
where we have used that $\left\|D\phi_{m,n}^{-1}\right\|_{C^0}\leq K$.
We claim that the sequence $(\sigma_{m,n})_{m\in\mathbb{N}}$ converges in measure to the zero function.
This happens because
\begin{align*}
\mu\left(\left\{ x : \left| \sigma_{m,n}(x) \right|>0\right\}\right)
&=
\mu\left(\left\{ x : \left| \mathrm{id}_{\mathrm{R}^d}-D\phi_{m,n}^{-1}(f_{n}^{-1}(x)) \right| >0\right\}\right)\\
&\leq
\mu \left(f_{n}(\varphi_{n}(E_{m,M}))\right) \ .
\end{align*}
But $\mu\left(f_{n}(\varphi_n(E_{m,M}))\right)\to 0$ as $m\to \infty$, since $\mathrm{diam}\left(E_{m,M}\right)\to 0$ and $f_{n}$ is uniformly continuous.
Once more, this implies that there exists a subsequence $J'$ of $J$ such that $\sigma_{m,n}(x)\to 0$ for Lebesgue almost every $x$ in $\psi_n\left(\mathsf{V}_n\right)$.
By Lebesgue's Dominated Convergence Theorem it follows that
\begin{equation*}
\lim_{J'\ni m\to\infty}\int_{\psi_n(\mathsf{V}_n)}\sigma_{m,n}\,d\mu \;=\;0 \ ,
\end{equation*}
which shows that~\eqref{eq:W1p-semi-norm_inv_ii} and hence~\eqref{eq:W1p-semi-norm_inv} holds.
This concludes part (ii), and hence the proof is complete.
\end{proof}
\subsection{Genericity of infinite topological entropy for Sobolev mappings}\label{sec:Sobolev-infinite_entropy}
In this section we prove that infinite topological 
entropy is a generic property in the Sobolev context. 

First, we give an argument specific to dimension two.
The novelty in this approach is that it recovers the 
``na\"{\i}ve'' argument where a periodic point is 
first `blown-up' to a periodic disk, 
and then a horseshoe with an appropriate number of 
branches is `glued-in' to this disk.
(See the first comment in Section~\ref{sect:open_problems} for more details.)
It is based on a generalised version of the 
Rad\'o-Kneser-Choquet Theorem (see~\cite{IKO,AS}). 
However, there is no known generalisation of this 
result to higher dimensions.
In fact, there are explicit counterexamples to the 
classical Rad\'o-Kneser-Choquet Theorem, 
see~\cite[Section 3.7]{DurenBook}.

Secondly, we present an argument analogous to that 
used in the H\"older case above, and is applicable 
in all dimensions greater than one.
As much of the argument is the same as in the H\"older 
case, we only give a sketch, drawing attention to 
where modifications are necessary.

\subsubsection{First argument}\label{subsubsect:first_argument}
Our goal in this section is to show that infinite 
topological entropy is a generic property for 
homeomorphisms of compact surfaces in certain 
Sobolev classes. 
More precisely, we will prove the following result.
\begin{theorem}\label{infsobothm1}
Let $1<p<\infty$. 
Let $M$ be a compact oriented surface. 
The set of orientation-preserving Sobolev homeomorphisms 
in $\mathcal{S}^{p,1}(M)$ with infinite topological entropy 
contains a residual subset of $\mathcal{S}^{p,1}(M)$. 
\end{theorem}
This theorem will be deduced from a corresponding 
result for maps in the plane, which we proceed to state.

Take $1<p<\infty$, which we assume to be fixed throughout this section. 
Let $\Omega$ and $\Omega^*$ be bounded open sets in the plane.
We conform with the notation introduced in \S~\ref{sec:Sobolev_prelim}. 
In particular, we denote by $\rho$ the Sobolev distance in 
$\mathcal{S}^{p,1}(\Omega, \Omega^*)$. 
Namely
\begin{equation}
\rho(f,g)=
\|f-g\|_{\mathbb{W}^{1,p}(\Omega,\mathbb{R}^2)}+\|f^{-1}-g^{-1}\|_{\mathbb{W}^{1,1}(\Omega^*,\mathbb{R}^2)} \ .
\end{equation}
We assume that 
$\Omega\cap \Omega^*\neq \emptyset$, 
and for each $k$,  $0\leq k\leq \infty$, 
we denote by $\Omega_k=\Omega_k(f)$ the subset 
of points $x\in \Omega$ such that $f^j(x)$ 
is defined for all $0\leq j< k$.
Recall that a point $x\in\Omega_\infty$ is 
(forward) recurrent if it belongs to its 
own $\omega$-limit set.
The set of recurrent points is called the {\it recurrent set}. 
Let us denote by 
$\mathcal{S}_{\infty}^{p,1}(\Omega,\Omega^*)$ 
the {\it closure\/} of the set of all those
Sobolev homeomorphisms 
$f\in \mathcal{S}^{p,1}(\Omega,\Omega^*)$ 
with non-trivial recurrent set.

For each $n\in \mathbb{N}$, let us denote by 
$\mathcal{G}_{n}$ the set of all 
$g\in \mathcal{S}^{p,1}(\Omega,\Omega^*)$ 
for which there exist 
$k\in \mathbb{N}$ and 
a topological disk $D\Subset \Omega_k(g)$ 
such that 
$g^k|_{D}\colon D\to g^k(D)$ is an $n^k$-branched horseshoe map. 
Note that 
$\mathcal{G}_{n}\subset \mathcal{S}_{\infty}^{p,1}(\Omega,\Omega^*)$, 
for all $n\in \mathbb{N}$. 
Note also that if 
$g\in \mathcal{G}_{n}$ 
then 
$h_{\mathrm{top}}(g^k)\geq \log{(n^k)}=k\log{n}$, 
and therefore we have
\begin{equation*}%\label{sobinf1}
g\in \mathcal{G}_{n}\; \implies \; h_{\mathrm{top}}(g)\geq \log{n}\ .
\end{equation*}
With this notation we can now state our theorem as follows. 
\begin{theorem}\label{infsobothm2}
Let $1<p<\infty$.
For each $n\in \mathbb{N}$ 
the set 
$\mathcal{G}_{n}$ 
is dense in 
$\mathcal{S}^{p,1}_\infty(\Omega,\Omega^*)$.
Consequently, the set of homeomorphisms with infinite entropy in $\mathcal{S}_\infty^{p,1}(\Omega,\Omega^*)$
contains a residual subset of $\mathcal{S}_{\infty}^{p,1}(\Omega,\Omega^*)$. 
\end{theorem}
\begin{remark}
As $\mathcal{S}_\infty^{p,1}(\Omega,\Omega^*)$ is closed in $\mathcal{S}^{p,1}(\Omega,\Omega^*)$ and $\mathcal{S}^{p,1}(\Omega,\Omega^*)$
has the Baire property, it follows that $\mathcal{S}_\infty^{p,1}(\Omega,\Omega^*)$ also has the Baire property.
\end{remark}
It will be straightforward to deduce 
Theorem~\ref{infsobothm1} 
from 
Theorem~\ref{infsobothm2}. 
As for the latter, note that the second assertion 
in the statement is an immediate consequence of the 
first.
Namely, since the $\rho$-distance in 
$\mathcal{S}^{p,1}(\Omega,\Omega^*)$ 
is greater than 
the $C^0$-distance, and since topological horseshoe 
maps are stable under small $C^0$ perturbations, 
it follows that each $\mathcal{G}_{n}$ is open in 
$\mathcal{S}_\infty^{p,1}(\Omega,\Omega^*)$. 
Therefore the proof of Theorem~\ref{infsobothm2} 
will be complete once we show that each 
$\mathcal{G}_{n}$ is {\it dense\/} in 
$\mathcal{S}_{\infty}^{p,1}(\Omega,\Omega^*)$.  

The geometric idea behind the proof of such a 
density result is very simple, and can informally 
be described as follows. 
Starting with an arbitrary 
$f\in \mathcal{S}_{\infty}^{p,1}(\Omega,\Omega^*)$, 
the first step is to apply the Sobolev Closing Lemma 
to get  
$g_1\in \mathcal{S}_{\infty}^{p,1}(\Omega,\Omega^*)$ 
close to $f$ which has a periodic orbit. 
The second step is to then perform a surgery on $g_1$ 
in order to get a new 
$g_2\in  \mathcal{S}_{\infty}^{p,1}(\Omega,\Omega^*)$ 
close to $g_1$ which still has the same periodic orbit 
as $g_1$ but which is now a smooth diffeomorphism in a 
neighbourhood of that periodic orbit. 
The third step is to use a bump 
function argument to replace $g_2$ by yet another homeomorphism 
$g_3\in \mathcal{S}_{\infty}^{p,1}(\Omega,\Omega^*)$ 
close to $g_2$, 
still with the same periodic orbit as $g_2$, 
having now a periodic cycle of {\it disks\/} around that periodic orbit on which $g_3$ moves 
points about by rigid translations -- in particular, if $k$ is the period, there is a disk $D$ around a point of the periodic cycle 
such that $g_3^k|_{D}$ is the identity. The fourth and final step is to perform another (smooth) surgery to replace 
$g_3$ by a new $g_4\in \mathcal{S}_{\infty}^{p,1}(\Omega,\Omega^*)$ very close to $g_3$ having the same periodic disk $D$ as $g_3$, but now with the 
property that $g_4^k|_{D}$ is a horseshoe map with compact support in $D$ and appropriately high entropy. The only difficult step is the 
second. The Sobolev surgery used in this step requires us to introduce the notion of a {\it $p$-harmonic map} as it uses a generalised version of a non-trivial 
theorem due to Rad\'o, Kneser and Choquet (Theorem \ref{rado-kneser-choquet} below), and is inspired by~\cite{IKO}.

\subsubsection*{$p$-Harmonic maps.}

As is customary, we identify $\mathbb{R}^2$ 
with the complex plane $\mathbb{C}$. 
Let $1<p<\infty$ and let 
$\Omega \subset \mathbb{C}$ 
be a bounded open set. A function 
$u\colon \Omega \to \mathbb{R}$ 
is said to be {\it $p$-harmonic\/} 
if $u\in W^{1,p}(\Omega)$ and 
\begin{equation*}
 \mathrm{div}\left(|\nabla u|^{p-2}\nabla u\right)\;=\;0
\end{equation*}
in the sense of distributions. 
Here and throughout, $\nabla$ denotes 
the gradient operator, and obviously 
`$\mathrm{div}$' denotes the divergence 
operator. 
Note that $p$-harmonic for $p=2$ simply 
means harmonic in the usual sense. 
It is a fact (from the theory of elliptic 
partial differential equations) that 
$p$-harmonic functions are minimisers for 
the so-called $p$-energy functional
\begin{equation*}
 \mathcal{E}_p(u)\;=\; \int_{\Omega} |\nabla u|^{p}\,d\mu\ , 
\end{equation*}
where as before $\mu$ denotes Lebesgue measure. 
The norm of the gradient is the standard Euclidean norm, 
namely 
$|\nabla u|= \sqrt{u_x^2+u_y^2}$. 
\begin{definition}
A homeomorphism 
$f=u+iv\colon\Omega \to \Omega^*\subset \mathbb{C}$ 
is said to be \emph{coordinate-wise $p$-harmonic}, 
or simply \emph{$p$-harmonic}, if its 
components 
$u,v\colon \Omega \to \mathbb{R}$ 
are both $p$-harmonic. 
\end{definition}
By analogy with the case of real functions, 
given a map 
$f=u+iv\in W^{1,p}(\Omega,\mathbb{C})$, 
we define its $p$-energy as the sum of the 
$p$-energies of its real and imaginary parts, 
{\it i.e.\/},
\begin{equation*}
 \mathcal{E}_p(f)\;=\; \int_{\Omega} \left(|\nabla u|^{p}+|\nabla v|^{p}\right)\,d\mu\ , 
\end{equation*}
Just as in the case of real functions, 
$p$-harmonic homeomorphisms are minimisers 
of the $p$-energy. 

It is easily seen that the $p$-energy of 
$f\in W^{1,p}(\Omega,\mathbb{C})$ 
controls the $L^p$-norm of $|Df|$ and vice-versa.
Indeed, since in the present context we have 
$|Df|=|u_x|+|u_y|+|v_x|+|v_y|$, 
we have the double inequality
\begin{equation}\label{norm-energy}
\mathcal{E}_p(f)\;\leq\; \int_{\Omega} |Df|^p\,d\mu\;\leq\; c_p \mathcal{E}_p(f)\ ,
\end{equation}
where $c_{p}>1$ is a constant depending only on $p$ {\footnote{In fact, one can take $c_p=2^{\frac{3p}{2}}$.}} 

The only non-trivial fact we will use about $p$-harmonic homeomorphisms is the following generalization due to 
Alessandrini and Sigalotti \cite{AS} of a theorem due to Rad\'{o}, Kneser and Choquet. The formulation below is adapted 
from~\cite{IKO}.

\begin{theorem}\label{rado-kneser-choquet}
 Let $D,D^*$ be two Jordan domains in the plane. 
 Assume that both Jordan curves $\partial D$, $\partial D^*$ are positively oriented, 
 and that $D^*$ is \emph{convex}. 
 Given $1<p<\infty$ and a homeomorphism $h\colon \partial D\to \partial D^*$ which preserves orientation, 
 there exists an orientation-preserving homeomorphism $\phi\colon \overline{D}\to \overline{D}^*$ such that $\phi|_{\partial D}\equiv h$ and 
 $\phi$ is $p$-harmonic. Moreover, $\phi|_{D}$ is a $C^\infty$-diffeomorphism onto $D^*$, and in particular its Jacobian is everywhere 
 positive in $D$. 
\end{theorem}

When we have a diffeomorphism $\phi$ in $W^{1,p}$, the $p$-energy of $\phi$ always bounds the $1$-energy of $\phi^{-1}$. 
This is the content of the following simple lemma, which will be used in combination with Theorem \ref{rado-kneser-choquet}. 

\begin{lemma}\label{inverse-energy-lemma}
 Let $\phi\colon D\to D^*$ be $C^1$-diffeomorphism between two bounded domains in the plane, and suppose $\phi\in W^{1,p}(D,\mathbb{C})$ for some 
 $1<p<\infty$. Then $\phi^{-1}\in W^{1,1}(D^*,\mathbb{C})$, and in fact
 \begin{equation}\label{inverse-energy}
  \mathcal{E}_{1}(\phi^{-1})\;\leq\; 4\left(\mathrm{Area}(D)\right)^{1-\frac{1}{p}}\mathcal{E}_p(\phi)^{\frac{1}{p}}\ .
 \end{equation}
\end{lemma}

\begin{proof}
There is no loss of generality in assuming that $\phi$ preserves orientation. Let us write $\phi=u+iv$ and $\phi^{-1}=U+iV$.
We need to bound
\begin{align}\label{4integrals}
&\mathcal{E}_{1}(\phi^{-1})
\;=\; \int_{D^*} |D\phi^{-1}|\,d\mu \notag\\
&\;=\; \int_{D^*} |U_x|\,d\mu + \int_{D^*}|U_y|\,d\mu + \int_{D^*}|V_x|\,d\mu +\int_{D^*}|V_y|\,d\mu
\end{align}
in terms of the $p$-energy of $\phi$. We proceed to bound each of the four integrals in the right-hand side of~\eqref{4integrals}. 
By the chain rule we have $D\phi\circ \phi^{-1}\cdot D\phi^{-1}=\mathrm{id}$, and therefore
\begin{equation*}
D\phi^{-1} 
  =   \left[\begin{matrix} U_x& U_y\\V_x& V_y\end{matrix}\right] 
\;=\; \frac{1}{J_{\phi}\circ \phi^{-1}}\left[\begin{matrix} \,v_y\circ \phi^{-1} & - u_y\circ \phi^{-1}\\ 
                                                               -v_x\circ \phi^{-1} & \,u_x\circ \phi^{-1}\end{matrix}\right] 
  =   \left( D\phi\circ \phi^{-1} \right)^{-1}  \ ,
\end{equation*}
where $J_\phi=\mathrm{det}(D\phi)=u_xv_y-u_yv_x>0$ is the Jacobian of $\phi$. Note also that $J_\phi\circ \phi^{-1}= (J_{\phi^{-1}})^{-1}$. 
Comparing entries in the matrices above, we get
\begin{align*}
 U_x=\;v_y\circ\phi^{-1}\cdot J_{\phi^{-1}}\ \ \ &,\ \ \ U_y= -u_y\circ \phi^{-1}\cdot J_{\phi^{-1}} \ \ ,\\
 V_x= -v_x\circ\phi^{-1}\cdot J_{\phi^{-1}}\ \ \ &,\ \ \ V_y= \;u_x\circ \phi^{-1}\cdot J_{\phi^{-1}} \ \ .
\end{align*}
From the first of the four inequalities above, we deduce by a simple application of the change of variables formula that
\begin{equation*}
 \int_{D^*} |U_x|\,d\mu \;=\; \int_{D^*} |v_y|\circ \phi^{-1}\cdot J_{\phi^{-1}}\,d\mu\;=\; \int_{D} |v_y|\,d\mu\ .
\end{equation*}
But H\"older's inequality tells us that
\begin{equation*}
 \int_{D} |v_y|\,d\mu\;\leq\; \left( \int_{D} d\mu   \right)^{1-\frac{1}{p}} \left(\int_{D} |v_y|^p\,d\mu\right)^{\frac{1}{p}}\ .
\end{equation*}
This obviously implies that
\begin{equation*}
 \int_{D^*} |U_x|\,d\mu \;\leq\; \left(\mathrm{Area}(D)\right)^{1-\frac{1}{p}} \left(\int_{D} |v_y|^p\,d\mu\right)^{\frac{1}{p}}
 \;\leq\; \left(\mathrm{Area}(D)\right)^{1-\frac{1}{p}}\mathcal{E}_p(\phi)^{\frac{1}{p}}\ .
\end{equation*}
The same estimate holds for the remaining three integrals in the right-hand side of \eqref{4integrals}. Adding 
up all these estimates yields \eqref{inverse-energy}, as desired. 
\end{proof}

\subsubsection*{Replacement trick}
The following proposition shows 
that we can always replace a Sobolev 
homeomorphism in 
$\mathcal{S}^{p,1}(\Omega,\Omega^*)$ 
by another which is very close to it 
and is in fact smooth in the neighbourhood 
of a point specified in advance. 
In the proof we will implicitly use, 
in addition to the auxiliary results 
of the previous section, the following 
elementary remark.
\begin{remark}\label{rmk:elementary_remark} 
If 
$f\colon\Omega\to \Omega^*$ 
is a homeomorphism and $a^*\in \Omega^*$ 
and $r_0>0$ satisfy 
$\overline{D}(a^*, r_0)\subset \Omega^*$, 
then for all but countably many $r\in [0,r_0]$ 
the Jordan curve 
$f^{-1}(\partial D(a^*,r))$ 
has zero Lebesgue measure. 
\end{remark}
\begin{proposition}\label{prop:replacement_trick}
Let 
$f\colon\Omega\to \Omega^*$ 
be a Sobolev homeomorphism in 
$\mathcal{S}^{p,1}(\Omega,\Omega^*)$, 
and let 
$a\in \Omega$ and $a^*\in \Omega^*$ 
satisfy $f(a)=a^*$. 
Then for each $\epsilon>0$ there exists 
a topological disk $\mathcal{O}^*$, 
compactly contained in $\Omega^*$, and 
$g\in \mathcal{S}^{p,1}(\Omega,\Omega^*)$ 
having the following properties:
\begin{enumerate}
\item[(i)] 
Both $\mathcal{O}^*$ and its pre-image $\mathcal{O}=f^{-1}(\mathcal{O}^*)$ have diameter less than $\epsilon$. 
\item[(ii)] 
The map $g$ agrees with $f$ on $\Omega\setminus \mathcal{O}$.
\item[(iii)] 
The restriction $g|_{\mathcal{O}}:\, \mathcal{O}\to \mathcal{O}^*$ is a $C^{\infty}$ diffeomorphism. 
\item[(iv)] 
The map $g$ is $\epsilon$-close to $f$, {\it i.e.\/}, $\rho(f,g)<\epsilon$. 
\item[(v)] 
We have $a\in \mathcal{O}$, $a^*\in \mathcal{O^*}$, and the equality $g(a)=a^*$ holds true. 
\end{enumerate}
\end{proposition}

\begin{proof}
First we prove that for each $\epsilon>0$ 
there exists a topological disk $\mathcal{O}^*$ 
and 
$g\in \mathcal{S}^{p,1}(\Omega,\Omega^*)$ 
such that the properties (i)--(iv) hold. 
We take care of property (v) only at the 
end of the proof. We proceed by steps in 
the following way. 
Let us choose $0<\epsilon_0<\epsilon$. 
(How small $\epsilon_0$ needs to be  will 
be determined in the course of the argument). 
 
\begin{figure}[t]
\begin{center}
\psfrag{f}[][]{ $f$} 
\psfrag{b}[][][1]{$B_k$}
\psfrag{fb}[][][1]{$\!\!\!\!\!\!\!\!f^{-1}\!(\!B_k\!)$} 
\psfrag{D}[][][1]{$\Delta$}
\psfrag{fD}[][][1]{$f^{-1}(\Delta)$}
%\psfrag{a}[][]{$f^{q_{n}-q_{n-1}}(c_{0})$} 
%\psfrag{b}[][]{$f^{q_{n}}(c_{0})$} 
%\psfrag{c}[][]{$c_{0}$} 
%\psfrag{d}[][]{$f^{-q_{n-1}}(c_{0})$} 
%\psfrag{e}[][]{$f^{2q_{n-1}}(c_{0})$} 
\includegraphics[width=3.5in]{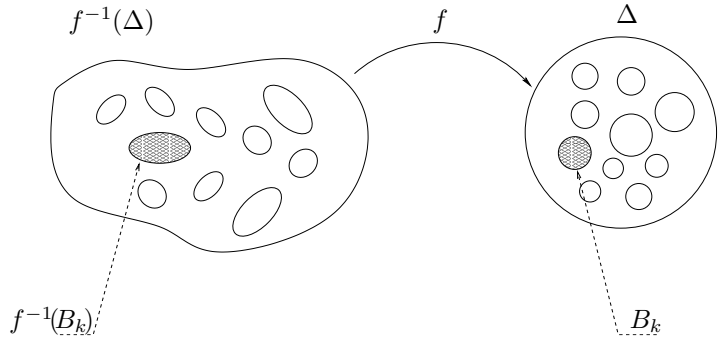}
\end{center}
\caption[topdisk]{\label{topdisk} A topological disk carrying small Sobolev norm for $f$.}
\end{figure}

 \begin{enumerate}
  \item[(1)] 
  By uniform continuity of $f^{-1}$, there exists 
  $0<\delta<\epsilon_0/2$ 
  such that 
  $D(a^*,\delta)\subset \Omega^*$ 
  and 
  $\mathrm{diam}\left(f^{-1}(D(a^*,\delta))\right)<\epsilon_0$.
  \item[(2)] 
  Let $N\in \mathbb{N}$ satisfy 
  $N>\epsilon_0^{-1}\max\left\{c_p\mathcal{E}_p(f),c_1\mathcal{E}_1(f^{-1})\right\}$. 
  Choose $N$ pairwise disjoint balls (disks) 
  $\Delta_1, \Delta_2,\ldots, \Delta_N\subset D(a^*,\delta)$ 
  and note that, by the inequalities~\eqref{norm-energy},
  \begin{equation*}
   \frac{1}{N}\sum_{j=1}^{N} \int_{\Delta_j} |Df^{-1}|\,d\mu
   \;\leq\;
   \frac{1}{N}\int_{D(a^*,\delta)} |Df^{-1}|\,d\mu
%   \;\leq\;
%   \frac{c_1}{N}\mathcal{E}_1(f^{-1})
   \;<\; 
   \epsilon_0\ .
  \end{equation*}
  Therefore at least one of the disks $\Delta_1,\Delta_2,\ldots,\Delta_N$, call it $\Delta$, satisfies 
  \begin{equation*}
   \int_{\Delta} |Df^{-1}|\,d\mu \;<\; \epsilon_0 \ .
  \end{equation*}
  Now choose pairwise disjoint balls $B_1,B_2,\ldots,B_N\subset \Delta$ 
  for which we have $\mu(f^{-1}(\partial B_j))=0$, for all $1\leq j\leq N$. 
  This is possible by Remark~\ref{rmk:elementary_remark}.
  See Figure~\ref{topdisk}. Then
  \begin{equation*}
   \frac{1}{N}\sum_{j=1}^{N} \int_{f^{-1}(B_j)} |Df|^p\,d\mu
   \;\leq\; 
   \frac{1}{N}\int_{f^{-1}(\Delta)} |Df|^p\,d\mu
%   \;\leq\;
%   \frac{c_p}{N}\mathcal{E}_p(f)
   \;<\; 
   \epsilon_0\ ,
  \end{equation*}
  and from this it follows that there exists $k\in \{1,2,\ldots,N\}$ such that
  \begin{equation}\label{smallnorm_init}
  \int_{f^{-1}(B_k)} |Df|^p\,d\mu
  \;<\; 
  \epsilon_0\ .
  \end{equation}
  Let us define $\mathcal{O}^*=B_k$ and $\mathcal{O}=f^{-1}(B_k)$. 
  Then the above considerations imply that
  \begin{equation}\label{smallnorm}
  \mathcal{E}_p(f|_{\mathcal{O}})<\epsilon_0\ 
  \ \ \textrm{and}\ \ \ 
  \mathcal{E}_1(f^ {-1}|_{\mathcal{O}^ *})<\epsilon_0\ .
  \end{equation}

  \item[(3)] 
  Let 
  $\phi\colon \overline{\mathcal{O}} \to \overline{\mathcal{O}^*}$ 
  be the $p$-harmonic homeomorphism with 
  $\phi|_{\partial \mathcal{O}}\equiv f|_{\partial \mathcal{O}}$ 
  whose existence is guaranteed by Theorem~\ref{rado-kneser-choquet}. 
  From~\eqref{smallnorm} above and since $\phi$ minimises $p$-energy, 
  we know that 
  $\mathcal{E}_p(\phi) \leq \mathcal{E}_p(f|_{\mathcal{O}}) < \epsilon_0$. 
  \item[(4)] 
  Define $g\colon \Omega\to \Omega^*$ by setting 
  \begin{equation*}
  g(z)=
  \left\{\begin{array}{ll}
  f(z) & z\in \Omega\setminus \mathcal{O}\\ 
  \phi(z) &z\in \mathcal{O}
  \end{array}\right. \ \ .
  \end{equation*}
  Then $g$ is a homeomorphism. 
  Since 
  $\mu(\partial \mathcal{O})=\mu(f^{-1}(\partial B_k))=0$, 
  we see that 
  $g\in W^{1,p}(\Omega,\mathbb{C})$. 
  Similarly, since 
  $\mu(\partial\mathcal{O}^*)=\mu(\partial B_k)=0$, 
  we also have 
  $g^{-1}\in W^{1,1}(\Omega^*,\mathbb{C})$. 
  Hence 
  $g\in \mathcal{S}^{p,1}(\Omega, \Omega^*)$. 
  Moreover, 
  $g|_{\mathcal{O}}\equiv \phi\colon \mathcal{O}\to \mathcal{O}^*$ is a $C^{\infty}$-diffeomorphism.
  \item[(5)] 
  Let us now estimate the distance $\rho(f,g)$. 
  Since the support of $f-g$ lies in 
  $\mathcal{O}=f^{-1}(\mathcal{O}^*)=g^{-1}(\mathcal{O}^*)$, we have 
   \begin{equation}\label{fgeq1}
   \|f-g\|_{C^0(\Omega)}\leq \mathrm{diam}(\mathcal{O}^*)\;=\;2\delta\;<\;\epsilon_0\ .
   \end{equation}
   Likewise, since the support of $f^{-1}-g^{-1}$ lies in 
   $\mathcal{O}^*$, we have 
   \begin{equation}\label{fgeq2}
   \|f^{-1}-g^{-1}\|_{C^0(\Omega^*)}\leq \mathrm{diam}(\mathcal{O})<\epsilon_0 \ .
   \end{equation}
   Moreover, we have 
   \begin{align*}
    \|Df - Dg\|_{L^p(\Omega)} \;&=\; \|D(f|_{\mathcal{O}}) -D\phi\|_{L^p(\mathcal{O})}  \\
      &\leq \|D(f|_{\mathcal{O}})\|_{L^p(\mathcal{O})} + \|D\phi\|_{L^p(\mathcal{O})} \\
      &\leq \; c_p^{\frac{1}{p}}\left[\mathcal{E}_p(f|_{\mathcal{O}})^{\frac{1}{p}} + \mathcal{E}_p(\phi)^{\frac{1}{p}}\right] \ . 
   \end{align*}
   Using step (3) above we deduce that
   \begin{equation}\label{fgeq3}
    \|Df - Dg\|_{L^p(\Omega)} \;\leq\; 2(c_p\epsilon_0)^{\frac{1}{p}}\ .
   \end{equation}
   Finally, applying Lemma~\ref{inverse-energy-lemma} to $\phi$ with $D^*=\mathcal{O}^*$ and $D=\mathcal{O}$, and taking into account 
   that the area of $\mathcal{O}$ is less than $\pi\epsilon_0^2<1$ (if $\epsilon_0$ is small enough), we get
   \begin{align*}
    \|Df^{-1} - Dg^{-1}\|_{L^1(\Omega^*)}\;&=\; \|D(f^{-1}|_{\mathcal{O}^*}) - D\phi^{-1}\|_{L^1(\mathcal{O}^*)} \\
             &\leq\; c_1\left[\mathcal{E}_1(f^{-1}|_{\mathcal{O}^*}) + \mathcal{E}_1(\phi^{-1})\right]  \\
             &\leq\; c_1\left[\mathcal{E}_1(f^{-1}|_{\mathcal{O}^*})  + 4\mathcal{E}_p(\phi)^{\frac{1}{p}}\right]  
             \ .
   \end{align*}
   Again using step (3) and the second inequality in \eqref{smallnorm}, we deduce that  
   \begin{equation}\label{fgeq4}
   \|Df^{-1} - Dg^{-1}\|_{L^1(\Omega^*)}\;\leq\; 5c_1\epsilon_0^{\frac{1}{p}} \ .
   \end{equation}
   Putting together \eqref{fgeq1}, \eqref{fgeq2}, \eqref{fgeq3}, and \eqref{fgeq4}, it follows that $\rho(f,g)<\epsilon$, provided 
   $\epsilon_0$ is chosen so small that $2\epsilon_0+2(c_p\epsilon_0)^{\frac{1}{p}}+5c_1\epsilon_0^{\frac{1}{p}}<\epsilon$. 
 \end{enumerate}
   \noindent
   The proposition is almost proved. The only problem is 
   that the map $g$ we constructed above does not necessarily satisfy property (v). 
   We fix this problem as follows. 
   The argument we have given so far proves that 
   for each $n\in \mathbb{N}$ there exist: 
   \begin{itemize}
   \item[(a)] 
   a homeomorphism $f_n\in \mathcal{S}^{p,1}(\Omega, \Omega^*)$ which is $\epsilon_n$-close to $f$ in the 
   Sobolev metric, where $\epsilon_n=2^{-n}$, say; and 
   \item[(b)] 
   two topological disks $\mathcal{O}_n\subset D(a,\epsilon_n)\subset \Omega$ and 
   $\mathcal{O}^*_n\subset D(a^*,\epsilon_n)\subset \Omega^*$ with $f_n(\mathcal{O}_n)=\mathcal{O}^*_n$ such that $f_n|_{\mathcal{O}_n}$ is a 
   $C^{\infty}$-diffeomorphism and such that $f_n|_{\Omega\setminus\mathcal{O}_n}\equiv f|_{\Omega\setminus\mathcal{O}_n}$. 
   \end{itemize}
   For each $n$, choose a point $a_n\in \mathcal{O}_n$ and 
   let $a_n^*=f_n(a_n)\in \mathcal{O}^*_n$. 
   Using Lemma \ref{lem:E-perturbation}, we find a smooth diffeomorphism $\varphi_n\colon\Omega\to \Omega$  
   with support in the disk 
   $D(a,2\epsilon_n)$ such that $\varphi_n(a)=a_n$, and with the property that the $C^1$-norms of $\varphi_n$ and $\varphi_n^{-1}$ are 
   bounded by a constant independent 
   of $n$. In the same way we find a smooth diffeomorphism $\psi_n\colon\Omega^*\to \Omega^*$  with support in the disk $D(a^*,2\epsilon_n)$ 
   such that $\psi_n(a^*)=a_n^*$, also with the property that the $C^1$-norms of $\psi_n$ and $\psi_n^{-1}$ are bounded independently of $n$.
   Now let $g_n\colon\Omega\to \Omega^*$ be the homeomorphism $g_n=\psi_n^{-1}\circ f_n\circ \varphi_n\in \mathcal{S}^{p,1}(\Omega, \Omega^*)$. 
   We clearly have  $g_n\to f$ and $g_n^{-1}\to f^{-1}$ uniformly in $\Omega$ and $\Omega^*$, respectively. 
   By an argument analogous to the one used in the proof of the Sobolev Closing Lemma (Theorem~\ref{thm:sobolevclosing}) we know that 
   there exists a subsequence $n_k\to \infty$ such that 
   $Dg_{n_k}\to Df$ and $Dg_{n_k}^{-1}\to Df^{-1}$ {\it in measure\/}. 
   Passing to a further subsequence if necessary, we may assume that both $Dg_{n_k}$ and $Dg_{n_k}^{-1}$ converge pointwise Lebesgue 
   almost everywhere to $Df$ and $Df^{-1}$, respectively. 
   Then, just as in the proof of Theorem~\ref{thm:sobolevclosing}, a simple application of Lebesgue's Dominated Convergence Theorem shows that 
   $|Dg_{n_k}- Df|\to 0$ in $L^p(\Omega)$, and $|Dg_{n_k}^{-1}- Df^{-1}|\to 0$ in $L^1(\Omega^*)$. This shows that $\rho(g_{n_k},f)\to 0$ as 
   $k\to \infty$. 
   But then any $g=g_{n_k}$ for sufficiently large $k$ satisfies all five properties in the statement. This completes the proof. 
   \end{proof}

\subsubsection*{Blow-up}
The next proposition shows that, 
for smooth diffeomorphisms, 
in the neighbourhood of a periodic orbit, 
the map may be replaced by a translation. 
\begin{proposition}\label{prop:blow-up}
Let $\Omega,\Omega^*\subseteq \mathbb{R}^d$ 
be open domains with $\Omega$ path-connected. 
Let 
$f\in C^2(\Omega,\Omega^*)$ 
be an orientation-preserving embedding 
with periodic point $x_0$ of minimal period $k$.
There exists $C>0$ with the following property.
For each $r_0>0$ sufficiently small there exists 
\begin{itemize}
\item[(i)] 
an embedding $g\in C^2(\Omega,\Omega^*)$,
\item[(ii)]
concentric disks $D_{1,j}\subset D_{0,j}$ about $f^j(x_0)$ in $\Omega$, 
of radius $r_0$ or less,
for each $j=0,1,\ldots,k-1$, 
\end{itemize}
such that 
\begin{itemize}
\item[(a)]
$g|_{D_{1,j}}$ is a translation from $D_{1,j}$ to $D_{1,j+1}$ (addition taken modulo $k$), 
for each $j=0,1\ldots,k-1$,
\item[(b)]
$g|_{\Omega\setminus\bigcup_{j=0}^{k-1} D_{0,j}}=f$, 
and
\item[(c)]
$d_{\mathrm{Lip}}(f,g), \ d_{\mathrm{Lip}}(f^{-1},g^{-1})<C$.
\end{itemize}
\end{proposition}
\begin{proof}
The Proposition will follow if we can show it 
in the simplified case when $x_0=0$ is a fixed point.
Namely, it suffices to show when the origin is fixed that 
there exists a $C^2$-embedding $g$ and concentric 
disks $D_1\subset D_0$ about $x_0=0$ in $\Omega$ such that 
$g|D_1=\mathrm{id}$, 
$g|(\Omega\setminus D_0)=f$, 
and 
$d_{\mathrm{Lip}}(f,g)<C$.
The general case then follows by 
applying appropriate translations, 
choosing 
isometric disks $D_{0,j}$ about $f^j(x_0)$ which 
are pairwise disjoint, and applying the special case inductively.
The first part of the construction is standard, and may be found in, {\it e.g.\/}, Hirsch~\cite{HirschBook}.
However, the estimate afterwards, although straightforward, could not be found in the literature, so we include it for completeness.

First, 
take a disk $D_0=D(0,r_0)$ about the origin, contained in $\Omega$. 
Construct a $C^2$-smooth isotopy 
\begin{equation*}
A\colon [0,1]\times D_0\to \Omega^*,
\qquad
A_0=Df(0), \ 
A_1=f,
\end{equation*}
{\it i.e.\/}, a smooth isotopy between $f|_{D'}$ and $Df(0)|_{D'}$.
This may be done via Alexanders' trick.
Next, we take a $C^2$-smooth isotopy 
\begin{equation*}
M\colon [0,1]\to \mathrm{GL}(2,\mathbb{R}), 
\qquad
M_0=Df(0)^{-1}, \
M_1=\mathrm{id} ,
\end{equation*}
{\it i.e.\/}, a smooth isotopy between 
$\mathrm{id}$
and
$Df(0)$. 
Such an isotopy exists as $GL_+(2,\mathbb{R})$ is connected and $f$ is orientation-preserving. 
%%%%%%%%%%%%%%%
\begin{comment}
More specifically, let $K(\theta_0)A(\rho_0)N(n_0)$ denote the $KAN$-decomposition of 
$\det(Df(0))^{1/2} Df(0)^{-1}$.
Recall that $\mathrm{id}$ has $KAN$-decomposition 
given by
$K(\theta_1)A(\rho_1)N(n_1)=K(0)A(1)N(0)$.
Let 
$s(t)=(\theta_t,\rho_t,n_t)\colon [0,1]\to S^1\times \mathbb{R}^2$ 
denote the obvious parametrisation of the straight line between 
the matrices
$\det(Df(0))^{1/2} Df(0)^{-1}$ 
and
$\mathrm{id}$ 
in these coordinates.
Then 
\begin{equation}
M_t=\det(Df(0))^{(t-1)/2} K(\theta_t)A(\rho_t)N(n_t) \qquad t\in[0,1]
\end{equation}
%or $A_t=(1+t[\det(Df(0))^{1/2}-1]) K(\theta_t)A(\rho_t)N(n_t)$.?
\end{comment}
%%%%%%%%%%%%%
Define
\begin{equation*}
F\colon [0,1]\times D_0\to \Omega^*, 
\qquad
F(t,x)=M_t\cdot A(t,x) \ . 
\end{equation*} 
Then
$F_0=\mathrm{id}$ and $F_1=f$.
Thus $F$ is a $C^2$-smooth isotopy between the identity and $f$.
Let $X_t$ denote the time-dependent vector field on (a subset of) $\Omega$
induced by $F_t$.
Let $X=\partial_t\times X_t$ denote the corresponding vector field on
$\bigcup_{t\in [0,1]}(\{t\}\times F_t(D_0))$
induced by the fat isotopy 
$\bar F(t,x)=(t,F_t(x))$.

We construct a new isotopy $G$ as follows. 
Since $\partial D_0$ is compact and $0$ is fixed by the isotopy
%, {\it i.e.\/}, $F_t(0)=0$ for all $t$, 
there exists a positive $r<r_0$ such that 
$|F_t(x)|>r$ for all $x\in \partial D_0$ and $t\in [0,1]$.
Take a bump function 
$\beta\in C^\infty([0,r],\mathbb{R})$ 
such that 
$\beta|_{[0,r/3]}\equiv 0$ 
and 
$\beta|_{[2r/3,r]}\equiv 1$.
Define the vector field on 
$\bigcup_{t\in [0,1]}(\{t\}\times F_t(D_0))$
given by
\begin{equation*}
Y=\left\{\begin{array}{ll}
\partial_t\times \beta X_t & \mbox{in} \ [0,1]\times D(0,r)\\
X & \mbox{otherwise}
\end{array}\right. \ .
\end{equation*}
Since $\beta\equiv 1$ in a neighbourhood of $r$, 
the vector field $Y$ is smooth.
Let 
$Y_t$ denote the corresponding time-dependent vector field, 
{\it i.e.\/}, $Y=\beta X_t$.
Let 
$G\colon [0,1]\times D_0\to \Omega^*$ 
denote the corresponding $C^2$-smooth isotopy.
Denote by $g$ the time-one map.

Set $D_1=D(0,r_1)$ where $r_1=r/3$.
Since $Y_t|_{D_1}\equiv 0$ for all $t$, 
by construction we have that 
$g|_{D_1}\equiv \mathrm{id}$. 
Also, as $g$ agrees with $f$ on a collared 
neighbourhood of $\partial D_0$ in $D_0$ 
it extends smoothly to a map, which we also 
denote by $g$, on the whole of $\Omega$.

Since 
$X=Y$ outside of $[0,1]\times D_0$ and 
$|F_t(x)|>r$ for all $x\in \partial D_0$, 
it follows that 
$F_t(x)=G_t(x)$ for all $x$ in a neighbourhood 
of $\partial D_0$ in $D_0$ and all $t\in [0,1]$.
Hence, in this neighbourhood of $\partial D_0$ 
in $D_0$, the time-one maps agree, {\it i.e.\/}, $f=g$.

It remains to estimate 
$d_{\mathrm{Lip}}(f,g)$ and 
$d_{\mathrm{Lip}}(f^{-1},g^{-1})$.
Observe that there exists a positive $K$, 
depending upon $f$ only, such that for all $t\in [0,1]$,
\begin{equation*}
\max_{x\in D_0}\left\|\partial_s\partial_{x}\left(F_t-G_t\right)(x)\right\|, \  
\max_{x\in F_s(D_0)}\left\|\partial_s\partial_{x}\left(F_t^{-1}-G_t^{-1}\right)(x)\right\|
\leq K \ .
\end{equation*}
This may be seen, for example, by observing that, for each $t\in [0,1]$ and $x\in D_0$, 
\begin{equation*}
\partial_t F_t(x)=X_t\left(F_t(x)\right) \quad  \mbox{and} \quad 
\partial_t G_t(x)=Y_t\left(G_t(x)\right) \ ,
\end{equation*}
so changing the order of differentiations 
and applying the chain rule together with 
the explicit expression for $Y$ in terms 
of $X$ gives the bound.
(Observe that the estimate for the inverses 
requires changing the sign of the time parameter.)

Fix distinct points $x_0$ and $x_1$. 
Define, for all $t\in [0,1]$,
\begin{equation*}
\varphi(t)=\left|\left(F_t(x_0)-G_t(x_0)\right)-\left(F_t(x_1)-G_t(x_1)\right)\right| \ .
\end{equation*}
Let $z\colon [0,|x_0-x_1|]\to \Omega$ be an 
arclength parametrisation of a smooth curve 
in $\Omega$ between $z(0)=x_0$ and $z(|x_0-x_1|)=x_1$.
Since $F_0=G_0$ we find that
\begin{align*}
\varphi(t)
&=\left|\int_0^t \partial_s\left[\left(F_s(x_0)-G_s(x_0)\right)-\left(F_s(x_1)-G_s(x_1)\right)\right]\,ds\right|\\
%&=\left|\int_0^t \partial_s\left(\int_0^{|x_0-x_1|}\partial_{x}(F_s-G_s)(z(u))\dot{z}(u)du\right) ds\right|\\
&=\left|\int_0^t \int_0^{|x_0-x_1|}\partial_s\partial_{x}\left(F_s-G_s\right)(z(u))\dot{z}(u)\,du\,ds\right|\\
&\leq \int_0^t\int_0^{|x_0-x_1|}|\partial_s\partial_{x}\left(F_s-G_s\right)(z(u))\dot{z}(u)|\,du\,ds\\
&\leq \int_0^t\int_0^{|x_0-x_1|}\|\partial_s\partial_{x}\left(F_s-G_s\right)(z(u))\|\,du\,ds\\
&\leq t|x_0-x_1|\max_s\max_{z}\|\partial_s\partial_{x}\left(F_s-G_s\right)(z)\| \ .
\end{align*}
Hence, for each $t\in [0,1]$, 
$\left[F_t-G_t\right]_{\mathrm{Lip}}\leq tK$.
Therefore, setting $t=1$ gives 
\begin{equation*}
[f-g]_{\mathrm{Lip}}\leq K \ .
\end{equation*}
Since $d_{C^0}(f,g)$ can be made arbitrarily small by making $r_0$, the radius of $D_0$, sufficiently small, 
the uniform bound on $d_{\mathrm{Lip}}(f,g)$ follows.
A similar argument also gives the bound for $d_{\mathrm{Lip}}(f^{-1},g^{-1})$.
\end{proof}
\begin{remark}
The above statement holds more generally in the $C^r$-category, $r\geq 2$.
\end{remark}
\begin{proof}[Proof of Theorem~\ref{infsobothm2}]
As mentioned in the paragraph following the statement of Theorem~\ref{infsobothm2},
it suffices to show that, for each positive integer $n$, the set
$\mathcal{G}_{n}$
is dense in the space
$\mathcal{S}^{p,1}_\infty(\Omega,\Omega^*)$.
Thus, given a positive real number $\epsilon$ and a mapping $f$ in $\mathcal{S}^{p,1}_\infty(\Omega,\Omega^*)$,
we wish to show that there exists
$g\in\mathcal{S}^{p,1}_\infty(\Omega,\Omega^*)$
such that
$\rho(f,g)<\epsilon$
and, for some positive integer $k$,
$g^k$ possesses a horseshoe with $n^k$ branches.

%%%%SOBOLEV REPLACEMENT-PERIODIC POINT%%%%
Since 
$f\in \mathcal{S}^{p,1}_\infty(\Omega,\Omega^*)$, 
there exists a point 
$y\in \Omega$ 
which is a forward recurrent point for $f$.
In particular, it is non-wandering.
%\verb=replace recurrent with non-wandering=
By the Sobolev Closing Lemma (Theorem~\ref{thm:sobolevclosing}),
there exists
$g_1\in\mathcal{S}^{p,1}_\infty(\Omega,\Omega^*)$,
a point $x\in \Omega$ and a positive integer $k$ such that $\rho(f,g_1)<\epsilon/4$ and $g_1^k(x)=x$.
Assume that $k$ is the minimal period of $x$, and let $x_j=g_1^j(x)$ for $j=0,1,\ldots,k-1$.

%%%SMOOTH REPLACEMENT-PERIODIC POINT%%%
Applying the Replacement Trick 
(Proposition~\ref{prop:replacement_trick}) 
inductively around each $x_j$ we find that there exists
a map $g_2\in\mathcal{S}^{p,1}(\Omega,\Omega^*)$, 
and 
topological disks $\mathcal{O}_j^*$ about $x_{j+1}$ (where addition is taken mod $k$)
such that 
\begin{enumerate}
\item[(1)] 
$\mathcal{O}_j^*$, 
and its preimage 
$\mathcal{O}_j=g_1^{-1}(\mathcal{O}_j^*)$, 
have diameter less that $\epsilon/4$
\item[(2)]
the collection of sets $\mathcal{O}_j$, $j=0,1,\ldots,k-1$, are pairwise disjoint,
\item[(3)] 
$g_2$ agrees with $g_1$ outside of $\bigcup_{j=0}^{k-1} \mathcal{O}_j$, 
\item[(4)] 
$g_2\colon(\mathcal{O}_j,x_j)\to(\mathcal{O}_j^*,x_{j+1})$ is smooth for $j=0,1,\ldots,k-1$,
\item[(5)]
$\rho(g_1,g_2)<\epsilon/4$. 
\end{enumerate}
%%%IDENTITY REPLACEMENT-PERIODIC DISK%%%
Applying Proposition~\ref{prop:blow-up}
to the restriction $g_2|_{\bigcup_{j=0}^{k-1} \mathcal{O}_j}$, 
we find that there exists a positive real number $C$ 
such that for any positive $r$ sufficiently small, 
there exists a positive
$r'<r$ and 
a $C^2$-smooth embedding 
$g_3\colon \bigcup_{j=0}^{k-1} \mathcal{O}_j\to\bigcup_{j=0}^{k-1}\mathcal{O}_j^*$ 
such that 
\begin{enumerate}
\item[(6)]
$D(x_j,r)\subset \mathcal{O}_j$ for all $j$,
\item[(7)]
$g_3|_{D(x_j,r')}$ is a translation, for all $j$,
\item[(8)]
$g_3$ agrees with $g_2$ outside of $\bigcup_{j=0}^{k-1} D(x_j,r)$,  
\item[(9)]
$\bigl[g_2-g_3\bigr]_{\mathrm{Lip}}, \bigl[g_2^{-1}-g_3^{-1}\bigr]_{\mathrm{Lip}}<C$.
\end{enumerate}
Observe that we may extend $g_3$ to $\Omega$ by 
setting it equal to $g_2$ outside $\bigcup_{j=0}^{k-1}\mathcal{O}_j$.
We wish to estimate the $\rho$-distance between $g_2$ and $g_3$.
First, by shrinking $r$ if necessary we may assume that 
$d_{C^0}(g_2,g_3), d_{C^0}(g_2^{-1},g_3^{-1})<\epsilon/8$.
Since $g_2$ and $g_3$ agree outside $\bigcup_{j=0}^{k-1} D(x_j,r)$ and 
since there exists $K$ such that, for any smooth map $G$ on a compact domain $\Omega$, 
$|DG(z)|\leq K[G]_{\mathrm{Lip}}$ for any $z\in\Omega$, we find
\begin{align*}
\int_\Omega \left|Dg_2-Dg_3\right|^p\,d\mu
&=\int_{\bigcup_{j=0}^{k-1} D(x_j,r)}\left|Dg_2-Dg_3\right|^p\,d\mu\\
&\leq K^p\bigl[g_2-g_3\bigr]^p_{\mathrm{Lip}}\sum_{j=0}^{k-1} \mu\left(D(x_j,r)\right)\\
&=K^pC^p\pi k r^2 \ .
\end{align*}
Hence, by shrinking $r$ again we may assume that 
$\left[g_2-g_3\right]_{W^{1,p},\Omega}<\epsilon/8$.
Adopting the same argument for the inverse, we may therefore assume that $r$ has been chosen sufficiently small so that 
$\rho(g_2,g_3)<\epsilon/4$.

%%%HORSESHOE REPLACEMENT-PERIODIC DISK%%%
Now define $g$ as follows.
Let $h$ denote a standard $n$-branched horseshoe of the unit disk $D(0,1)$, fixing a neighbourhood of the boundary.
We may assume that, for some constant $c$ independent of $n$, $[h]_{\mathrm{Lip}},[h^{-1}]_{\mathrm{Lip}}\leq cn$. 
Choose $r''<r'$ and let $a_j\colon D(0,1)\to D(x_j,r')$ be given by $a_j(z)=r''z+x_j$.
Define
\begin{equation*}
g(z)=\left\{\begin{array}{ll}
g_3\circ a_j\circ h\circ a_j^{-1}(z) & z\in D(x_j,r''), \ \mbox{some} \ j\\
g_3 & \mbox{otherwise}
\end{array}\right. \ .
\end{equation*} 
Since Lipschitz constants are invariant under affine rescaling, we find that
\begin{equation*}
[g_3-g]_{\mathrm{Lip}}, [g_3^{-1}-g^{-1}]_{\mathrm{Lip}}\leq 1+cn \ .
\end{equation*}
By the same argument as before,
since $g_3$ and $g$ agree outside $\bigcup_{j=0}^{k-1} D(x_j,r'')$, 
%and since there exists $K$ such that, for any smooth map $G$ on a compact domain $\Omega$, 
%$|DG(z)|\leq K[G]_{\mathrm{Lip}}$ for any $z\in\Omega$, 
for the constant $K$ defined as above,
we find
\begin{align*}
\int_\Omega |Dg_3-Dg|^p\,d\mu
&=\int_{\bigcup_{j=0}^{k-1} D(x_j,r'')}|Dg_3-Dg|^p\,d\mu\\
&\leq K[g_3-g]_{\mathrm{Lip}}^p\sum_{j=0}^{k-1}\mu(D(x_j,r''))\\
&\leq K(1+cn)^p\pi k(r'')^2 \ .
\end{align*}
A similar estimate holds for the inverses.
Therefore, choosing $r''$ sufficiently small, 
we may assume that $\rho(g_3,g)<\epsilon/4$.
Thus
\begin{equation*}
\rho(f,g)\leq \rho(f,g_1)+\rho(g_1,g_2)+\rho(g_2,g_3)+\rho(g_3,g)<\epsilon \ .
\end{equation*}
Finally, observe that since $g|_{D(x_j,r'')}$ is a translation from $x_j$ to $x_{j+1}$, it follows that 
$g^k|_{D(x_0,r'')}$ is topologically conjugate to $h^k$.
Thus $g$ lies in $\mathcal{G}_{n}$ and $\rho(f,g)<\epsilon$.
Hence $\mathcal{G}_{n}$ is dense, and the theorem follows.
\end{proof}
\begin{proof}[Proof of Theorem~\ref{infsobothm1}]
The proof is totally analogous to the proof of Theorem~\ref{infsobothm2} above.
Namely, let $\mathcal{G}_n$ denote the subset of $g\in\mathcal{S}^{p,1}(M)$ 
for which some iterate $g^k$ possesses an $n^k$-branched horseshoe.
By the argument preceding the statement of Theorem~\ref{infsobothm2}, 
$\mathcal{G}_n$ is open in $\mathcal{S}^{p,1}(M)$.
Thus, to prove Theorem~\ref{infsobothm1} it suffices to show that $\mathcal{G}_n$ is dense.

Take $f\in \mathcal{S}^{p,1}(M)$ and a neighbourhood $\mathcal{N}$ of $f$ in $\mathcal{S}^{p,1}(M)$.
As $M$ is compact, the non-wandering set of $f$ is non-empty.
Take a non-wandering point $y$ in $M$ and apply the Sobolev Closing Lemma.
Then there exists $g_1\in \mathcal{N}$ with periodic point $x$ of some mimimal period $k$.
For $j=0,1,\ldots,k-1$, 
take charts $(U_j,\varphi_j)$ about $x_j=g_1^j(x)$
with pairwise disjoint domains and ranges.
Define
\begin{equation*}
\Omega=\bigcup_{j=0}^{k-1} \varphi_j(U_j\cap g_1^{-1} (U_{j+1})) \quad \mbox{and} \quad
\Omega^*=\bigcup_{j=0}^{k-1} \varphi_{j+1}(g_1(U_j)\cap U_{j+1}) \ ,
\end{equation*}
where, as usual, addition is taken modulo $k$. 
Consider the map $G_1\colon\Omega\to\Omega^*$ which, for $j=1,2,\ldots,k-1$, is defined on $\varphi_j(U_j\cap g_1^{-1}(U_{j+1}))$ by 
\begin{equation*}
G_1=
\varphi_{j+1}\circ g_1\circ \varphi_j^{-1} \ .
\end{equation*}
Then this defines a map in 
$\mathcal{S}^{p,1}(\Omega,\Omega^*)$.
In fact it lies in 
$\mathcal{S}^{p,1}_\infty(\Omega,\Omega^*)$ 
as it possesses a periodic orbit.
Applying Theorem~\ref{infsobothm2}, 
we find that for each positive $\epsilon$ there exists 
$G_2\in \mathcal{S}^{p,1}(\Omega,\Omega^*)$ with 
$\rho(G_1,G_2)<\epsilon$, 
such that in a neighbourhood of $\varphi_0(x_0)$, 
$G_2$ possesses an $n^k$-branched horseshoe, 
and
outside of this neighbourhood $G_2$ coincides with $G_1$.
Consequently, $G_2$ induces a Sobolev homeomorphism 
$g_2$ in $\mathcal{S}^{p,1}(M)$, also with the property 
that $g_2^k$ possesses an $n^k$-branched horseshoe. 
Hence $g_2$ lies in $\mathcal{G}_n$.
Moreover, from the definition of the Sobolev-Whitney topology on $\mathcal{S}^{p,1}(M)$, 
for $\epsilon$ sufficiently small, $g_2$ can be chosen to lie in any neighbourhood of $g_1$. 
Thus we may assume $g_2$ lies in $\mathcal{N}$.
Consequently $\mathcal{G}_n$ is dense in $\mathcal{S}^{p,1}(M)$ and the theorem follows.
\end{proof}

%%%%%%%%%%%%%%%%%%%%%%%%%%%%%%%%
\subsubsection{Second argument.}
%%%%%%%%%%%%%%%%%%%%%%%%%%%%%%%%
The argument in the H\"older case 
can be adapted to give a proof of Theorem~\ref{thm:generic_Sobolev_homeomorphism}.
As in the H\"older case, this follows directly from the following result.
\begin{theorem}
Let $M$ be a compact manifold of dimension $d$.
Assume either 
\begin{enumerate}
\item[(a)]
$d=2$ and $1\leq p, p^*<\infty$;
\item[(b)]
$d>2$ and $d-1<p, p^*<\infty$.
\end{enumerate}
Let $f\in \mathcal{S}^{p,p^*}(M)$.
For each neighbourhood $\mathcal{N}$ of $f$ 
in $\mathcal{S}^{p,p^*}(M)$ and each positive 
integer $N$ there exists 
$g\in \mathcal{S}^{p,p^*}(M)$ 
such that
\begin{itemize}
\item[(i)]
$g\in\mathcal{N}$
\item[(ii)]
there exists a positive integer $k_0$, a topological solid cylinder $S$ in $M$ and solid sub-cylinders $S_1, S_2,\ldots, S_{N^{k_0}}$
such that $g^{k_0}$ maps $S_j$ across $S$ for $j=1,2,\ldots, N^{k_0}$ 
\end{itemize}
The second property implies that $h_{\mathrm{top}}(g)\geq \log N$ and that this property is satisfied in an open neighbourhood of $g$.
\end{theorem}
\begin{proof}
The strategy of proof is the same as for the H\"older case 
(Theorem~\ref{thm:open+dense_entropy>logN-little_Holder}). 
Specifically, the notation, setup and construction of the perturbation 
will be the same as in that case.
%Theorem~\ref{thm:open+dense_entropy>logN-little_Holder}. 
We will only remark on the necessary changes, 
such as the choice of sizes of neighbourhoods, etc., 
and will go through the estimates for the size of the perturbation in more detail.

\vspace{5pt}

\noindent
{\sl Setup:\/}
We may assume, by making a small perturbation if necessary, that $f$ is bi-Lipschitz.
Take a finite collection of sub-basic sets
\begin{equation}\label{eq:subbasic-Sobolev-entropy}
\mathcal{N}_{\mathbb{W}^{1,p},\mathbb{W}^{1,p^*}}(f;(U_n,\varphi_n), (V_n,\psi_n), K_n,L_n,\varepsilon_n)
\end{equation}
We will use the notation $f_n$, $\mathsf{U}_n$, $\mathsf{V}_n$, $\mathsf{W}_n$, $c_1$, etc., as before.
Thus $c_1$ satisfies inequalities~\eqref{ineq:transition-Lip1} and~\eqref{ineq:transition-Lip2} for $f$ and $\mathsf{U}_n$,
as well as their counterparts for $f^{-1}$ and $\mathsf{V}_n$.
Given a perturbation $g$ of $f$ and an index $n$, 
let $g_n$ denote the map $g$ expressed in the pair of charts $(U_n,\varphi_n)$ and $(V_n,\psi_n)$.

Fix a positive real number $\epsilon$.
This will denote the order of the size of the perturbation.
Let $\delta$ be a positive real number.
This will denote the size of the support of the perturbation.
Take $\delta$ sufficiently small so that 
properties (a) and (b) from the proof of the H\"older case are satisfied, together with the following property
\begin{enumerate}
%\item[(a)]
%$\delta$ is less than the Lebesgue number of the common refinement of the finite open covers $\{U_n\}$ and $\{V_n\}$.
%Thus any ball of radius $\delta$ or less lies in some set of the form $U_m\cap U_n$.
%\item[(b)]
%$\mathsf{U}_n$ contains a $2\delta$-neighbourhood of $K_n$, and
%$\mathsf{V}_n$ contains a $2\delta$-neighbourhood of $L_n$.
%Thus, for any $x\in M$, either
%$B_M(x,\delta)\cap K_n=\emptyset$, or
%$B_M(x,\delta)\subset \mathsf{U}_n$.,
%and similarly for $L_n$ and $\mathsf{V}_n$.
\item[(c')]
for each index $n$, whenever $x$ lies in a $\delta$-neighbourhood of $K_n$,
\begin{equation*}
\int_{\varphi_n(B(x,\delta))}|Df_n|^p\,d\mu \leq \epsilon
\end{equation*}
and similarly, whenever $x$ lies in a $\delta$-neighbourhood of $L_n$,
\begin{equation*}
\int_{\psi_n(B(x,\delta))}|Df_n^{-1}|^p\,d\mu \leq \epsilon \ .
\end{equation*}
\end{enumerate}
(Observe that this is possible 
%by the ACL property of Sobolev functions, and 
as the $\delta$-neighbourhoods of $K_n$ and $L_n$ 
are contained in the compact neighbourhoods 
$\mathsf{U}_n$ and $\mathsf{V}_n$ respectively.)

\vspace{5pt}

\noindent
{\sl Support of the perturbation:\/}
Observe that Claim 2 from the proof of Theorem~\ref{thm:open+dense_entropy>logN-little_Holder} 
also holds for Sobolev mappings.
However, rather than Claim 2(5) we will require the following, 
which holds by taking $r_1$ sufficiently small
\begin{enumerate}
\item[(5')]
the following sets are pairwise disjoint
\begin{equation*}
\varphi_0^{-1}\left(E(x_0^0,x_0^{k_0};c\cdot r_0)\right), 
\varphi_1^{-1}(B(x^1,r_1)),\ldots, \varphi_{k_0-1}^{-1}(B(x^{k_0-1},r_1)) \ .
\end{equation*}
In fact, if $\mathcal{B}$ denotes the collection of all such sets then, for any $n$,
\begin{equation*}
\sum_{B\in \mathcal{B} : f^{-1}B\subset\mathsf{U}_n} \int_{\varphi_n(f^{-1}B)} |Df_n|^p\,d\mu <\epsilon
\end{equation*}
and
\begin{equation*}
\sum_{B\in\mathcal{B} : B\subset \mathsf{V}_n} \int_{\psi_n(B)} |Df_n|^p\,d\mu <\epsilon \ .
\end{equation*}
\end{enumerate}

\vspace{5pt}

\noindent
{\sl Construction of the perturbation:\/}
Define $g$ as in the H\"older case, for the choice of of neighbourhoods (or more specifically, the choice of $r_1$) as described above.
%%%%%%%%%%%%%%%
\begin{comment}
Observe that
\begin{equation}\label{def:g-Sobolev-entropy}
g=\left\{
\begin{array}{ll}
\varphi_{1}^{-1}\circ\phi^{0}\circ\varphi_{1}\circ f\circ\varphi_{0}^{-1}\circ \phi\circ\varphi_{0} & \mbox{in} \ E_M\\
\varphi_{k+1}^{-1}\circ\phi^{k}\circ \varphi_{k+1}\circ f & \mbox{in} \ f^{-1}(B_M^{k+1}), \ \mbox{for} \ 0<k<k_0\\
f & \mbox{elsewhere}
\end{array}
\right.
\end{equation}
\end{comment}
%%%%%%%%%%%%%
Then $g$ is a homeomorphism.
Pre- and post-composing by smooth mappings preserves the space of bi-Sobolev mappings.
Thus the map $g$ also lies in $\mathcal{S}^{p,p^*}(M)$.
As in the H\"older case, when $g$ is expressed in the pair of 
charts $(U_n,\varphi_n)$ and $(V_n,\psi_n)$, for each index $n$, we also have
\begin{equation}\label{def:gn-Sobolev-entropy}
g_n=\left\{\begin{array}{ll}
\phi_n^0\circ f_n\circ \phi_n & \mbox{in} \ \varphi_n(E_M)\\
\phi_n^k\circ f_n & \mbox{in} \ \varphi_n(f^{-1} (B_M^{k+1})), \ \mbox{for} \ 0<k<k_0\\
f_n & \mbox{elsewhere}
\end{array}\right. \ ,
\end{equation}
where
\begin{equation*}%\label{def:phikn+phin-Sobolev}
\phi_n^k=\psi_n\circ\varphi_{k+1}^{-1}\circ \phi^k\circ \varphi_{k+1}\circ \psi_n^{-1} 
\quad
\mbox{and}
\quad
\phi_n=\varphi_n\circ\varphi_0^{-1}\circ\phi\circ\varphi_0\circ\varphi_n^{-1} \ .
\end{equation*}
\vspace{5pt}

\noindent
{\sl Size of the perturbation:\/}
We will show that if $\epsilon$ is taken sufficiently small, then $g$ lies in each sub-basic set
~\eqref{eq:subbasic-Sobolev-entropy}
Fix an index $n$.
%%%%%%%%%%%%%%%
\begin{comment}
Recall that
\begin{align}
\|f_n-g_n\|_{\mathbb{W}^{1,p}(\varphi_n(K_n),\mathbb{R}^d)}
=
\|f_n-g_n\|_{C^0(\varphi_n(K_n))}
+\left(\int_{\varphi_n(K_n)}|Df_n-Dg_n|^{p}\,d\mu\right)^{\frac{1}{p}} \ .
\end{align}
\end{comment}
%%%%%%%%%%%%%
%%%C^0-DISTANCE%%%
The argument as in the H\"older case, shows that
\begin{equation}
\|f_n-g_n\|_{C^0(\varphi_n(K_n))}
=O(\epsilon)
=\|f_n^{-1}-g_n^{-1}\|_{C^0(\psi_n(L_n))} \ .
\end{equation}
Thus we only need to consider the 
$W^{1,p}$-semi-norm of $f_n-g_n$ together with the
$W^{1,p^*}$-semi-norm of $f_n^{-1}-g_n^{-1}$.
%%%W1p-SEMI-NORM%%% 
Consider the $W^{1,p}$-semi-norm.
Observe that each preimage
$f^{-1}(B_M^{k+1})$ 
lies in a ball of radius $\delta$, and thus, by property (b) above, each 
$f^{-1}(B_M^{k+1})$ either lies inside $\mathsf{U}_n$ or is disjoint from $K_n$. 
Similarly for $E_M$.
Let $\mathcal{B}_n$ denote the collection of all sets 
$E_M$ and 
$f^{-1}(B_M^{k+1})$, $0<k<k_0$, 
that are contained in $\mathsf{U}_n$.
Consequently, by~\eqref{def:gn-Sobolev-entropy},
on $\mathsf{U}_n$ the map $g_n$ is expressible as either
$g_n=\Phi_n\circ f_n\circ \phi_n$ 
or 
$g_n=\Phi_n\circ f_n$,
depending upon whether $E_M$ lies in $\mathsf{U}_n$ or not. 
Here $\Phi_n$ is the composition of all $\phi_n^k$ such that $f^{-1}(B_M^{k+1})$ is contained in $\mathsf{U}_n$.
(Observe that, as the supports of the $\phi^k$ are pairwise disjoint, the maps $\phi^k$ commute.
Thus the order in which the $\phi_n^k$ are composed does not matter.)
Then
\begin{align}
[f_n-g_n]_{W^{1,p},\varphi_n(K_n)}
&\leq
[f_n-g_n]_{W^{1,p},\varphi_n(\mathsf{U}_n)}\notag\\
&=
\left(\sum_{B\in\mathcal{B}_n}\int_{\varphi_n(B)}|Df_n-Dg_n|^p\,d\mu\right)^{\frac{1}{p}}\notag\\
&\leq
\sum_{B\in\mathcal{B}_n}\left(\int_{\varphi_n(B)}|Df_n-Dg_n|^p\,d\mu\right)^{\frac{1}{p}} \ . \label{eq:Lp-dist-fn-gn}
\end{align}
First consider the case when $B\in\mathcal{B}_n$ is of the form $f^{-1}(B_M^{k+1})$.
Since $\phi_n^k$ is smooth and of compact support it follows that $|\mathrm{id}-D\phi_n^k|$ is bounded, not just essentially bounded.
By H\"older's inequality
\begin{align}
&
\left(\int_{\varphi_n(f^{-1}(B_M^{k+1}))}|Df_n-Dg_n|^p\,d\mu\right)^{\frac{1}{p}}\notag\\
&\leq 
\left(\int_{\varphi_n(f^{-1}(B_M^{k+1}))}|\mathrm{id}-D\phi_n^k\circ f_n|^p|Df_n|^p\,d\mu\right)^{\frac{1}{p}}\notag\\
&\leq
\|\mathrm{id}-D\phi_n^k\|_{L^\infty}
\left(\int_{\varphi_n(f^{-1}(B_M^{k+1}))}|Df_n|^p\,d\mu\right)^{\frac{1}{p}} \ . \label{ineq:Bk}
\end{align}
Before consider the case when $B$ in $\mathcal{B}_n$ equals $E_M$, as in~\eqref{eq:Lp-dist-fn-gn}, observe that the mapping
$h_n=\phi_n^0\circ f_n$ is also Sobolev.
Thus, by the chain rule (Lemma~\ref{chainrule} (ii)) we have
\begin{align}
\left(\int_{\varphi_n(E_M)}|Dh_n|^p\,d\mu\right)^{\frac{1}{p}}
&=
\left(\int_{\varphi_n(E_M)}|D\phi_n^0\circ f_n|^p |Df_n|^p\,d\mu\right)^{\frac{1}{p}}\notag\\
&\leq
\|D\phi_n^{0}\|_{L^\infty}\left(\int_{\varphi_n(E_M)}|Df_n|^p\,d\mu\right)^{\frac{1}{p}} \ . \notag
\end{align}
Now return to the case $B=E_M$.
Observe that
\begin{equation}
\left(\int_{\varphi_n(E_M)}|Df_n-Dg_n|^p\,d\mu\right)^{\frac{1}{p}}\leq A_1+A_2+A_3 \ ,\notag 
\end{equation}
where
\begin{align}
A_1
&=\left(\int_{\varphi_n(E_M)}|Df_n-Dh_n|^p\,d\mu\right)^{\frac{1}{p}} \ , \notag\\
A_2
&=\left(\int_{\varphi_n(E_M)}|Dh_n-Dh_n\circ\phi_n|^p\,d\mu\right)^{\frac{1}{p}} \ ,\notag\\
A_3
&=\left(\int_{\varphi_n(E_M)}|Dh_n\circ\phi_n-Dg_n|^p\,d\mu\right)^{\frac{1}{p}} \ . \notag
\end{align}
The same argument as for inequality~\eqref{ineq:Bk} gives
\begin{align}
A_1
&=
\left(\int_{\varphi_n(E_M)}|Df_n-Dh_n|^p\,d\mu\right)^{\frac{1}{p}}\notag\\
&\leq 
\|\mathrm{id}-D\phi_n^0\circ f_n\|_{L^\infty}
\left(\int_{\varphi_n(E_M)}|Df_n|^p\,d\mu\right)^{\frac{1}{p}} \ ,\notag
\end{align}
while the triangle inequality, together with the change of variables formula and H\"older's inequality,
implies that
\begin{align}
A_2
&\leq 
\left(\int_{\varphi_n(E_M)}|Dh_n|^p\,d\mu\right)^{\frac{1}{p}}
+
\left(\int_{\phi_n^{-1}\circ \varphi_n(E_M)}|Dh_n|^p J_{\phi_n^{-1}}\,d\mu\right)^{\frac{1}{p}}\notag\\
&\leq
\left(1+\|J_{\phi_n^{-1}}\|_{L^\infty}^{\frac{1}{p}}\right)
\|D\phi_n^0\|_{L^\infty}
\left(\int_{\varphi_n(E_M)}|Df_n|^p\,d\mu\right)^{\frac{1}{p}} \ .\notag
\end{align}
The chain rule (Lemma~\ref{chainrule} (i)) together with H\"older's inequality also implies that
\begin{align}
A_3
&=
\left(\int_{\varphi_n(E_M)} |Dh_n\circ\phi_n-Dh_n\circ\phi_n D\phi_n|^p\,d\mu\right)^{\frac{1}{p}} \notag\\
&\leq
\|\mathrm{id}-D\phi_n\|_{L^\infty} \|D\phi_n^0\|_{L^\infty}\left(\int_{\varphi_n(E_M)}|Df_n|^p\,d\mu\right)^{\frac{1}{p}}  \ .\notag
\end{align}
By~\eqref{ineq:transition-Lip1} and~\eqref{ineq:transition-Lip2},
together with Lemma~\ref{lem:E-perturbation} and Corollary~\ref{cor:transport_cylinder+horseshoe_lip}, 
we know that
$\|D\phi_n^k\|_{L^\infty}$, 
$\|D\phi_n\|_{L^\infty}$, and
$\|J_{\phi_n^{-1}}\|_{L^\infty}$ 
(each taken on the set $\varphi_n(\mathsf{U}_n)$) 
are bounded from above by
some constant independent of $\epsilon$.
Therefore, by the above considerations together with 
%absolute continuity
property (c') in the case $B=E_M$ and 
property (5') in the case $B=f^{-1}(B^{k+1}_M)$,
we find that
\begin{align}
[f_n-g_n]_{W^{1,p},\varphi_n(K_n)}
=
O\left(\sum_{B\in\mathcal{B}_n}\left(\int_{B} |Df_n|^p\,d\mu\right)^{\frac{1}{p}}\right) 
=
O(\epsilon^{\frac{1}{p}}) \ .\notag
\end{align}
As the above argument deals with both pre- and post-composition by smooth maps, the same argument gives
$[f_n^{-1}-g_n^{-1}]_{W^{1,p^*},\psi_n(L_n)}=O(\epsilon^{\frac{1}{p^*}})$
and hence, as $\epsilon$ was arbitrary, the result follows.
\end{proof}
We also get the analogue of Yano's 
Corollary on conjugacy classes, and of 
Corollary~\ref{cor:generic_not_conjugate-Holder} 
above, in the Sobolev setting.
\begin{corollary}\label{cor:generic_not_conjugate-Sobolev}
Let $M$ be a compact manifold of dimension $d$.
Assume either
\begin{enumerate}
\item[(a)]
$d=2$ and $1\leq p, p^*<\infty$;
\item[(b)]
$d>2$ and $d-1<p, p^*<\infty$.  
\end{enumerate}
A generic homeomorphism in $\mathcal{S}^{p,p^*}(M)$ 
is not topologically conjugate to any diffeomorphism (or any bi-Lipschitz homeomorphism).
\end{corollary}
The arguments preceding Corollary~\ref{cor:generic_measures_max_entropy-Holder} also gives us the following.
\begin{corollary}\label{cor:generic_measures_max_entropy-Sobolev}
Let $M$ be a compact manifold of dimension $d$.
Assume either
\begin{enumerate}
\item[(a)]
$d=2$ and $1\leq p, p^*<\infty$;
\item[(b)]
$d>2$ and $d-1<p, p^*<\infty$.  
\end{enumerate}
A generic homeomorphism in $\mathcal{S}^{p,p^*}(M)$ 
has uncountably many measures of maximal entropy.
\end{corollary}
As in the H\"older case, by the argument following Corollary~\ref{cor:generic_measures_max_entropy-Holder} we also get the following.
\begin{corollary}\label{cor:generic_eqm_states-Sobolev}
Let $M$ be a compact manifold of dimension $d$.
Assume either
\begin{enumerate}
\item[(a)]
$d=2$ and $1\leq p, p^*<\infty$;
\item[(b)]
$d>2$ and $d-1<p, p^*<\infty$.  
\end{enumerate}
For a generic homeomorphism $f$ in $\mathcal{S}^{p,p^*}(M)$ 
the set of equilibrium states of $(f,\phi)$, is independent of $\phi\in C^0(X,\mathbb{R})$.
In fact, generically the set of equilibrium states, for any $\phi\in C^0(X,\mathbb{R})$ coincides with the set of measures of maximal entropy.

\end{corollary}
%
%
%%%%%%%%%%%%%%%%%%%%%%%%%%%%%%%%%%%%%%%%%%%%%%%
\section{Concluding Remarks and Open Problems.}\label{sect:open_problems}
The following is a list of comments and open problems suggested by this work.
\begin{enumerate}
%\item 
%(M. Benedicks) Is there a `sharp' formulation of the necessary smoothness for generic infinite entropy, in a similar manner to the 
%Dini condition on the convergence of Fourier series of a given function?
\item
To prove genericity of infinite entropy in a given smoothness category {\sc Cat} there 
is also a ``na\"{i}ve'' perturbation argument 
whereby we perturb a given initial map to some map with
entropy $\log n$ or greater. 
The strategy is as follows.  
First, the initial map will possess a non-wandering point so,
by the Closing Lemma, this map can be perturbed to a map with a periodic orbit of some period $k$.
Next, we perturb this new map by `blowing-up' this periodic point 
to a periodic ball, also of period $k$, on which the $k$th iterate is the identity.
Finally, a horseshoe with $n^k$ branches is `glued-in' to this periodic ball.
Then the $k$th iterate of this final map will have a horseshoe with entropy $\log n^k$, and thus the original map will have entropy at least $\log n$.
This was the approach in Section~\ref{subsubsect:first_argument}.
For details in the continuous case, see also~\cite[Proof of Lemma 3.1]{GW}.
However, to `blow-up' we require the Annulus Conjecture to hold in {\sc Cat}.
\begin{Annulus}
In any dimension $d$, for any {\sc Cat}-embedding $\varphi\colon B^d\to\mathbb{R}^d$ 
with the property that $\varphi(B^d)\Subset B^d$, the sets 
$B^d\setminus \varphi(B^d)$ and $[0,1)\times S^{d-1}$ are {\sc Cat}-homeomorphic.
\end{Annulus}
In the continuous category, this was shown by Quinn, following work by Kirby and others~\cite{Edwards1984}.
Sullivan~\cite{Sullivan1979} proved that the Annulus Conjecture holds for the quasiconformal and bi-Lipschitz categories. 
(See also~\cite{TukiaVaisala1982a,TukiaVaisala1982b}.)
Therefore we ask: 

\medskip

\noindent
{\sl Question:}
Does the Annulus Conjecture hold for bi-H\"older or bi-Sobolev mappings? 

\medskip

If the answer is yes, for either bi-H\"older or bi-Sobolev, then we can recover our results on genericity of infinite entropy more easily 
via the argument in~\cite[Proof of Lemma 3.1]{GW} and the Closing Lemmas proved in this paper.
\item
Via Morrey's inequality or otherwise, one can show that, on a bounded open set $\Omega$ in $\mathbb{R}^d$ with appropriate ({\it i.e.\/}, Lipschitz) boundary, 
$W^{1,p}(\Omega)$ embeds continuously into $C^{\alpha}(\Omega)$ where $\alpha=1-\frac{d}{p}$.
\begin{itemize}
\item
Does a similar embedding result hold for the spaces of bi-Sobolev and bi-H\"older homeomorphisms, with the Sobolev-Whitney and H\"older-Whitney topologies, introduced in this paper?
\item
If such an embedding exists, is its image dense?
If so, most of the results in Part I of this paper would follow from those in Part II. 
\end{itemize}
\item
Can the results in this paper be generalised to arbitrary metric measure spaces where the measure satisfies a doubling condition?
Note that H\"older maps are easily defined for these space, while a definition of Sobolev maps has been given by Haj\l{}asz (see, {\it e.g.\/},~\cite{Hajlasz}).
\item
Does the Sobolev Closing Lemma hold for $1\leq p, p^*<d$?
\item
Can either the Sobolev or H\"older Closing Lemma be used to give, via an approximation argument, a new proof of Pugh's $C^1$-Closing Lemma?
\end{enumerate} 

\subsection*{Acknowledgements}
We would like to thank everyone who supported us throughout this project.
We thank IME-USP, ICERM (Brown University), Imperial College London, and the CUNY Graduate Center for their hospitality.
Finally, special thanks go to 
Charles Pugh for several useful remarks and references to the literature concerning the Closing Lemma,
Michael Benedicks for useful discussions on maps with low regularity, 
Fr\'ed\'eric H\'elein for giving one of the authors (P.H.) a copy of the preprint~\cite{SchoenUhlenbeck},
and Dennis Sullivan and \'{E}tienne Ghys for their comments and questions.
\begin{appendix}
%
%
%
%
%
%
%
%
%
%%%%%%%%%%%%%%%%%%%%%%%%%%%%%%%%%%%%%%%%%%%%%%%%%%%%%%%%%%%%%%%%%%%%%%%%%%%%%%%%%%%%%%%%%%%%
%%%%%%%%%%%%%%%%%%%%%%%%%%%%%%%%%%%%%%%%%%%%%%%%%%%%%%%%%%%%%%%%%%%%%%%%%%%%%%%%%%%%%%%%%%%%
%%%%%%%%%%%%%%%%%%%%%%%%%%%%%%%%%%%%%%%%%%%%%%%%%%%%%%%%%%%%%%%%%%%%%%%%%%%%%%%%%%%%%%%%%%%%
%%%%%%%%%%%%%%%%%%%%%%%%%%%%%%%%%%%%%%%%%%%%%%%%%%%%%%%%%%%%%%%%%%%%%%%%%%%%%%%%%%%%%%%%%%%%
%
%
%
%
%
%
%
%
%%%%%%%%%%%%%%%%%%%%%%%%%%%%%%%%%%%%%%%%%%%%%%%%%%%%%%
%%%%%%%%%%%%%%%%%%%%%%%%%%%%%%%%%%%%%%%%%%%%%%%%%%%%%%
\section{Planar homeomorphisms with infinite topological entropy}

In this section we build a family of examples of 
orientation-preserving homeomorphisms of the plane 
having compact support and infinite topological entropy 
which are bi-H\"older of every exponent. More precisely, 
we construct for each $p>0$ an orientation-preserving 
homeomorphism $f\colon \mathbb{R}^2\to\mathbb{R}^2$, 
with compact support and $h_{\mathrm{top}}(f)=\infty$, 
such that $f,f^{-1}\in C^{\Lambda_p}(\mathbb{R}^2)$, 
where $\Lambda_p$ is the modulus of continuity $\Lambda_p(t)=t|\log{t}|^{p}$. 
In particular, $f$ and $f^{-1}$ are $\alpha$-H\"older continuous for all $0<\alpha<1$. 

\subsubsection*{Terminology} Given a domain $\Omega\subseteq \mathbb{R}^2\equiv \mathbb{C}$, we denote by $W^{1,2}_{\mathrm{loc}}(\Omega)$ 
the Sobolev space of maps $\phi:\Omega\to \mathbb{R}^2$ such that the partial derivatives 
$\phi_z,\phi_{\overline{z}}$ exist Lebesgue almost everywhere in $\Omega$ and 
are locally in $L^2$. The {\it distortion\/} (or {\it dilatation\/}) of a map $\phi:\Omega\to \mathbb{R}^2$ is the function $K_\phi$ given by
\begin{equation*}
K_\phi
\;=\;
\frac{|\phi_z|+|\phi_{\overline{z}}|}{|\phi_z|-|\phi_{\overline{z}}|} \ .
\end{equation*}
The map $\phi$ is said to have {\it finite distortion\/} if $K_\phi(z)<\infty$ for Lebesgue a.e. $z\in \Omega$. We say that 
$\phi$ is a {\it regular map with finite distortion\/} if $\phi\in W^{1,2}_{\mathrm{loc}}(\Omega)$, its Jacobian $J_\phi=|\phi_z|^2-|\phi_{\overline{z}}|^2$ 
is locally integrable and $\phi$ has finite distortion. Such a map is differentiable almost everywhere, and 
the inequality $|D\phi(z)|^2\leq K_\phi(z)J_\phi(z)$ holds for Lebesgue a.e. $z\in \Omega$. 
Examples of regular maps with finite distortion include local diffeomorphisms, as well as quasiconformal 
homeomorphisms of (some portion of) the plane. (See, {\it e.g.},~\cite{HenclKoskelaBook}.)

\subsubsection*{Tool} The main tool in the construction below is the following result due to Goldstein and Voldop'yanov, 
a proof of which can be found in \cite[pp.~530--534]{AIM}.

\begin{proposition}\label{regmap}
Let $f\colon\Omega\to \mathbb{R}^2$ be a regular map of finite distortion. Then $f$ is continuous, and for all $z,w\in \Omega$ we have
\begin{equation}\label{GVineq}
\left|f(z)-f(w)\right|^2
\;\leq\; 
\frac{2\pi{\displaystyle{\int}_{2D}} |Df|^2}{\log{\left(e+ \dfrac{\mathrm{diam}(D)}{|z-w|}\right)}}
\end{equation}
for every disk $D$ such that $z,w\in D\subset 2D\subset \Omega$.  
{\footnote{Here, $2D$ is the disk concentric with $D$ having twice the radius.}}
\end{proposition}

Note that, for disks $D$ whose diameters are comparable to the distance $|z-w|$, one can safely replace the denominator on the 
right-hand side of~\eqref{GVineq} by a constant. This will suffice for our purposes.

\subsubsection*{Ingredients} 
Let $Q=[0,1]^2\subset \mathbb{R}^2$ be the unit square, 
and consider the concentric square 
$R=[\frac{1}{3},\frac{2}{3}]^2\subset Q$. 
The ingredients in our construction are a sequence  of 
homeomorphisms 
$g_n\colon Q\to Q$, $n\geq 1$, 
with the following properties:
\begin{enumerate}
\item[(i)] 
The support of each $g_n$ is contained in $R$;
\item[(ii)] 
For all $n\geq 1$, $g_n$ and $g_n^{-1}$ are regular maps of finite distortion;
\item[(iii)] 
There exist $C>0$ and $p>0$ such that, for each $n\geq 1$,  
\begin{equation*}
|Dg_n(z)|\leq Cn^{p} \ \ \ \ \textrm{and}\ \ \ \ |Dg_n^{-1}(z)|\leq Cn^{p}
\end{equation*} 
for Lebesgue almost every $z\in R$;
\item[(iv)] 
We have $h_{\mathrm{top}}(g_n)\to \infty$ as $n\to \infty$. 
\end{enumerate}
For instance, one could take $g_n$ to be a $C^1$-smooth horseshoe map with $n$ branches supported in the square $R$, so that 
$h_{\mathrm{top}}(g_n)=\log{n}$. Such a horseshoe can be built so as to satisfy (iii) at {\it every\/} point. 
 
\subsubsection*{Construction} 
Let us write $(0,1]=\bigcup_{n=1}^{\infty} I_n$, where 
\begin{equation*}
I_n
\;=\;
\left[\frac{1}{2^n}, \frac{1}{2^{n-1}}\right]\ .
\end{equation*}
Let $Q_n=I_n\times I_n\subset Q$, and let 
$f_n=A_n\circ g_n\circ A_n^{-1}\colon Q_n\to Q_n$, 
where 
$A_n\colon Q\to Q_n$ is the affine map
\begin{equation*}
A_n(x,y)
\;=\;
\left(\frac{x+1}{2^n}, \frac{y+1}{2^n}\right) \ .
\end{equation*}
Note that $|Df_n|=|Dg_n\circ A_n^{-1}|$, so that $|Df_n|\leq Cn^p$, by (iii) above. 
Note also that the support of $f_n$ is contained in the rescaled square $R_n=A_n(R)$ so, 
in particular, $f_n|_{\partial Q_n}\equiv \mathrm{id}$. 
Next, define $f\colon\mathbb{R}^2\to \mathbb{R}^2$ as follows:
\begin{equation*}
f(z)
\;=\;
\left\{\begin{array}{ll} 
{f_n(z)} & \ \ \textrm{if}\ \ z\in Q_n\ , \\
{}       & {} \\
{z}      & \ \ \textrm{if}\ \ z\in\mathbb{R}^2\setminus \bigcup_{n=1}^{\infty} Q_n\ .
\end{array}\right.
\end{equation*}
Then it is clear that $f$ is an orientation-preserving homeomorphism of the plane (see Figure \ref{figsquare1}). 
\begin{figure}[t]
\begin{center}
\psfrag{Q}[][]{ $Q$} 
\psfrag{q1}[][][1]{$Q_1$}
\psfrag{q2}[][][1]{$Q_2$} 
\psfrag{q3}[][][1]{$Q_3$}
\psfrag{f1}[][][1]{$f_1$}
\psfrag{f2}[][][1]{$f_2$}
\psfrag{f3}[][][1]{$f_3$}
%\psfrag{a}[][]{$f^{q_{n}-q_{n-1}}(c_{0})$} 
%\psfrag{b}[][]{$f^{q_{n}}(c_{0})$} 
%\psfrag{c}[][]{$c_{0}$} 
%\psfrag{d}[][]{$f^{-q_{n-1}}(c_{0})$} 
%\psfrag{e}[][]{$f^{2q_{n-1}}(c_{0})$} 
\includegraphics[width=3.5in]{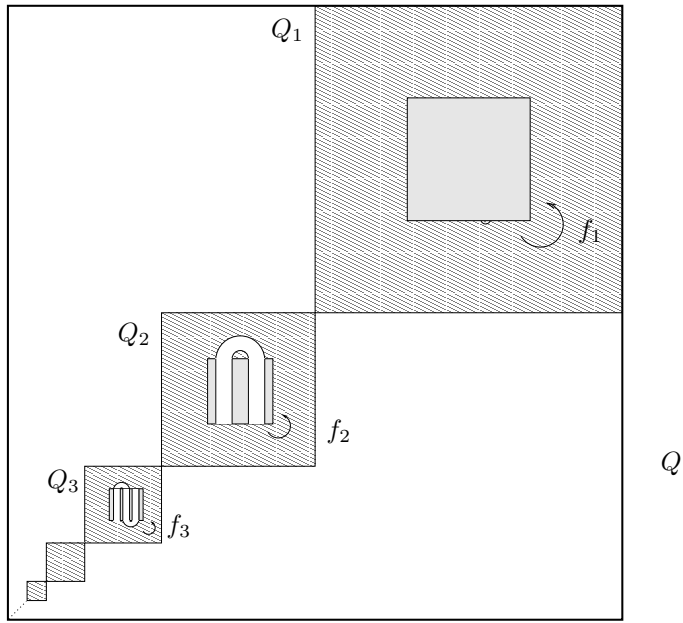}
\end{center}
\caption[slope]{\label{figsquare1} Sewing the maps $f_n$ together with the identity yields $f$.}
\end{figure}
\begin{theorem}
The homeomorphism $f$ and its inverse $f^{-1}$ are both regular maps with finite distortion, and have infinite topological entropy. 
Moreover, there exist constants $C>0$ and $\delta>0$
such that, for all $z,w\in \mathbb{R}^2$ with $|z-w|<\delta$ we have
\begin{equation}\label{gauge}
|f(z)-f(w)|
\;\leq\; 
C|z-w|\left(\log{\frac{1}{|z-w|}}\right)^{p}\ ,
\end{equation}
and similarly for $f^{-1}$. In particular, $f$ and $f^{-1}$ are $\alpha$-H\"older continuous for every $0<\alpha<1$. 
\end{theorem}
\begin{proof}
We have, rather trivially, 
\begin{equation*}
h_{\mathrm{top}}(f)
=\sup{ h_{\mathrm{top}}(f_n)}
=\sup{ h_{\mathrm{top}}(g_n)}
=\infty\ .
\end{equation*}
In other words, $f$ has infinite topological entropy, and hence so does $f^{-1}$. 
Thus, the real issue is to verify the regularity of $f$ (and $f^{-1}$). Throughout the proof, we shall denote by 
$C_1,C_2,\ldots$ positive absolute constants (it is possible to keep track of how they depend 
on the above data of the construction, but this will not be relevant for our purposes). 

First, if $D$ is any open disk such that $Q\subset D$, then by (iii) above we have
\begin{align*}
\int_{D}|Df|^2 
\;&=\; \int_{D\setminus \bigcup_{n=1}^{\infty} Q_n} |Df|^2 + \int_{\bigcup_{n=1}^{\infty}Q_n} |Df|^2\\
&\leq\; {\mathrm{Area}(D)} + C\sum_{n=1}^{\infty} n^{2p}{\mathrm{Area}(Q_n)} \\ 
&=\; {\mathrm{Area}(D)} + C\sum_{n=1}^{\infty} n^{2p}2^{-2n} <\infty\ .\\
\end{align*}
Since $f$ equals the identity outside $D$, this shows that $f\in W^{1,2}_{\mathrm{loc}}(\mathbb{R}^2)$. 

Next, we verify the modulus of continuity statement. We do it only for $f$, the verification for $f^{-1}$ being entirely analogous. 
Suppose $z,w\in \mathbb{R}^2$ are any two points such that 
\begin{equation*}
|z-w|
\;<\; 
\delta
\;=\;
\min\left\{\frac{1}{16}, e^{-p}\right\} 
\end{equation*}
Then there are three cases to consider:

\noindent{\it First case:\/} 
We have 
$z,w\in \mathbb{R}^2\setminus \bigcup_{n=1}^{\infty}R_n$. 
In this case, $f(z)=z$ and $f(w)=w$, 
and there is nothing to prove.

\noindent{\it Second case:\/} 
We have 
$z,w\in \bigcup_{n=1}^{\infty}R_n$. 
In this case, we may assume without loss of generality that 
$z\in R_m$ and $w\in R_n$ with $m\leq n$. There are two sub-cases to consider:
\begin{enumerate}
\item[(a)] $m=n$: 
in this sub-case, we let $D$ be the closed disk having the segment joining $z$ to $w$ as  
diameter. Note that, by construction, $2D\subset Q_m$. Applying the inequality \eqref{GVineq} of 
Proposition \ref{regmap}, we get
\begin{align}\label{estimate1}
|f(z)-f(w)|\;&\leq\; C_1 \left(\int_{2D} |Df|^2\,\right)^{\frac{1}{2}} \notag \\
&\leq\; 
C_2 \left(m^{2p}\mathrm{Area}(2D)\right)^{\frac{1}{2}}\;=\; 2C_2\sqrt{\pi} m^p|z-w|\ .
\end{align}
However, since $|z-w|<2^{-m}$, we have 
\begin{equation}\label{estimate2}
m\;<\;\frac{1}{\log{2}}\log{\frac{1}{|z-w|}} \ .
\end{equation}
Combining \eqref{estimate1} with \eqref{estimate2} yields \eqref{GVineq} in this sub-case. 
\item[(b)] $m<n$: in this sub-case, the distance between $z$ and $w$ is comparable to the diameter 
of $R_m$; in fact,
\begin{equation}\label{estimate3}
\frac{\sqrt{2}}{2}\frac{1}{2^m}
\;\leq\; 
|z-w|
\;\leq\; 
\sqrt{2}\frac{5}{3}\frac{1}{2^m}\ .
\end{equation}
Note that, since we are assuming that $|z-w|<\frac{1}{16}$, it follows from the first inequality in \eqref{estimate3} 
that $m>2$. Let $D$ be the disk with center at the origin and radius $\sqrt{2}\times 2^{-m+1}$, which contains 
both $z$ and $w$. Then its double $2D$ contains precisely the squares $Q_k$ with $k\geq m-1$. 
Thus, we have
\begin{align*}
\int_{2D}|Df|^2 
\;&=\; 
\int_{2D\setminus \bigcup_{k=m-1}^{\infty} Q_k} |Df|^2 + \int_{\bigcup_{k=m-1}^{\infty}Q_k} |Df|^2\\
&\leq\; 
{\mathrm{Area}(2D)} + C\sum_{k=m-1}^{\infty} k^{2p}{\mathrm{Area}(Q_k)} \\ 
&=\; 
{\mathrm{Area}(2D)} + C\sum_{k=m-1}^{\infty} k^{2p}2^{-2k} .\\
\end{align*}
Applying the inequality \eqref{GVineq} of Proposition \ref{regmap}, we get
\begin{equation}\label{estimate4}
|f(z)-f(w)|
\;\leq\; 
C_3\left({\mathrm{Area}(2D)} + C\sum_{k=m-1}^{\infty} k^{2p}2^{-2k}\right)^{\frac{1}{2}}
\end{equation}
Now, on the one hand we have 
\begin{equation}\label{estimate5}
\mathrm{Area}(2D)
=
\pi \left(\sqrt{2} \cdot 2^{-m+2}\right)
<
64\pi|z-w|^2\ ,
\end{equation}
where we have used \eqref{estimate3}. On the other hand, we can estimate the series in \eqref{estimate4} as follows. 
We have (by the integral test)
\begin{equation*}
 \sum_{k=m-1}^{\infty} k^{2p}2^{-2k}\;<\; \int_{m-2}^{\infty} x^{2p} e^{-(2\log{2})x}\,dx\ .
\end{equation*}
Here we use the following general fact{\footnote{See for instance \cite[pp.~103--104]{BM}.}}: for all $a>0$, all
$\lambda>1$ and all $\nu\in \mathbb{N}$, 
\begin{equation*}
\int_{a}^{\infty} x^\nu e^{-\lambda x}\,dx
\;=\; 
e^{-\lambda a}\sum_{j=0}^{\nu} 
\frac{\nu!}{j!}\frac{a^j}{\lambda^{\nu-j+1}} \;<\; (\nu+1)! a^{\nu}e^{-\lambda a}\ .
\end{equation*}
Applying this fact with $a=m-2$, $\lambda=2\log{2}$ and $\nu=2p$, we deduce using \eqref{estimate3} that
\begin{equation}\label{estimate6}
\sum_{k=m-1}^{\infty} k^{2p}2^{-2k}
\;<\; 
C_4m^{2p}2^{-2m}\;<\; C_5|z-w|^2\left(\log{\frac{1}{|z-w|}} \right)^{2p}\ .
\end{equation}
Combining \eqref{estimate4} with \eqref{estimate5} and \eqref{estimate6} yields~\eqref{GVineq} in this sub-case as well. 
\end{enumerate}

\noindent{\it Third case:\/} 
We have $z\in R_m$ for some $m$, but 
$w\in \mathbb{R}^2\setminus \bigcup_{n=1}^{\infty}R_n$. 
In this case, join $z$ to $w$ by a straight 
line segment and let $w'$ be the unique point at the boundary of $R_m$ belonging to that line segment. 
We have, of course, 
$|z-w|=|z-w'|+ |w'-w|$. 
Moreover, 
$f(w)=w$ and $f(w')=w'$, 
and since $z$ and $w'$ fall in the second case above, we have 
\begin{align}\label{estimate7}
|f(z)-f(w)|
\;&\leq\; 
|f(w)-f(w')|+|f(w')-f(z)| \notag\\ 
&\leq\; 
|w-w'|+C_6|z-w'|^{p}\log{\frac{1}{|z-w'|}}\ .
\end{align}
Since 
$\max\left\{|w-w'|,|w'-z|\right\}\leq |z-w|<\delta$ 
and the function $t\mapsto t(\log{(1/t)})^{p}$ is increasing for $0<t< e^{-p}$, 
we see that~\eqref{estimate7} implies~\eqref{GVineq}.
This finishes the proof of our theorem.
\end{proof}
For each positive integer $n$, by performing the same construction as above but just on the union of the squares 
$Q_n, Q_{n+1},\ldots$ 
we also get the following corollary.
\begin{corollary}
There exists a sequence of homeomorphisms $f_n$ of the unit square $[0,1]^2$
such that, for each $n$,
\begin{itemize}
\item
$f_n$ has infinite topological entropy;
\item
$f_n$ and $f_n^{-1}$ are regular maps with finite distortion;
\item
for each $\alpha\in (0,1)$, $f_n$ is bi-$\alpha$-H\"older continuous;
\end{itemize}
and such that
$\lim_{n\to\infty} f_n= \mathrm{id}$ 
where convergence is taken
\begin{itemize}
\item
in the $C^\alpha$-topology for any $\alpha\in[0,1)$,
\item
in the $W^{1,p}$-topology for any $p\in [1,\infty)$.
\end{itemize}
\end{corollary}
\begin{remark}
A similar construction can be performed in the area-preserving category.
Namely, if you take a monotone decreasing sequence $1=r_1>r_2>\ldots$ converging to zero, and for each $k$, 
let $A_k$ denote the annulus in the plane given in polar coordinates by
$\left\{(r,\theta) : r_{k}\geq r\geq r_{k+1}\right\}$.
Let $A_k^{\mathrm{in}}$ and $A_{k}^{\mathrm{out}}$ denote the inner and outer boundary components.
Take a (weak) monotone twist map $f_k$ of the annulus $A_k$
such that 
$h_{\mathrm{top}}(f_k)=\log k$ 
and $f_k$ is a rotation in a neighbourhood of $\partial A_k$, such that 
$f_k|_{A_{k}^{\mathrm{in}}}=f_{k+1}|_{A_{k+1}^{\mathrm{out}}}$.
Define 
\begin{equation*}
f(x)=\left\{
\begin{array}{ll}
f_k(x) & x \ \in A_k\\
0 &  x=0
\end{array}\right. \ .
\end{equation*}
Then $f$ is a homeomorphism.
For appropriately chosen $r_k$, $f$ is a bi-H\"older and bi-Sobolev homeomorphism.  
\end{remark}
%
%
%
%
%
%
%
%
%
%
%
%%%%%%%%%%%%%%%%%%%%%%%%%%%%%%%%%%%%%%%%
%%%%%%%%%%%%%%%%%%%%%%%%%%%%%%%%%%%%%%%%
\section{Basic Moves}\label{sect:basic_moves}
We describe the basic moves from which all 
perturbations in Parts I and II are constructed.
The geometric properties of these 
perturbations also allow us to estimate 
how far they are from the identity with 
respect to the Lipschitz and H\"older 
distances.

First we introduce the following terminology, 
to be used in this and the following Appendix.
We call the midpoint of the axis of $C$ the {\it centre point\/} 
of $C$ and we call the centres of the boundary balls $C^-$ and 
$C^+$ of $C$ the {\it endpoints\/} of $C$.
We say that a topological solid cylinder $C\subset \mathbb{R}^d$ 
has {\it separated boundary balls\/} if the following property 
is satisfied: 
Denote the boundary balls of $C$ by $C^-$ and $C^+$ and denote 
the corresponding endpoints by $c^-$ and $c^+$ respectively.
Let $\Pi\subset\mathbb{R}^d$ denote the perpendicular bisector of the line segment $[c^-,c^+]$.
Then the boundary balls $C^+$ and $C^-$ 
lie in different connected components of $\mathbb{R}^d\setminus \Pi$.

\begin{lemma}[Lipschitz Gluing Principle]
Let 
$(\Omega,d_{\Omega})$ 
and 
$(\Omega',d_{\Omega'})$ 
be metric spaces, where $\Omega$ is 
geodesically convex.
Let $\Omega_1, \Omega_2,\ldots$ be 
geodesically convex, pairwise disjoint 
subdomains of $\Omega$ and let $\Omega_0=\Omega\setminus\bigcup_k\Omega_k$.
Let 
$f\in C^0(\Omega,\Omega')$ 
be a function which is Lipschitz on 
$\Omega_0,\Omega_1,\Omega_2,\ldots$.
Then $f$ is Lipschitz and
\begin{equation*}
[f]_{\mathrm{Lip}}
\leq 
\max_{k=0,1,2\ldots} [f]_{\mathrm{Lip},\overline{\Omega}_k} \ .
\end{equation*}
\end{lemma}
%
\begin{comment}
%%%%%%%%%%%%%%%%%%%%%%%%%%%%%%%%%%%%%%%%%%%%%%%%
\begin{proof}
To simplify notation, let 
$\Omega_0=\Omega\setminus\bigcup_{k=1}^\infty \Omega_k$.
Choose $x,y\in \Omega$.
By taking a geodesic from $x$ to $y$ 
and intersecting it with 
$\partial\Omega_0,\partial\Omega_1,\ldots$ 
we may assume that there are consecutive points 
$x=z_0,z_1,\ldots$ 
along the geodesic so that $z_k$ and $z_{k+1}$ 
lie in some $\mathrm{cl}\Omega_{r_k}$.
Then
\begin{align}
|f(x)-f(y)|
&\leq \sum_{k=0}^\infty |f(z_k)-f(z_{k+1})|\\
&\leq \sum_{k=0}^\infty [f]_{\mathrm{Lip},\mathrm{cl}(\Omega_{r_k})}|z_{k}-z_{k+1}|\\
&\leq \left(\max_{r=0,1\ldots} [f]_{\mathrm{Lip},\mathrm{cl}(\Omega_r)}\right) \sum_{k=0}^\infty |z_{k}-z_{k+1}|\\
&= \left(\max_{r=0,1,\ldots} [f]_{\mathrm{Lip},\mathrm{cl}(\Omega_r)}\right) |x-y|
\end{align}
where the last equality follows as the $z_k$ lie along a geodesic.
Hence the lemma is shown.
\end{proof}
%%%%%%%%%%%%%%%%%%%%%%%%%%%%%%%%%%%%%%%%%%%%%%%%
\end{comment}
%
\begin{lemma}[Lipschitz Bounds for Bump Functions]\label{lem:Lip_bounds_for_bump}
Given $p,q\in\mathbb{R}^d$, and positive real numbers $r_1<r_2$,
there exists $b\in C^1(\mathbb{R}^d,\mathbb{R})$ satisfying 
$b|_{E(p,q;r_1)}\equiv 1$,
$b|_{\mathbb{R}^d\setminus E(p,q;r_2)}\equiv 0$
and 
$[b]_{\mathrm{Lip}}\leq K/|r_2-r_1|$, 
where $K$ is independent of $p,q,r_1$ or $r_2$.
\end{lemma}
\begin{lemma}\label{lem:E-perturbation}
Take any $p,q\in\mathbb{R}^d$ and $0<r_1<r_2$.
Then there exists $\phi\in\mathrm{Diff}_+^1(\mathbb{R}^d)$, such that
$\phi(p)=q$,
$\mathrm{supp}(\phi)\subset \overline{E(p,q;r_2)}$ and 
\begin{equation*}
\max\left\{
\bigl[\phi-\mathrm{id}\bigr]_{\mathrm{Lip}}, 
\bigl[\phi^{-1}-\mathrm{id}\bigr]_{\mathrm{Lip}}
\right\}
<K \ ,
\end{equation*}
where $K$ depends upon $|p-q|/|r_2-r_1|$ only. 
\end{lemma}
\begin{proof}
Let $b\in C^1(\mathbb{R}^d,\mathbb{R})$ denote 
the bump function given in Lemma~\ref{lem:Lip_bounds_for_bump} for the 
domains $E(p,q;r_1)\subset E(p,q;r_2)$.
Observe that $[b]_{\mathrm{Lip}}\leq K/|r_2-r_1|$.
Consider the vector field
\begin{equation*}
X(x)=b(x)(q-p)
\end{equation*}
There exist positive real numbers $L$ and $M$ such that, for all $x,y\in\mathbb{R}^d$,
\begin{equation}\label{ineq:vectorfield_bounded+lipschitz}
|X(x)|\leq L, \qquad |X(x)-X(y)|\leq M|x-y| \ .
\end{equation}
(In fact, $L=|q-p|$ and $M=[b]_{\mathrm{Lip}}|q-p|\leq K|q-p|/|r_2-r_1|$.)
%Therefore, by the existence and uniqueness theorem 
%for ODE's, the corresponding flow, which we denote 
%by $\Phi_t(x)$, exists.(This initially only gives 
%local existence and uniqueness, but since the vector 
%field has compact support the local solutions may 
%be patched together.)
Let $\Phi_t$ denote the corresponding flow.
For 
$x$ in $E(p,q;r_1)$ and 
$|t|\leq L\cdot \mathrm{dist}(x,\partial E(p,q;r_1))$ 
we have that 
$\Phi_t(x)=x+t(q-p)$.
Similarly, for 
$x$ in $\mathbb{R}^d\setminus E(p,q;r_2)$ 
we have that 
$\Phi_t(x)=x$.
Let $\phi$ denote the time-one map $\Phi_1$.
Then $\phi$ is a $C^1$-smooth diffeomorphism 
with support in $\overline{E(p,q;r_2)}$ 
such that $\phi(p)=q$.
It remains to bound the Lipschitz constant of $\phi$.
Fix distinct points $x$ and $y$ in $\mathbb{R}^d$.
%GRONWALL LEMMA
%[Gronwall--Integral version] 
%Given C^0, non-negative functions f,g on [a,b) and 
% non-decreasing function A, whose negative part is inetrgable on each compact subinterval of [a,b):
%f(t)\leq A(t)+\int_a^t f(s)g(s)\,ds
%implies
%f(t)\leq A(t)\exp\left(\int_a^t g(s)\,ds\right)
%%%%%%%%%METHOD1%%%%%%%%%
%%%%%%%%%%%%%%%%%%%%%%%%%
\begin{comment}
Define 
$f(t)=\left|\Phi_t(x)-\Phi_t(y)\right|$ for $t\in\mathbb{R}$.
Then inequality~\eqref{ineq:vectorfield_bounded+lipschitz} implies
\begin{align*}
f(t)
=\left|\Phi_t(x)-\Phi_t(y)\right|
&=\left|x-y+\int_0^t X(\Phi_s(x))-X(\Phi_s(y))\,ds\right|\\
&\leq |x-y|+\int_0^t M\left|\Phi_s(x)-\Phi_s(y)\right|\,ds\\
&= |x-y|+\int_0^t Mf(s)\,ds
\end{align*}
Applying Gr\"onwall's inequality and setting $t=1$,
we get the following inequality
\begin{equation*}
\left|\Phi_t(x)-\Phi_t(y)\right|
\leq |x-y|\exp\left(\int_0^t M \,ds\right)
=e^{Mt}|x-y|
\end{equation*}
Hence,
$[\phi]_{\mathrm{Lip}}\leq K=e^M$. 
By symmetry, the same argument shows that 
$\left[\phi^{-1}\right]_{\mathrm{Lip}}\leq K$.
\end{comment}
%%%%%%%%%%%%%%%%%%%%%%%%%
%%%%%%%%%METHOD2%%%%%%%%%
%%%%%%%%%%%%%%%%%%%%%%%%%
Define 
$f(t)=\left|(\Phi_t(x)-x)-(\Phi_t(y)-y)\right|$.
Then inequality~\eqref{ineq:vectorfield_bounded+lipschitz} implies
%Since $\Phi_t(x)=x+\int_0^t X(\Phi_s(x))\,ds$
\begin{align*}
f(t)
&=\left|(\Phi_t(x)-x)-(\Phi_t(y)-y)\right|\\
&\leq \int_0^t\left|X(\Phi_s(x))-X(\Phi_s(y))\right|\,ds\\
&\leq \int_0^t M\left|\Phi_s(x)-\Phi_s(y)\right|\,ds\\
&\leq M\int_0^t |x-y|+\left|(\Phi_s(x)-x)-(\Phi_s(y)-y)\right|\,ds\\
&=Mt|x-y|+\int_0^t Mf(s)\,ds\ .
\end{align*}
Applying Gr\"onwall's inequality and setting $t=1$, 
we get the following inequality 
%\begin{equation}
%f(t)=
%=|(\Phi_t(x)-x)-(\Phi_t(y)-y)|
%\leq Mt|x-y|e^{Mt}
%\end{equation}
%Hence, 
\begin{equation*}
\left[\phi-\mathrm{id}\right]_{\mathrm{Lip}}\leq K=M e^M \ .
\end{equation*}
By symmetry, the same argument show that 
$\left[\phi^{-1}-\mathrm{id}\right]_{\mathrm{Lip}}\leq K$.
Thus the lemma is shown.
\end{proof}
Applying the construction in Lemma~\ref{lem:E-perturbation} to a solid cylinder instead of a single point gives the following.
\begin{lemma}[Basic Move 1 -- Translation to the Origin]\label{lem:basic_move_1}
Let $C$ be a rigid solid cylinder in $\mathbb{R}^d$.
Let $c$ denote the centre of $C$ and let $r$ be the least positive real number such that $C$ is contained in the elongated neighbourhood $E(0,c;r)$.
Then there exists a diffeomorphism of $\mathbb{R}^d$ with support in $E(0,c;r+|c|)$ such that $\phi|_{E(0,c;r)}$ is translation from $c$ to $0$.
In particular, $\phi(C)$ is a rigid cylinder with centre $0$.
Moreover, 
$\bigl[\phi\bigr]_{\mathrm{Lip}}, 
\bigl[\phi^{-1}\bigr]_{\mathrm{Lip}}<K$, 
where $K$ is a uniform constant.
\end{lemma}
The construction in Lemma~\ref{lem:E-perturbation} can also be performed with any linear flow, not just a translational flow.
The case of rotational flows and hyperbolic flows gives the following two Lemmas.
\begin{lemma}[Basic Move 2 -- Rotation about the Origin]\label{lem:basic_move_2}
Let $C$ and $C'$ be isometric rigid solid cylinders centred at the origin in $\mathbb{R}^d$, both contained in $B(0,r)$.
Then there exists a diffeomorphism $\phi$ of $\mathbb{R}^d$ with support in $B(0,2r)$ which maps $C$ isometrically onto $C'$.
Moreover, $[\phi]_{\mathrm{Lip}}, [\phi^{-1}]_{\mathrm{Lip}}<K$, 
where $K$ is a uniform constant (depending only upon the dimension $d$).
\end{lemma}
We note that the uniformity on the constant comes from the compactness of the group of rotations $SO(d)$.
%
\begin{comment}
%%%%%%%%%%%%%%
\begin{proof}
Let $b$ be a bump function satisfying $b|B(0,r)\equiv 1$, $b|(\mathbb{R}^d\setminus B(0,2r))$.
Let $Y$ be a rotational vector field about the origin, {\it i.e.\/}, it has the form $Y(x)=|x|Z(x/|x|)$ 
where $Z$ is a linear vector field preserving $D(0,1)$, which we think of as an element of $SO(d)$.
Let $X(x)=b(x)Y(x)$. 
Then, since |Z(x)|=1$ if follows that $|X(x)|=|xb(x)|\cdot |Z(x)|\leq \max |x|\max |b(x)|=2r$ 
Similarly
\begin{align*}
|X(x)-X(y)|
&\leq [xb(x)]_{\mathrm{Lip}}[Z]_{\mathrm{Lip}}|x-y|\\
&\leq 2r[b]_{\mathrm{Lip}}[Z]_{\mathrm{Lip}}\leq M|x-y|
\end{align*}
where we have used that $[b]_{\mathrm{Lip}}\leq K/r$ and $SO(d)$ is compact.
The result now follows from same argument as the translational case.
\end{proof}
%%%%%%%%%%%%%%
\end{comment}
%
%
\begin{lemma}[Basic Move 3 -- Cylinder Affinity]\label{lem:basic_move_3}
Let $C'$ and $C''$ be rigid solid cylinders in $\mathbb{R}^d$, concentric and with the same axis.
Let $C$ denote the smallest rigid solid cylinder, concentric and with the same axis as $C'$ and $C''$, containing $C'$ and $C''$. 
Let $E$ denote the elongated neighbourhood with the same endpoints as $C$ and radius $\mathrm{rad}(C)$.
Let $E'$ denote the elongated neighbourhood with the same endpoints as $C$ and radius $\mathrm{rad}(C)+\mathrm{diam}(C)$. 
Then there exists a diffeomorphism, supported in $E'$, mapping $C'$ across $C''$.
Moreover $[\phi]_{\mathrm{Lip}},[\phi^{-1}]_{\mathrm{Lip}}<K$ where $K$ depends only upon $\max (\|A\|,\|A^{-1}\|)$, 
where $A$ is the hyperbolic linear map sending $C'$ to $C''$.
\end{lemma}
%
\begin{comment}
%%%%%%%%%%%%%%%%
\begin{proof}
Let $b=b_3\circ \rho$ denote the obvious bump function on $E'$.
Let $X(x)=b(x)A(x)$.
Then 
\begin{equation*}
|X(x)|\leq |b(x)|\cdot\|A\|\cdot |x|\leq \|A\| \mathrm{diam}(C)
\end{equation*}
and, since $[b]_{\mathrm{Lip}}\leq L/\mathrm{diam}(C)$ for some positive real $L$, we find
\begin{align*}
|X(x)-X(y)|
&\leq |b(x)-b(y)|\cdot |Ax|+|b(y)|\cdot|Ax-Ay|\\
&\leq \left([b]_{\mathrm{Lip}}\mathrm{diam}(C)+1\right)\|A\|\cdot|x-y|\\
&\leq (L+1)\|A\|\cdot|x-y|
\end{align*}
The same argument as in the translational case now gives the result.
\end{proof}
%%%%%%%%%%%%%%%%
\end{comment}
%
\begin{corollary}\label{cor:cylinder_isometry}
There exists $K>0$ such that the following property is satisfied:
Let $C$ and $C'$ be isometric rigid solid cylinders, both contained in $B(p,r)\subset \mathbb{R}^d$.
Then there exists a diffeomorphism $\phi$ of $\mathbb{R}^d$ such that
\begin{enumerate}
\item 
the support of $\phi$ is contained in $B(p,10r)$
\item
$\phi$ maps $C$ isometrically onto $C'$
\item
$[\phi]_{\mathrm{Lip}}\leq K$
\end{enumerate}
\end{corollary}
\begin{proof}
Denote the centres of $C$ and $C'$ by $c$ and $c'$ respectively.
An isometry from $C$ to $C'$ can be decomposed as 
$\tau'\circ \alpha\circ \tau^{-1}$, where $\tau$ and $\tau'$ are translations 
$c$ and $c'$, respectively, 
and $\alpha$ is a linear isometry.
Consequently we may construct diffeomorphisms $\phi_\tau$, $\phi_\alpha$ and $\phi_{\tau'}$, associated with $\tau$, $\alpha$ and $\tau'$ respectively,
by performing the Basic Moves 1 and 2. 
The resulting diffeomorphism $\phi=\phi_{\tau'}^{-1}\circ\phi_{\alpha}\circ\phi_{\tau}$ satisfies the stated properties.
\end{proof}
Combining Corollary~\ref{cor:cylinder_isometry} and Basic Move 1 (Lemma~\ref{lem:basic_move_1}) also give the following.
\begin{corollary}\label{cor:cylinder_isometry_2}
Let $p,q\in\mathbb{R}^d$ and $r>0$.
Let $C_p$ and $C_q$ be isometric solid cylinders contained in $B(p,r)$ and $B(q,r)$ respectively.
Then there exists a diffeomorphism $\phi$ of $\mathbb{R}^d$
and a positive real number $K$, depending upon $|p-q|/r$ only,
such that
\begin{enumerate}
\item
the support of $\phi$ is contained in $E(p,q;10r)$
\item
$\phi$ maps $C_p$ isometrically onto $C_q$
\item
$[\phi]_{\mathrm{Lip}}\leq K$
\end{enumerate}
\end{corollary}

%
%
%%%%%%%%%%%%%%%%%%%%%%%%%%%%%%%%%%%%%%%%%%%%%%%%%%%%%%%%%%%%%%%%%%%%%%
\section{Mapping Solid Cylinders Across Solid Cylinders.}\label{sect:cylinder_perturbations}
In this section we construct certain perturbations that map solid cylinders.
For notation and terminology concerning solid cylinders we refer to  \S~\ref{sect:Holder-infinite_entropy}.
\begin{lemma}\label{lem:bi-Lipschitz_cylinder_separation}
Let $\Omega_0$ and $\Omega_1$ be open domains in $\mathbb{R}^d$.
Given $f\in\mathcal{H}^{1}(\Omega_0,\Omega_1)$ 
with bi-Lipschitz constant $\kappa$
%{\it i.e.\/},
%\begin{equation}\label{ineq:bi-Lipschitz}
%\kappa^{-1} |x-y|\leq |f(x)-f(y)|\leq \kappa|x-y|, \qquad 
%\forall x,y\in\Omega_0
%\end{equation}
%
there exist positive real numbers $K_0$ and $K_1$, 
depending upon $\kappa$ only,  
%actually $0<K_0<1$ and $K_1>1$ 
with the following property:
Let $C_0$ be a rigid solid cylinder in $\Omega_0$ satisfying 
\begin{equation}\label{eq:admissible_cylinder}
\mathrm{rad}(C_0)<K_0\cdot \mathrm{len}(C_0) \ .
\end{equation}
Then there exists a rigid solid cylinder $C_1$ in $\mathbb{R}^d$ such that
$f$ maps $C_0$ across $C_1$ and
\begin{equation*}
\mathrm{rad}(C_1)\leq K_1\cdot \mathrm{len}(C_1) \ .
\end{equation*}
\end{lemma}
\begin{proof}
Choose $K_0=1/3\kappa^2$.
Let $C_0$ be a rigid solid cylinder in $\Omega_0$ of 
axial length $\ell_0$ and 
coaxial radius $\varrho_0\leq K_0\ell_0$. 
Denote the endpoints of $C_0$ by $c_0^-$ and $c_0^+$.
We will show that it suffices to take $C_1$ 
whose axis is concentric with and parallel to the line segment 
$[f(c_0^-),f(c_0^+)]$.
Below, the length and radius will be determined. 

First we consider the length.
Denote the boundary ball of $C_0$ containing $c_0^\pm$ by $C^\pm$.
Take any point $x$ in the boundary disk $C^-$.
Then 
$\varrho_0\leq K_0\ell_0$ 
and 
$K_0=1/3\kappa^2$ 
implies that 
$\kappa|c_0^{-}-x|\leq \frac{1}{3}\kappa^{-1}|c_0^{-}-c_0^{+}|$.
Hence applying the $\kappa$-bi-Lipschitz property of $f$ twice gives
\begin{align*}
|f(c_0^{-})-f(x)|\leq \kappa |c_0^{-}-x|
%&\leq \kappa K_0 |c_0^{-}-c_0^{+}| \\
%&\leq \kappa^2 K_0 |f(c_0^{-})-f(c_0^{+})| \\
&\leq \tfrac{1}{3} |f(c_0^{-})-f(c_0^{+})| \ .
\end{align*}
Let $\ell=|f(c_0^{-})-f(c_0^{+})|$.
Then the above implies that
$f(C^-)$ is contained in the closure of $B\left(f(c_0^-),\frac{1}{3}\ell\right)$.
The same argument also shows that 
$f(C^+)$ is contained in the closure of $B\left(f(c_0^+),\frac{1}{3}\ell\right)$.
Consequently the topological solid cylinder $f(C_0)$ has separated boundary balls.

Next we consider the radius.
Choose a point $x$ in $C_0$ such that 
the distance $\varrho$ between $f(x)$ and the line segment $[f(c_0^-),f(c_0^+)]$ is maximal.
Take the orthogonal projection $y$ of $f(x)$ to the line between $f(c_0^-)$ and $f(c_0^+)$.
Let $\delta_1=|f(c_0^-)-f(x)|$.
% and $\delta_2=|f(c_0^-)-y|$.
Then 
\begin{equation*}
\delta_1\leq \kappa |c_0^{-}-x|
\leq \kappa (\ell_0^2+\varrho_0^2)^{1/2}
\leq \kappa (1+K_0^2)^{1/2}\ell_0 \ .
\end{equation*}
Hence
\begin{align*}
\varrho\leq\delta_1 
\leq \kappa (1+K_0^2)^{1/2}\ell_0
\leq \kappa (1+K_0^2)^{1/2}\kappa \ell
=K\ell \ .
\end{align*}
Let $C_1$ denote the rigid solid cylinder in $\mathbb{R}^d$ with axis  
\begin{align*}
[c_1^-,c_1^+]
=
\left[f(c_0^{-})+\tfrac{1}{3}(f(c_0^{+})-f(c_0^{-})),f(c_0^{+})-\tfrac{1}{3}(f(c_0^{+})-f(c_0^{-}))\right]
\end{align*}
and coaxial radius $\varrho_1=2\varrho$.
Observe that $C_1$ has length $\ell_1=\tfrac{1}{3}\ell$.
Then $f$ maps $C_0$ across $C_1$.
Moreover $\varrho_1\leq 2K \ell= 6K \ell_1$.
Setting $K_1=6K$ finishes the proof.
\end{proof}
\begin{lemma}\label{lem:transport_cylinder_lip}
Let $\Omega_0$ and $\Omega_1$ be open domains in $\mathbb{R}^d$.
Given $f\in \mathcal{H}^{1}(\Omega_0,\Omega_1)$
there exist positive real numbers $K_0, K_1$ and $\gamma$, 
depending upon the the bi-Lipschitz constant $\kappa$ of $f$ only, 
% {\it i.e.\/}, $\kappa=\max([f]_{\mathrm{Lip}},[f^{-1}]_{\mathrm{Lip}})$.
such that the following property holds:
Take $p\in\Omega_0$.
Choose $r_1$ sufficiently small to ensure that
\begin{equation*}
B(p,r_1)\subset \Omega_0 \ , \qquad
B(f(p),r_1)\subset \Omega_1 \ .
\end{equation*}
Let $r_2=r_1/20(1+K_1)$ and $r_3=r_2/\kappa$.
Take rigid solid cylinders $C_0$ and $C_1$, contained in $B(p,r_3)$ and $B(f(p),r_3)$,
such that
\begin{enumerate}
\item[(i)]
$\mathrm{rad}(C_0)\leq K_0\mathrm{len}(C_0)$,
\item[(ii)]
$\frac{1}{2}\mathrm{rad}(C_0)\leq \mathrm{rad}(C_1)\leq 2\mathrm{rad}(C_0)$,
\item[(iii)]
$r_3\leq \mathrm{len}(C_0), \mathrm{len}(C_1)\leq 2r_3$,
\end{enumerate}
Then there exists $\phi\in \mathrm{Diff}^1(\Omega_1)$, supported in $B(f(p),r_1)$, such that
\begin{enumerate}
\item[(a)]
$\max\left\{
[\phi]_\mathrm{Lip},[\phi^{-1}]_{\mathrm{Lip}}
\right\}
\leq \gamma$,
\item[(b)]
$\phi\circ f$ maps $C_0$ across $C_1$.
\end{enumerate}
\end{lemma}
\begin{proof}
Let $K_0$ and $K_1$ denote the positive real numbers, 
depending upon $\kappa$ only, from 
Lemma~\ref{lem:bi-Lipschitz_cylinder_separation}.
For $i=0,1$, assume that $C_i=C(a_i,b_i;\varrho_i)$ for some $a_i, b_i$ and $\varrho_i$.
Denote by $C'$ the rigid solid cylinder in $\mathbb{R}^d$ from Lemma~\ref{lem:bi-Lipschitz_cylinder_separation}.
Then we may assume that $C'$ has the following properties
\begin{itemize}
\item[(1)]
$f$ maps $C_0$ across $C'$,
\item[(2)]
the axis of $C'$ is contained in the line segment $[f(c_0^{-}),f(c_0^{+})]$,
\item[(3)]
$C'$ is symmetric about the perpendicular bisector $\Pi$ of $f(c_0^{-})$ and $f(c_0^{+})$,
\item[(4)]
$\mathrm{len}(C')=\ell'=\frac{1}{3}|f(c_0^{-})-f(c_0^{+})|$ and $\mathrm{rad}(C')=\rho'=K_1\ell'$.
\end{itemize}
Moreover $C'$ is contained in $B(f(p),r_1/10)$.
\begin{comment}
%%%%%%%%%%%%%%%%%%%%%%%%%%%%%%%%%%%%%%%%%%%%%%%%%%%%%%%%%%%
%$C'$ is contained in $B(f(p),r_1/10)$. 
%%%%%%%%%%%%%%%%%%%%%%%%%%%%%%%%%%%%%%%%%%%%%%%%%%%%%%%%%%%
The argument is as follows. 
First, $f(c_0^{-})$ and $f(c_0^{+})$ lie in $B(f(p),r_2)$, 
As $B(f(p),r_2)$ is convex, so does the midpoint 
of $f(c_0^{-})$ and $f(c_0^{+})$.
But the centre of $C'$, call it $c'$, coincides 
with this midpoint.
Hence 
$|f(p)-c'|\leq r_2$. 
Secondly, observe that, since $\ell'\leq 2r_2$, 
the triangle inequality and property (4) imply
\begin{equation*}
|c'-x|
\leq \tfrac{1}{2}\ell'+\varrho'
= \tfrac{1}{2}l'+K_1\ell'
\leq r_2+2K_1r_2 
\qquad \forall x\in C'
\end{equation*}
Combining these with the triangle inequality once more gives
\begin{equation*}
|f(p)-x|
\leq |f(p)-c'|+|c'-x|
\leq r_2+r_2(1+2K_1)
= r_1/10 \qquad
\forall x\in C'
\end{equation*}
Therefore $C'$ is contained in $B(f(p),r_1/10)$, as claimed.
%%%%%%%%%%%%%%%%%%%%%%%%%%%%%%%%%%%%%%%%%%%%%%%%%%%%%%%%%%%
\end{comment}
%
Denote by $C''$ the rigid solid cylinder 
which is concentric with $C'$, symmetric 
about bisecting plane $\Pi$, and isometric with $C_1$.
Observe that $C''$ is also contained in $B(f(p),r_1/10)$.
\begin{comment}
%%%%%%%%%%%%%%%%%%%
To see this, observe that, 
as $C_1$ lies in $B(p,r_3)$, 
$C''$ lies in $B(c',r_3)$.
As $c'$ lies inside $B(p,r_2)$' $C''$ lies inside $B(p,r_2+r_3)$
Since $r_2+r_3<2r_2=r_1/10(1+K_1)\leq r_1/10$, the claim follows.
%%%%%%%%%%%%%%%%%%%
\end{comment} 
Denote the axial length of $C''$ and the 
coaxial radius of $C''$ respectively by 
$\rho''$ and $\ell''$.

By Basic Move 3 (Lemma~\ref{lem:basic_move_3}), 
there exists a diffeomorphism $\sigma$, 
supported in some elongated neighbourhood $E$ 
of $C'\cup C''$ contained in $B(f(p),r_1)$, 
which maps $C'$ onto $C''$ and 
a positive real number $\gamma_1$, independent of $\ell',\ell'',\varrho'$ and $\varrho''$, such that
\begin{equation}\label{ineq:sigma_case2}
[\sigma]_{\mathrm{Lip}}\leq 
c_1 \max\left(\ell''/\ell',\varrho''/\varrho'\right) \ , \qquad
[\sigma^{-1}]_{\mathrm{Lip}}
\leq 
\gamma_1 \max\left(\ell'/\ell'',\varrho'/\varrho''\right) \ .
\end{equation}
Observe that, as $f$ is bi-Lipschitz, 
$\ell_0$ and $\varrho_0$ are comparable, 
by a constant depending only upon $\kappa$, 
to $\ell'$ and $\varrho'$ respectively.
By hypotheses (i)--(iii), 
$\ell_0$ and $\varrho_0$ are comparable 
to $\ell_1=\ell''$ and $\varrho_1=\varrho''$.
Therefore there exists a positive real 
number $\gamma_2$, depending upon $\kappa$ only, such that
\begin{equation}\label{ineq:sigma}
\max\left\{
[\sigma]_{\mathrm{Lip}}, 
[\sigma^{-1}]_{\mathrm{Lip}}
\right\}
\leq \gamma_2 \ .
\end{equation}
Next, since $C''$ and $C_1$ are isometric, 
Corollary~\ref{cor:cylinder_isometry} implies that
there exists a diffeomorphism $\tau$, 
supported in $B(f(p),r_1)$, and a positive real number $\gamma_3$ such that
\begin{equation}\label{ineq:tau}
\max\left\{
[\tau]_{\mathrm{Lip}}, 
[\tau^{-1}]_{\mathrm{Lip}}
\right\}
\leq \gamma_3 \ ,
\end{equation} 
which maps $C''$ isometrically onto $C_1$.
It follows that the diffeomorphism 
$\phi=\tau\circ\sigma$ is supported in $B(f(p),r_1)$ 
and maps $f(C_0)$ across $C_1$. 
Therefore $\phi\circ f$ maps $C_0$ across $C_1$.
Moreover, inequality~\eqref{ineq:sigma} and~\eqref{ineq:tau} imply
\begin{equation*}
\max\left\{
[\phi]_{\mathrm{Lip}},
[\phi^{-1}]_{\mathrm{Lip}}
\right\}
\leq \gamma_2\gamma_3 \ ,
\end{equation*}
which completes the proof of the lemma.
\end{proof}
%
%
%
%
%
\begin{comment}
%%%%%%%%%%%%%%%%%%%%%%%%%%%%%%
\begin{proposition}\label{prop:horseshoes_exist}
There exists a positive real number $\gamma$ 
such that the following is satisfied:
For each positive integer $N$,
there exists an orientation-preserving 
$C^1$-diffeomorphism $\phi$ of $\mathbb{R}^d$, 
supported in $E(-e,e;3)$,
where $e=(0,0,\ldots,0,1)$, 
such that $\phi$ maps each solid sub-cylinder 
\begin{equation*}
B^{d-1}\times \left[\frac{2m-N}{N},\frac{2m+1-N}{N}\right],\qquad
m=0,1,\ldots,N-1 
\end{equation*}
across the solid cylinder
$B^{d-1}\times B^1$.
Moreover, 
$\max\left\{[\phi]_{\mathrm{Lip}}, [\phi^{-1}]_{\mathrm{Lip}}\right\}\leq \gamma N$.
\end{proposition}
%%%%%%%%%%%%%%%%%%%%%%%%%%%%%
\end{comment}
%
%
\begin{proposition}\label{prop:horseshoes_exist}
Let $C(a,b;\varrho)\subset \mathbb{R}^d$ be a rigid solid cylinder.
There exists a positive real number $\gamma$, depending upon $|a-b|$ and 
$\varrho$ only, such that the following property is satisfied:
For each positive integer $N$,
there exists an orientation-preserving $C^1$-diffeomorphism $\phi$ of 
$\mathbb{R}^d$, supported in the elongated neighbouhood $E(a,b;3\varrho)$,
and there exist solid sub-cylinders $C_1,C_2,\ldots,C_N$ of $C(a,b;\varrho)$ 
such that $\phi$ maps $C_j$ across $C(a,b;\varrho)$ for each $j=1,2\ldots,N$.
Moreover 
$\max\left\{[\phi]_{\mathrm{Lip}}, [\phi^{-1}]_{\mathrm{Lip}}\right\}\leq \gamma N$.
\end{proposition}
\begin{proof}
The case when $a=(0,\ldots,0,1)$, $b=(0,\ldots,0,-1)$ and $\varrho=1$ 
is the classical construction of a horseshoe with $N$ branches.
The general case follows by conjugating via a map $\psi$ 
from an elongated neighbourhood of the standard solid cylinder to $C(a,b;\varrho)$
and applying the H\"older Rescaling Principle (Proposition~\ref{prop:holderrescaling1}).
\end{proof}
Combining the above Proposition~\ref{prop:horseshoes_exist} 
with Lemma~\ref{lem:transport_cylinder_lip} gives the following.
\begin{corollary}\label{cor:transport_cylinder+horseshoe_lip}
Let $\Omega_0$ and $\Omega_1$ be open domains in $\mathbb{R}^d$.
Let $N$ be a positive integer.
Given $f\in \mathcal{H}^{1}(\Omega_0,\Omega_1)$
there exist positive real numbers $K_0, K_1$ and $\gamma$, 
depending upon the the bi-Lipschitz constant $\kappa$ of $f$ only, 
such that the following property holds:
Take $p\in\Omega_0$.
Choose $r_1$ sufficiently small to ensure that
\begin{equation*}
B(p,r_1)\subset \Omega_0 \ , \qquad
B(f(p),r_1)\subset \Omega_1 \ .
\end{equation*}
Let $r_2=r_1/20(1+K_1)$ and $r_3=r_2/\kappa$.
Take rigid solid cylinders $C_0$ and $C_1$, contained in $B(p,r_3)$ and $B(f(p),r_3)$,
such that
\begin{enumerate}
\item[(i)]
$\mathrm{rad}(C_0)\leq K_0\mathrm{len}(C_0)$,
\item[(ii)]
$\frac{1}{2}\mathrm{rad}(C_0)\leq \mathrm{rad}(C_1)\leq 2\mathrm{rad}(C_0)$,
\item[(iii)]
$r_3\leq \mathrm{len}(C_0), \mathrm{len}(C_1)\leq 2r_3$.
\end{enumerate}
Then there exists 
$\phi\in \mathrm{Diff}^1(\Omega_1)$, supported in $B(f(p),r_1)$, 
and there exist
pairwise disjoint subcylinders $C_{0,1},\ldots,C_{0,N}$, 
such that
\begin{enumerate}
\item[(a)]
$\max\left\{
[\phi]_\mathrm{Lip},[\phi^{-1}]_{\mathrm{Lip}}
\right\}\leq \gamma$,
\item[(b)]
$\phi\circ f$ maps the subcylinder $C_{0,k}$ across $C_1$ for each $k=1,\ldots,N$.
\end{enumerate}
\end{corollary}
\begin{remark}
Observe that we can use the same radii, $r_2$ and $r_3$ 
expressed in terms of $r_1$, 
in Corollary~\ref{cor:transport_cylinder+horseshoe_lip} 
as in Lemma~\ref{lem:transport_cylinder_lip}. 
This follows as given any cylinder $C(a,b;\rho)$ contained 
in the ball $B(p,r_3)$, the corresponding 
elongated neighbourhood $E(a,b;3\rho)$, 
as used in Proposition~\ref{prop:horseshoes_exist},
will be contained in the ball $B(p,r_1)$. 
\end{remark}

\end{appendix}


\begin{thebibliography}{999}

\bibliographystyle{plain}

\bibitem{ArendtKreuter}
\newblock W.~Arendt and M.~Kreuter.
\newblock \emph{Mapping theorems for Sobolev spaces of vector-valued functions\/}.
\newblock \texttt{arXiv:1611.06161v1 [math FA]}

\bibitem{AKM} 
\newblock R.L.~Adler, A.G.~Konheim and M.H.~MacAndrew.
\newblock \emph{Topological entropy\/}.
\newblock Trans. Amer. Math. Soc. \textbf{114} (1965), no.2, 309--319.

%\bibitem{AHK} 
%\newblock E.~Akin, M.~Hurley and J.~Kennedy.
%\newblock \emph{Dynamics of topologically generic homeomorphisms.}
%\newblock Mem. Amer. Math. Soc. \textbf{164} (2003), no. 783.

\bibitem{AGW}
\newblock E.~Akin, E.~Glasner and B.~Weiss.
\newblock \emph{Generically there is but one self homeomorphism of the {C}antor set\/}.
\newblock Trans. Amer. Math. Soc., \textbf{360} (2008), no. 7, 3613--3630.

\bibitem{AS} 
\newblock G.~Alessandrini and M.~Sigalotti.
\newblock \emph{Geometric properties of solutions to the anisotropic {$p$}-Laplace equation in dimension two\/}. 
\newblock Ann. Acad. Sci. Fenn. Math., \textbf{26} (2001), no. 1, 249--266.

\bibitem{AIM}
\newblock K.~Astala, T.~Iwaniec and G.~Martin.
\newblock \emph{Elliptic Partial Differential Equations and Quasiconformal Mappings in the Plane\/}.
\newblock Princeton Mathematical Series {\bf{48}}, Princeton University Press, 2009.

\bibitem{BM} 
\newblock G.~Boros and V.~Moll.
\newblock \emph{Irresistible Integrals -- Symbolics, Analysis and Experiments in the Evaluation of Integrals\/}.
\newblock Cambridge University Press, 2004. 

\bibitem{Bowen71} 
\newblock R.~Bowen.
\newblock \emph{Entropy for group endomorphisms and homogeneous spaces\/}.
\newblock Trans. Amer. Math. Soc., \textbf{153} (1971), 401--414. 

\bibitem{Calderon1951}
\newblock A.P.~Calder{\'o}n.
\newblock \emph{On the differentiability of absolutely continuous functions\/}.
\newblock Riv. Mat. Univ. Parma, \textbf{2} (1951), 203--213.

\bibitem{Cesari1941}
\newblock L.~Cesari.
\newblock \emph{Sulle funzioni assolutament continue in due variability\/}.
\newblock Ann. Scuola Norm. Sup. Pisa, \textbf{10} (1941), 91--101.

%\bibitem{DaiZhouGeng} 
%\newblock X.~Dai, Z.~Zhou, and X.~Geng.
%\newblock \emph{Some Relations between Hausdorff-dimensions and Entropies\/}.
%\newblock J. Sci. China, Ser. A
%\textbf{41}, (1998), 1068--1075. %%%  (10)

\bibitem{DHS}
\newblock  L.~D’Onofrio, S.~Hencl, R.~Schiattarella.
\newblock \emph{Bi-Sobolev homeomorphism with zero Jacobian almost everywhere\/}. 
Calculus of Variations and Partial Differential Equations, {\bf{51}} (2014), 139--170. 

\bibitem{DurenBook}
\newblock P.~Duren.
\newblock \emph{Harmonic Mappings in the Plane\/}.
\newblock Cambridge Tracts in Mathematics \textbf{156}, Cambridge University Press, 2004.

\bibitem{Edwards1984}
\newblock R.D.~Edwards.
\newblock \emph{The solution of the $4$-dimensional {A}nnulus conjecture (after {F}rank {Q}uinn)\/}.
\newblock Contemp. Math., \textbf{32} (1984), 211--264.
 
\bibitem{GehringLehto}
\newblock F.W.~Gehring and O.~Lehto.
\newblock \emph{On the total differentiability of functions of a complex variable\/}.
\newblock Ann. Acad. Sci. Fenn. Ser. {A} {I} Math., \textbf{272} (1959), 1--9.

\bibitem{Fo} 
\newblock G.~Folland. 
\newblock \emph{Real Analysis\/}.
\newblock John Wiley \&\ Sons, 1999.

\bibitem{GW}
\newblock E.~Glasner and B.~Weiss.
\newblock \emph{The topological {R}ohlin property and topological entropy\/}.
\newblock Amer. J. Math., \textbf{128} (2001), 1055--1070. 

\bibitem{GR90} 
\newblock V.~Gol'dshtein and Yu.~Reshetnyak.
\newblock \emph{Quasiconformal Mappings and Sobolev Spaces\/}.
\newblock Kluwer Academic Publishers, 1990. 

\bibitem{Hajlasz}
\newblock P.~Haj\l{}asz.
\newblock \emph{Sobolev spaces on metric-measure spaces\/}.
\newblock Contemp. Math., \textbf{338} (2003), 173--218.

\bibitem{HenclKoskelaBook}
\newblock S.~Hencl and P.~Koskela.
\newblock \emph{Lectures on Mappings of Finite Distortion\/}.
\newblock Lecture Notes in Mathematics 2096, Springer-Verlag, 2014.

\bibitem{HirschBook}
\newblock M.W.~Hirsch. %Morris W. Hirsch
\newblock \emph{{D}ifferential {T}opology\/}.
\newblock Graduate Texts in Mathematics \textbf{33}, Springer-Verlag, 1976.

\bibitem{Ito70} 
\newblock S.~Ito. 
\newblock \emph{An estimate from above for the entropy and topological entropy of a $C^1$-diffeomorphism}. 
\newblock Proceedings of the Japanese Academy of Sciences, \textbf{46} (3), (1970), 226--230.

\bibitem{IKO} 
\newblock T.~Iwaniec, L.~Kovalev and J.~Onninen.
\newblock \emph{Diffeomorphic approximation of Sobolev homeomorphisms\/}. 
\newblock Arch. Rational Mech. Anal., 201 (2011), 1047--1067.

\bibitem{KandH}
\newblock A.~Katok and B.~Hasselblatt.
\newblock \emph{Introduction to the Modern Theory of Dynamical Systems\/}.
\newblock Encyclopedia of Mathematics and Its Applications \textbf{54}, Cambridge University Press, 1995.

\bibitem{MRSY}
\newblock O.~Martio, V.~Ryazanov, U.~Srebro and E.~Yakubov.
\newblock \emph{Moduli in Modern Mapping Theory\/}.
\newblock Springer Monographs in Mathematics, Springer-Verlag, 2009. 

\bibitem{Maz'yaBook}
\newblock V.G.~Maz'ya.
\newblock \emph{Sobolev Spaces: with Applications to Elliptic Partial Differential Equations (2nd, revised and augmented Edition)\/}.
\newblock Grundlehren der Mathematischen Wissenschaften 342, Springer-Verlag, 1985.

%\bibitem{Misiurewicz2004} 
%\newblock M.~Misiurewicz. 
%\newblock \emph{On Bowen's definition of topological entropy}, 
%\newblock Discrete and Continuous Dynamical Systems,
%\textbf{10}, (2004), 827--833. %% Issue #=  (3)

\bibitem{Pugh67a}
\newblock C.C.~Pugh.
\newblock \emph{The Closing Lemma\/}.
\newblock Amer. J. Math., \textbf{87} (4), (1967), 956--1009.

\bibitem{Pugh67b}
\newblock C.~Pugh. 
\newblock \emph{An improved Closing Lemma and a general density theorem\/},
\newblock Amer. J. Math., \textbf{89} (4), (1967), 1010--1021.

\bibitem{PPSS74}
\newblock J.~Palis, C.~Pugh, M.~Shub and D.~Sullivan.
\newblock \emph{Genericity theorems in topological dynamics\/}.
\newblock Dynamical Systems--Warwick 1974, Lecture Notes in Mathematics 468,
241--250, Springer-Verlag, 1974.

\bibitem{Re} 
\newblock Yu.~Reshetnyak. 
\newblock \emph{Some geometrical properties of functions and mappings with generalized derivatives\/}. 
\newblock Sibirsk. Math. Zh. {\textbf{7}} (1966), 886--919.

\bibitem{Ru} W.~Rudin. 
\newblock \emph{Real and Complex Analysis, Second Edition\/}.
\newblock McGraw-Hill, 1974. 

\bibitem{SchoenUhlenbeck}
\newblock R.~Schoen and K.~Uhlenbeck.
\newblock \emph{Approximation theorems for Sobolev mappings\/}.
\newblock Preprint, 1984.

\bibitem{Sullivan1979}
\newblock D.~Sullivan.
\newblock \emph{Hyperbolic geometry and homeomorphisms\/}.
\newblock in \emph{Geometric Topology} (James Cantrell Ed.), Academic Press, (1979), 543--555.

\bibitem{TukiaVaisala1982a}
\newblock P.~Tukia and J.~V\"{a}is\"{a}l\"{a}.
\newblock \emph{Lipschitz and quasiconformal approximation and extension\/}.
\newblock Ann. Acad. Sci. Fenn. Ser. A I Math., \textbf{6}, (1982), no. 2, 303--342.

\bibitem{TukiaVaisala1982b}
\newblock P.~Tukia and J.~V\"{a}is\"{a}l\"{a}.
\newblock \emph{Quasiconformal extension from dimension $n$ to $n+1$\/}.
\newblock Ann. Math., \textbf{115}, no. 2, (1982), 331--348.

\bibitem{Vaisala1965}
\newblock J.~V\"{a}is\"{a}l\"{a}.
\newblock \emph{Two new characterizations of quasiconformality\/}.
\newblock Ann. Acad. Sci Fenn. Ser. {A} {I} Math., \textbf{362} (1965), 1--12.

\bibitem{WaltersBook}
\newblock P.~Walters.
\newblock \emph{An Introduction to Ergodic Theory\/}.
\newblock Graduate Texts in Mathematics \textbf{79}, Springer-Verlag, 1982. 

\bibitem{Yano80} 
\newblock K.~Yano. 
\newblock \emph{A remark on the topological entropy of homeomorphisms\/}. 
\newblock Invent. Math., \textbf{59} (1980), 215--220.

\bibitem{Zi} 
\newblock W.~Ziemer. 
\newblock \emph{Weakly Differentiable Functions\/}.
\newblock Graduate Texts in Mathematics \textbf{120}, Springer-Verlag, 1989. 

%\bibitem{ZygmundBook}
%\newblock A.~Zygmund.
%\newblock \emph{Trigonometric Series, 3rd. Ed., Vols. I \& II combined (with a foreword by Robert Fefferman)\/}.
%\newblock Cambridge Mathematical Library, Cambridge University Press, 2002.

\end{thebibliography}
\end{document}